%% file: ms.tex
\documentclass[a4paper]{amsart}

\usepackage[backrefs,lite]{amsrefs}

\usepackage{amssymb}              
\usepackage{lmodern,bm}              
\usepackage{stmaryrd}
\usepackage[all]{xy}
\usepackage{tikz}
\usetikzlibrary{calc} 

\CompileMatrices

\newtheorem{theorem}{Theorem}[section]

\newtheorem{lemma}[theorem]{Lemma}
\newtheorem{proposition}[theorem]{Proposition}
\newtheorem{corollary}[theorem]{Corollary}

\theoremstyle{definition}

\newtheorem{definition}[theorem]{Definition}
\newtheorem{construction}[theorem]{Construction}

\newtheorem{example}[theorem]{Example}
\newtheorem{remark}[theorem]{Remark}
\newtheorem{reminder}[theorem]{Reminder}

\numberwithin{equation}{theorem}

\def\vector2#1#2{\left(\begin{array}{c} #1 \\ #2 \end{array}\right)}
\def\Cl{{\rm Cl}}

\def\KK{{\mathbb K}}
\def\TT{{\mathbb T}}
\def\ZZ{{\mathbb Z}}

\def\NN{{\mathbb N}}
\def\QQ{{\mathbb Q}}
\def\PP{{\mathbb P}}

\def\Mat{{\rm Mat}}

\def\div{{\rm div}}
\def\quot{/\!\!/}

\def\bangle#1{{\langle #1 \rangle}}
\def\rq#1{\widehat{#1}}

\def\Aut{{\rm Aut}}

\def\GL{{\rm GL}}

\def\GL{{\rm GL}}

\def\Spec{{\rm Spec}}

\def\cone{{\rm cone}}

\def\id{{\rm id}}

\def\reg{{\rm reg}}

\def\diag{\mathrm{diag}}

\title[The automorphism group of a rational projective $\KK^*$-surface]{The automorphism group of \\ a rational projective $\KK^*$-surface}

\author[J\"urgen Hausen, Timo Hummel]{J\"urgen Hausen, Timo Hummel}

\address{Mathematisches Institut, Universit\"at T\"ubingen,
Auf der Morgenstelle 10, 72076 T\"ubingen, Germany}
\email{juergen.hausen@uni-tuebingen.de}

\address{Mathematisches Institut, Universit\"at T\"ubingen,
Auf der Morgenstelle 10, 72076 T\"ubingen, Germany}
\email{hummel@math.uni-tuebingen.de}


\begin{document}

\begin{abstract}
We consider possibly singular
rational projective $\KK^*$-surfaces and
provide an explicit description of the
unit component of the automorphism group
in terms of isotropy group orders and
intersection numbers of suitable invariant
curves.
As an application, we characterize the
almost homogeneous rational projective
$\KK^*$-surfaces and we specify the
two-dimensional groups acting almost
transitively as well as the corresponding
homogeneous spaces.
\end{abstract}

\maketitle

\input{intro}

\input{tvar-background.tex}

\input{demazure-roots.tex}

\input{toric-aut.tex}

\input{ambient-morph.tex}

\input{kstar-surfaces.tex}

\input{cont-frac.tex}

\input{qsm-str-ell-fps.tex}

\input{hor-ver-roots.tex}

\input{gen-root-groups.tex}

\input{equiv-resolution.tex}

\input{group-structure.tex}

\input{almhom.tex}

\input{references.tex}
\end{document}

%% file: intro.tex
\section{Introduction}

We work over an algebraically closed field $\KK$
of characteristic zero. 
By a \emph{$\KK^*$-surface} we mean a normal irreducible 
surface $X$ endowed with an effective morphical
action $\KK^* \times X \to X$ of the multiplicative 
group $\KK^*$.
The geometry of $\KK^*$-surfaces has been intensely 
studied by many 
authors; see for 
instance~\cite{FiKp1,FiKp2,FlZa,OrWa1,OrWa2,OrWa3,Pi}.
We consider the automorphism group $\Aut(X)$ 
of a rational projective $\KK^*$-surface $X$.
This is an affine algebraic group and our aim 
is to give a detailed explicit description of
the unit component of $\Aut(X)$ in
terms of basic data of the action. 

In order to formulate our main result,
let us recall the basic geometric features 
of $\KK^*$-surfaces. 
One calls a fixed point  
\emph{elliptic (hyperbolic, parabolic)} 
if it lies in the closure of 
infinitely many (precisely two, precisely one)
non-trivial $\KK^*$-orbit(s).
Elliptic and hyperbolic fixed points 
are isolated, whereas the parabolic 
fixed points form a closed smooth curve with 
at most two connected components.
Every projective normal 
$\KK^*$-surface $X$ has a \emph{source}
and a \emph{sink}, that means irreducible 
components $F^+, F^- \subseteq X$ of the fixed 
point set admitting open  $\KK^*$-invariant 
neighborhoods $U^+, U^- \subseteq X$ such that 
$$ 
\lim_{t \to 0} t \cdot x \ \in \ F^+
\text{ for all } 
x \in U^+,
\qquad\qquad
\lim_{t \to \infty} t \cdot x \ \in \ F^-
\text{ for all } 
x \in U^-,
$$
where these limits are the respective values 
at the points $0$ and $\infty$ of the unique 
morphism $\PP_1 \to X$ extending the orbit map 
$t \mapsto t \cdot x$.
The source, and as well the sink, consists 
either of a single elliptic fixed point 
or it is a smooth irreducible 
curve of parabolic fixed points;
we write $x^+$ and $x^-$ in the elliptic case 
and $D^+$ and $D^-$ in the parabolic case. 
Apart from the source and the sink, we 
find at most hyperbolic fixed points.
The raw geometric picture of a rational
projective $\KK^*$-surface $X$ is as 
follows:
\begin{center}
\begin{tikzpicture}[scale=0.6]
\sffamily
\coordinate(source) at (0,2);
\coordinate(sink) at (0,-2);
\coordinate(x0a) at (-2,1);
\coordinate(x0aa) at (-1.75,1.9);
\coordinate(x0z) at (-2,-1);
\coordinate(x0zz) at (-1.75,-2);
\coordinate(xra) at (2,1);
\coordinate(xraa) at (1.75,1.9);
\coordinate(xrz) at (2,-1);
\coordinate(xrzz) at (1.75,-2);
\coordinate(A0) at (-2.6,0);
\coordinate(Ar) at (2.6,0);

\path[fill, color=black] (source) circle (.6ex) 
node[above, font=\scriptsize]{$F^+$};
\path[fill, color=black] (sink) circle (.6ex) 
node[below, font=\scriptsize]{$F^-$};
\path[fill, color=black] (x0a) circle (.5ex) node[left]{};
\path[fill, color=black] (x0z) circle (.5ex) node[left]{};
\path[fill, color=black] (xra) circle (.5ex) node[right]{};
\path[fill, color=black] (xrz) circle (.5ex) node[right]{};
\draw[thick, bend right=30] (source) to (x0a) node[]{};
\node[align=center, font=\scriptsize] (D01) at (x0aa)  {$D_{01}$};
\draw[thick, bend right=30] (x0z) to (sink) node[]{};
\node[align=center, font=\scriptsize] (D0n0) at (x0zz)  {$D_{0n_0}$};
\draw[thick, bend left=30] (source) to (xra) node[]{};
\node[align=center, font=\scriptsize] (Dr1) at (xraa)  {$D_{r1}$};
\draw[thick, bend left=30] (xrz) to (sink) node[]{};
\node[align=center, font=\scriptsize] (Drnr) at (xrzz)  {$D_{rn_r}$};
\draw[thick, dashed, bend right=15] (x0a) to (x0z) node[]{};
\draw[thick, dashed, bend left=15] (xra) to (xrz) node[]{};
\node[align=center, font=\scriptsize] (Anull) at (A0)  {$\mathcal{A}_0$};
\node[align=center, font=\scriptsize] (Aerr) at (Ar)  {$\mathcal{A}_r$};

\end{tikzpicture}
\end{center}
The general orbit $\KK^* \cdot x \subseteq X$
has trivial isotropy group $\KK^*_x$
and its closure connects the source and the sink
in the sense that it contains one fixed point
from $F^+$ and one from  $F^-$.
Besides the general orbits, we have the
special non-trivial orbits.
Their closures are rational curves
$D_{ij} \subseteq X$ forming the \emph{arms}
$\mathcal{A}_i = D_{i1} \cup \ldots \cup D_{in_i}$ of $X$,
where $i = 0 , \ldots, r$,
the intersections $F^+ \cap D_{i1}$
and $D_{in_i} \cap F^-$ consist
each of a fixed point and any two subsequent 
$D_{ij}$, $D_{ij+1}$ intersect in a hyperbolic 
fixed point.
To every such rational curve $D_{ij}$ 
we associate an integer, namely the 
order~$l_{ij}$ of the $\KK^*$-isotropy group  
of the general point of $D_{ij}$.

Every $\KK^*$-surface $X$ admits a minimal 
equivariant resolution $\pi \colon \tilde X \to X$ of 
singularities.
If there is a parabolic fixed point curve
$D^+ \subseteq X$,
then we consider the points $x_i \in X$ 
lying in~$D^+ \cap D_{i1}$.
If $x_i$ is singular, then the fibre
$\pi^{-1}(x_{i})$ of the minimal 
resolution is  a connected part 
$E_{i1} \cup \ldots \cup E_{iq_i}$ of an
arm of $\tilde X$, where the curve
$E_{i1}$ intersects the proper transform
of $D^+$.
We define
$$
c_i(D^+) 
\ := \
\left(
\vcenter{\hbox{
$
- E_{i 1}^2 - \cfrac{1}{-E_{i2}^2-\cfrac{1}{\ldots-E_{iq_i}^2}}
$
}
}
\right)^{-1}
$$
if $x_i$ is singular and $c_i(D^+) = 0$ else.
We call an elliptic fixed point $x \in X$ 
\emph{simple}, if~$\pi^{-1}(x)$
is contained in an arm of $\tilde X$.
If $X$ admits a simple elliptic fixed point,
then we may assume that this is $x^-$.
The fibre $\pi^{-1}(x^-)=E_1\cup \ldots \cup E_q$
is a connected part of an arm of $\tilde X$ 
and~$E_q$ contains a smooth elliptic 
fixed point of $\tilde X$.
In this situation, we define 
$$
c(x^-) 
\ := \ 
\left(
\vcenter{
\hbox{
$
- E_q^2 - \cfrac{1}{-E_{q-1}^2-\cfrac{1}{\ldots -E_{1}^2}}
$
}
}
\right)^{-1}
$$
if $x^-$ is singular and $c(x^-):= 0$ else.
Finally, a point $x \in X$ is called 
\emph{quasismooth} if it is a toric surface
singularity;
see Definition~\ref{def:qsmooth}
and Corollary~\ref{cor:quasismoothchar}
for more background.
In case of a quasismooth simple elliptic
fixed point $x^- \in X$, we can always
assume the numeration of the
arms $\mathcal{A}_0, \ldots, \mathcal{A}_r$
to be \emph{normalized} in the sense that
$l_{0n_0}\ge \ldots \ge l_{rn_r}$ and 
$l_{in_i} = l_{jn_j}$ implies $D_{in_i}^2 \le D_{jn_j}^2$
whenever $i < j $ and $n_i,n_j \ge 2$.
In this situation, the exceptional curves
$E_1, \ldots,  E_q \subseteq \tilde X$ belong to the
arm $\tilde{\mathcal{A}}_0 \subseteq \tilde X$
mapping onto the arm $\mathcal{A}_0 \subseteq X$;
see Proposition~\ref{prop:roots-i0i1-switched}.
We denote by $\tilde l_{0 \tilde n_0}$ the order
of the isotropy group of the general point
of $E_q = \tilde D_{0 \tilde n_0} \subseteq \tilde{\mathcal{A}}_0$.
We are ready to state the main result of 
this article.

\begin{theorem}
\label{thm:main}
Let $X$ be a non-toric rational projective $\KK^*$-surface.
Then the unit component of the automorphism 
group $\Aut(X)$ of $X$ is given as a 
semidirect product
$$
\mathrm{Aut}(X)^0
\ = \
(\KK^\rho \rtimes_\varphi \KK^{\zeta}) \rtimes_\psi \KK^*,
\qquad
\rho \in \ZZ_{\ge 0},
\
\zeta \in \{0,1\}.
$$
If $X$ has neither a 
non-negative fixed point curve 
nor a quasismooth simple elliptic
fixed point, then $\rho = \zeta = 0$ holds.
Otherwise, precisely one of the following 
holds.
\begin{enumerate}
\item
There is a non-negative fixed point curve. 
Then we can assume that this curve is~$D^+ \subseteq X$.
In this situation, we have $\zeta = 0$ and 
$$
\rho 
\ = \ 
\max
\bigl(
0, \
(D^+)^2 +1 - \sum_{i=0}^r c_i(D^+)
\bigr).
$$
The group homomorphism 
$\psi \colon \KK^* \to \Aut(\KK^{\rho})$
fixing the semidirect product structure
is given by $t \mapsto t^{-1}E_\varrho$.
\item
There is exactly one quasismooth simple
elliptic fixed point.
We can assume that this is~$x^-$ and 
the numeration of the arms is normalized.
Then
$$
\qquad\qquad
\rho 
\ = \ 
\mathrm{max}
\biggl(
0, \
\left\lfloor
l_{1n_1}^{-1}
\min_{i \ne 0}
(
l_{in_i}D_{i n_i}^2 + (l_{in_i}-l_{1 n_1}) D_{i n_i} D_{1 n_1}
) 
- c(x^-)
\right\rfloor+1
\biggr)
$$
holds. Moreover, we have $\zeta = 1$ if and only
if for all $i \ne 1$ the following inequalities
are satisfied
$$
\qquad\qquad
l_{in_i} D_{in_i}^2 
\ \ge \ 
(l_{0 n_0} - l_{in_i}) D_{i n_i}  D_{0n_0}.
$$
The semidirect product structure on $\Aut(X)$ 
is determined by the following.
For $\zeta = 1$, the homomorphism
$\varphi \colon \KK \to \Aut(\KK^\rho)$ 
is given by 
$$
s \mapsto A = (a_{\mu \alpha}),
\text{ where } 
a_{\mu \alpha} 
\ = \ 
\begin{cases}
{\alpha-1 \choose \mu-1} s^{\alpha-\mu},
& 
\alpha \ge \mu, 
\\
0 
& 
\alpha < \mu.
\end{cases}
$$
Moreover, the group homomorphism 
$\psi \colon \KK^* \to \Aut(\KK^{\rho} \rtimes_\varphi \KK^{\zeta})$
is given by
$$ 
t 
\ \mapsto \ 
\begin{cases}
\diag(
t^{\tilde l_{0 \tilde n_0}}, 
\ldots, 
t^{\tilde l_{0 \tilde n_0} - (\rho-1) l_{1 n_1}}
),
& 
\zeta =  0,
\\
\diag(
t^{l_{0 n_0}}, 
\ldots, 
t^{l_{0 n_0} - (\rho-1) l_{1 n_1}},
t^{l_{1 n_1}}
),
& 
\zeta =  1.
\end{cases}
$$
\end{enumerate}
\end{theorem}

Automorphism groups of rational surfaces
have also been considered by several other
authors.
For instance, Sakamaki~\cite{Sa} studied
the case of cubic surfaces without
parameters.
More generally, Cheltsov and Prokhorov~\cite{ChPr}
and, independently, also Martin and
Stadlmayr~\cite{MaSt}
determined the Gorenstein log
del Pezzo surfaces with infinite
automorphism groups.
It turns out that 50 out of the 53 listed
surfaces of~\cite{ChPr,MaSt} are
in fact $\KK^*$-surfaces and the
descriptions of the automorphism
groups obtained there are in accordance
with Theorem~\ref{thm:main}.
Note that Theorem~\ref{thm:main} 
does not make any assumptions on
the singularities or the canonical divisor
and thus goes far beyond the Gorenstein
log del Pezzo case.
Of course, one may ask why Theorem~\ref{thm:main}
excludes the toric surfaces.
For the sake of completeness, we treat
them in Proposition~\ref{prop:tor-surf-aut}.

Let us take a closer look at the action 
of the automorphism group. 
We discuss the question when a non-toric rational
$\KK^*$-surface $X$ admits an \emph{almost transitive
action}, that means a morphical action of
an algebraic group $G$ having an open orbit.
In this case, $X$ is an equivariant
compactification of a $G$-homogeneous space
and it is natural to ask when it is even
an equivariant compactification of an algebraic
group.
Theorem~\ref{thm:main} allows us to give
answers in terms of isotropy group orders and
intersection numbers.

\begin{theorem}
\label{thm:maincor}
Consider a non-toric rational projective
$\KK^*$-surface~$X$.
Then the following statements are equivalent.
\begin{enumerate}
\item
The surface $X$ admits an almost transitive
action of a two-dimensional
algebraic group $G$.
\item
The surface $X$ is almost homogeneous in the
sense that the action of the automorphism group
$\Aut(X)$ on $X$ is almost transitive.
\item
There is a quasismooth simple elliptic
fixed point, say $x^- \in X$, and,
assuming the numeration of the arms
to be normalized, 
one of the following two series of
inequalities is valid:
$$
\qquad
l_{in_i}D_{i n_i}^2 + (l_{in_i}-l_{1 n_1}) D_{i n_i} D_{1 n_1}
\ \ge \
l_{1n_1}c(x^-),
\quad
i = 1, \ldots, r,
$$
$$
\qquad
l_{in_i} D_{in_i}^2 \ \ge \ (l_{0 n_0} - l_{in_i}) D_{i n_i}  D_{0n_0},
\quad i = 0, 2, \ldots, r.
$$
\end{enumerate}
Assume that one of the above statements holds
and let $G$ be a two-dimensional algebraic group
acting effectively and almost transitively on $X$.
\begin{enumerate}
\item[(iv)]
If only one of the series of inequalities of~(iii) is valid,
then $G$ is a non-abelian semidirect product
$\KK \rtimes_\varphi \KK^*$. Moreover:
\begin{enumerate}
\item
If the first series of inequalities holds, 
then $\varphi(t)(s) = t^{l_1n_1}s$ 
and for general $x \in X$, the isotropy
group $G_x$ is cyclic of order $l_{0n_0}$.
Thus,~$X$ is an equivariant
$G$-compactification if and only if
$l_{0n_0} = 1$.
\item
If the second series of inequalities holds, 
then $\varphi(t)(s) = t^{l_0n_0}s$ 
and for general $x \in X$, the isotropy
group $G_x$ is cyclic of order $l_{1n_1}$.  
Thus,~$X$ is an equivariant
$G$-compactification if and only if
$l_{1n_1} = 1$.
\end{enumerate}
\item[(v)]
If the series of inequalities of~(iii) both
are valid,
then the groups $G$ from~(iv)~(a) and~(iv)~(b)
both act, and, moreover, $X$ is an equivariant 
compactification of the vector group $G = \KK^2$.
\end{enumerate}
\end{theorem}

The case of almost homogeneous Gorenstein
del Pezzo surfaces 
has been investigated in~\cite{DeLo,DeLo0}.
Theorem~\ref{thm:maincor} is in accordance 
with the results obtained there and it delivers in
addition the general isotropy groups. 
Note that there exist normal surfaces $X$
which equivariantly compactify the abelian
group $\KK \times \KK^*$.
But any such $X$ is a toric surface according
to~\cite[Theorem~2]{ArKo}.
More explicit statements on the almost homogenoeus
case are given in Section~\ref{sec:almhom}.
For instance Proposition~\ref{prop:all-k-k*-groups}
specifies up to conjugation all the semidirect
products $\KK \times_\psi \KK^* \subseteq \Aut(X)$
acting almost transitively.
Moreover, the almost transitive $\KK^2$-actions on $X$
from Theorem~\ref{thm:maincor}~(v)
are so-called \emph{additive actions}
on $X$ in the sense of~\cite{ArRo,Dz}.
In Proposition~\ref{prop:alladdact}, we
determine up to conjugation by elements
from $\KK^*$ all the additive actions
on $X$.

Let us give an outline of the article,
showing its basic ingredients and some
main ideas.
Our working environment is the
\emph{Cox ring based approach}
of~\cite{HaSu,HaHe} to
rational $T$-varieties~$X$ of
complexity one, that means that $X$
is normal, rational and comes with
an effective torus action
$T \times X \to X$ such that the
general $T$-orbit is of codimension
one in $X$.
One of the basic features of this
approach is that it provides a
natural $T$-equivariant closed
embedding $X \subseteq Z$ into
a toric variety~$Z$.
In Section~\ref{sec:tvarlowcpl},
we present a brief general reminder.

We will also make use of the description of
the automorphism group of
a toric variety via the
\emph{Demazure roots} of its defining
fan; see~\cite{Cox,Dem} and
Section~\ref{sec:demroots} for a
quick summary.
First applications are the tools
provided in
Section~\ref{sec:toric-surf-aut}
and the explicit description of the
automorphism goups of toric surfaces
given there.
The understanding of $\Aut(X)$ for
a complete rational
$T$-variety $X$ of complexity one is
not yet as developed as in the toric case.
However, the main results of~\cite{ArHaHeLi}
show that $\Aut(X)^0$ is generated by~$T$
and the additive one-parameter groups,
also called root groups,
arising from \emph{Demazure $P$-roots};
see also Section~\ref{sec:demroots}.
In Theorem~\ref{thm:restrrootauts} we
provide a presentation of the automorphisms
arising from Demazure $P$-roots as 
restrictions of automorphisms of the
ambient toric variety~$Z$ which
are explicitly given in Cox coordinates.
The explicit nature of the result is
crucial for our purposes.
The general question to which extent
a variety
inherits its automorphisms from a suitable
ambient variety is interesting as well.
For Mori dream spaces $X$, a positive
result concerning $\Aut(X)^0$ is given
in~\cite[Thm.~4.4]{HaKeWo};
see also~\cite{PrSh} for further
results in the case of quasismooth Fano
weighted complete intersections.

From Section~\ref{sec:ratprojkstarsurf}
on, we focus on rational projective
$\KK^*$-surfaces.
We first recall basics on their geometry
and relate defining data to self
intersection numbers, see
Sections~\ref{sec:ratprojkstarsurf}
and~\ref{sec:cfrac}.
In Theorem~\ref{thm:qs-str-ell-fp}
we figure out geometric implications of the
existence of a quasismooth simple
elliptic fixed point:
a non-toric rational projective
$\KK^*$-surface~$X$ can have at most one
such fixed point and if there is one,
then any parabolic fixed point curve
is contractible or its intersection
with any arm of $X$ is a singularity
of $X$.
In Section~\ref{sec:horverProots},
we introduce horizontal and vertical
\emph{$P$-roots}, which basically
means adapting the more involved
notion of a Demazure $P$-root to the
surface case.
Together with $\KK^*$, the root groups
arising from
the $P$-roots generate $\Aut(X)^0$.
We link existence of $P$-roots to
the geometry of $X$.
From~\cite{ArHaHeLi} we infer that
$\Aut(X)$ acts with an open orbit
if and only if there is a horizontal
$P$-root.
Proposition~\ref{prop:hor-P-root-constraints}
shows that the presence of a horizontal
$P$-root forces a quasismooth simple
elliptic fixed point.
By Proposition~\ref{prop:qsefp2novroots},
existence of vertical $P$-roots exclude
quasismooth simple elliptic fixed points.
Consequently, $\Aut(X)$ does not
act almost transitively if we have
vertical $P$-roots.
Each vertical root is uniquely
associated with a parabolic fixed point
curve in the sense that the corresponding
root group moves that curve.
Proposition~\ref{prop:vroot+2novroot-}
shows that if there are vertical roots,
then they are all associated with the
same fixed point curve.

Starting with Section~\ref{sec:generatingroots},
we study the structure of the unit component
of $\Aut(X)$.
The first task is to figure out relations
among the root groups
arising from the $P$-roots.
A sufficiently detailed study allows us
to figure out minimal generating systems
of $P$-roots.
Proposition~\ref{prop:gen-hor-P-roots}
does this for the case that $\Aut(X)$
acts with an open orbit and
Proposition~\ref{prop:gen-ver-P-roots}
settles the remaining case.
In Section~\ref{sec:equiv-resolution}, we
show in terms of the combinatorics
of defining data that for the minimal
resolution of singularities $\tilde X \to X$
of a rational projective
$\KK^*$-surface, the groups $\Aut(\tilde X)^0$
and $\Aut(X)^0$ coincide. Whereas the
latter can as well be deduced from the general
existence of a functorial resolution in
characteristic zero, our investigation is
more specific and allows us to relate the
root groups of $X$ with those of $\tilde X$
in an explicit manner.
In Section~\ref{sec:group-structure},
we prove Theorem~\ref{thm:main}.
The basic idea is to gain the
desired information on $\Aut(X)^0$ and
its action on $X$ via morphisms
$X \leftarrow \tilde X \to X'$,
where $\tilde X \to X$ is the minimal
resolution and $\tilde X \to X'$ a suitable
birational contraction to a certain toric surface
that allows to keep track on the relevant
root groups.
Finally, in Section~\ref{sec:almhom} we
study almost transitive actions on~$X$,
specify the acting two-dimensional
groups and prove Theorem~\ref{thm:maincor}.

\tableofcontents

The authors would like to thank Ivan Arzhantsev
for valuable comments and suggestions.

%% file: tvar-background.tex
\section{$T$-varieties of low complexity}
\label{sec:tvarlowcpl}

Here we provide the necessary background
on toric varieties and rational varieties
with torus action of complexity one.
Throughout the whole article, the ground field~$\KK$
is algebraically closed and of characteristic 
zero.
We simply write $\KK$ for the additive group
of the ground field, $\KK^*$ for the
multiplicative one and $\TT^n$ for the $n$-fold
direct product $(\KK^*)^n$.

By a torus we mean an algebraic group~$T$
isomorphic to some $\TT^n$.
A quasitorus is a direct product of a torus
and a finite abelian group.
By a $T$-variety~$X$ we mean a normal, irreducible
variety $X$ with an effective action of a torus~$T$
given by a morphism $T \times X \to X$.
The complexity of a $T$-variety $X$ is
the difference $\dim(X)-\dim(T)$.

We turn to toric varieties, which by
definition are the $T$-varieties of
complexity zero.
The basic feature of toric varieties
is that they are completely described
via lattice fans.
We assume the reader to be familiar with
the foundations of this theory
as explained for instance in~\cite{Dan, Ful, CoLiSc}.

We will intensely use the Cox ring
and Cox's quotient construction for
toric varieties~\cite{Cox}.
Recall that for any normal variety $X$
with only constant global invertible
functions and finitely generated
divisor class group $\Cl(X)$,
one associates a Cox sheaf
$$
\mathcal{R}
\ = \
\bigoplus_{D \in \Cl(X)} \mathcal{O}_X(D),
$$
see~\cite[Chap.~1]{ArDeHaLa} for details.
The Cox ring $\mathcal{R}(X)$ is the
$\Cl(X)$-graded algebra of global
sections of the Cox sheaf.
If the Cox ring is finitely
generated, we can establish the
following picture
$$
\xymatrix{
{\Spec_X \, \mathcal{R}}
\ar@{=}[r]
&
{\hat X}
\ar@{}[r]|\subseteq
\ar[d]_{\quot H}
&
{\bar X}
\ar@{=}[r]
&
{\Spec_X \, \mathcal{R}(X)}
\\
&
X
&
&
}
$$
where $\bar X$ is  the total
coordinate space coming with
an action of the characteristic
quasitorus
$H = \Spec \, \KK[\Cl(X)]$ 
and the characteristic
space $\hat X$ which is an
open $H$-invariant subset
of $\bar X$ and has $X$ as
a good quotient for the
induced $H$-action.
In the case of toric varieties,
this picture can be established
in terms of defining lattice
fans as follows.

\begin{construction}
\label{constr:coxtoric}
Let $Z$ be the toric variety defined by 
a fan  $\Sigma$ in a lattice $N$ such that
the primitive generators $v_1,\ldots, v_r$ 
(of the rays) of $\Sigma$ span the 
rational vector space $N_\QQ = N \otimes_\ZZ \QQ$.
We have mutually dual exact sequences
$$ 
\xymatrix@R=1em{
0
\ar[r]
&
L
\ar[r]
&
\ZZ^r
\ar[r]^P
&
N
\\
0
\ar@{<-}[r]
&
K
\ar@{<-}[r]_Q
&
\ZZ^r
\ar@{<-}[r]_{P^*}
&
M
\ar@{<-}[r]
&
0
}
$$
where $P \colon \ZZ^r \to N$ sends the $i$-th 
canonical basis vector $e_i \in \ZZ^r$
to the $i$-th primitive generator $v_i \in N$;
we also speak of the generator matrix
$P = [v_1, \ldots, v_r]$ of $\Sigma$.
The lower sequence gives rise to an exact 
sequence
$$ 
\xymatrix@R=1em{
1
\ar[r]
&
H
\ar[r]
&
\TT^r
\ar[r]^P
&
T_Z
\ar[r]
&
1
}
$$
involving the quasitorus $H = \Spec \, \KK[K]$ 
and the acting torus $T_Z = \Spec \, \KK[M]$ 
of~$Z$.
Moreover, the divisor class group 
and the Cox ring of $Z$ are given as
$$ 
\Cl(Z) \ = \ K, 
\qquad\qquad 
\mathcal{R}(Z) \ = \ \KK[T_1,\ldots,T_r],
$$
where the $\Cl(Z)$-grading of $\mathcal{R}(Z)$ 
is given by $\deg(T_i) = Q(e_i)$.
Finally, we obtain a fan $\hat \Sigma$ in $\ZZ^r$
consisting of certain faces of the positive orthant,
namely 
$$ 
\hat \Sigma 
\ := \ 
\{\delta_0 \preceq \QQ_{\ge 0}^r; \ 
P(\delta_0) \subseteq \sigma \text{ for some } \sigma \in \Sigma\}.
$$
The toric variety $\hat Z$ associated with $\hat \Sigma$
is the characteristic space of $Z$,
sitting as an open toric subset in
the total coordinate space 
$\bar Z := \KK^r$.
As $P$ is a map of the fans~$\hat \Sigma$ and 
$\Sigma$, it defines a toric morphism 
$p \colon \hat Z \to Z$, the good quotient 
for the action of the quasitorus 
$H = \ker(p) \subseteq \TT^r$ on~$\hat Z$.
\end{construction}

\begin{remark}
Construction~\ref{constr:coxtoric} allows 
to put hands on the points of a toric variety:
every $x \in Z$ can be written as $x = p(z)$,
where $z \in \hat Z$ is a point with closed $H$-orbit
in $\hat Z$.
Such a presentation is unique up to multiplication 
by elements of $H$ and we call
$z = (z_1,\ldots, z_r)$ \emph{Cox coordinates} 
for the point $x \in Z$.
\end{remark}

We will use the Cox ring based approach
to torus actions as developed for the
case of rational $T$-varieties of
complexity one in~\cite{HaSu,HaHe},
and, more recently, in widest possible
generality in~\cite{HaHiWr}.
Let us first have a look at an example,
showing some of the main ideas.

\begin{example}
\label{ex:running-example-1}
Consider the surface $X$ in the   
weighted projective space
$\PP_{2,7,1,13}$ given as the
zero set of a weighted homogeneous
trinomial equation:
$$
X
\ = \
V(T_{01}^7+T_{12}^2+T_{21}T_{22})
\ \subseteq \
\PP_{2,7,1,13},
$$
where each of the variables
appears in exactly one monomial
as indicated by the
double-indexing $T_{ij}$.
Then $X$ comes with
a $\KK^*$-action, given by
$$
t \cdot [z]
\ = \
[z_{01},z_{11},t^{-1}z_{21},tz_{22}].
$$
The ambient space $\PP_{2,7,1,13}$
is a toric variety.
Its defining fan $\Sigma$
lives in $\ZZ^3$ and its rays
are generated by the
columns $v_{01}$, $v_{11}$, $v_{21}$
and $v_{22}$ of the matrix
$$
P
\ = \ 
\left[  
\begin{array}{rrrr}
-7 & 2 & 0 & 0
\\
-7 & 0 & 1 & 1
\\
 -4 & 1 & 1 & 0 
\end{array}
\right].
$$
This setting reflects the key features
of the $\KK^*$-action on our surface~$X$.
For instance, setting $D_{ij} := X \cap V(T_{ij})$,
we obtain the arms of $X$ as
$$
\mathcal{A}_0 \ = \ D_{01},
\qquad
\mathcal{A}_1 \ = \ D_{11},
\qquad
\mathcal{A}_2 \ = \ D_{21} \cup D_{22}.
$$
Moreover, the order $l_{ij}$ of the isotropy
group of the general point in $D_{ij}$
shows up in the upper two rows of the matrix
$P$, as we have
$$
l_{01} \ = \ 7,
\qquad
l_{11} \ = \ 2,
\qquad
l_{21} \ = \ 1,
\qquad
l_{22} \ = \ 1.
$$
Finally, $X$ inherits many geometric
properties from its ambient space
$Z := \PP_{2,7,1,13}$.
Most significantly, the Cox ring of
$X$ is the factor algebra
$$
\mathcal{R}(X)
\ = \
\mathcal{R}(Z)/\bangle{g}
\ = \ 
\KK[T_{01},T_{11},T_{21},T_{22}] / \bangle{T_{01}^7+T_{12}^2+T_{21}T_{22}},
$$
where the grading of the Cox rings $\mathcal{R}(X)$
and $\mathcal{R}(Z)$ by the divisor class group
$\Cl(X) = \Cl(Z) = \ZZ$ are given by
$$
\deg(T_{01}) \ = \ 2,
\quad
\deg(T_{02}) \ = \ 7,
\quad
\deg(T_{11}) \ = \ 1,
\quad
\deg(T_{21}) \ = \ 13.
$$
\end{example}

This picture extends as follows. The
arbitrary rational projective
$\KK^*$-surface~$X$ comes embedded
into a certain toric variety, is given
by specific trinomial equations as
above and the key features of
the $\KK^*$-action as well as the
geometry of $X$ can be extracted from
the defining data.
Here comes the construction
provided in~\cites{HaHe,HaSu};
see also~\cite[Sec.~3.4]{ArDeHaLa}.

\begin{construction}
\label{constr:RAP}
Fix $r \in \ZZ_{\ge 1}$, a sequence 
$n_0, \ldots, n_r \in \ZZ_{\ge 1}$, set 
$n := n_0 + \ldots + n_r$, and fix  
integers $m \in \ZZ_{\ge 0}$ and $0 < s < n+m-r$.
The input data are matrices 
$$
A 
 =  
[a_0, \ldots, a_r]
 \in  
\Mat(2,r+1;\KK),
\qquad
P
 = 
\left[ 
\begin{array}{cc}
L & 0 
\\
d & d'  
\end{array}
\right]
 \in  
\Mat(r+s,n+m; \ZZ),
$$
where $A$ has pairwise linearly independent 
columns and $P$ is built from an
$(s \times n)$-block $d$, an $(s \times m)$-block 
$d'$ and an $(r \times n)$-block $L$ 
of the form  
$$
L
\ = \ 
\left[
\begin{array}{cccc}
-l_0 & l_1 &   \ldots & 0 
\\
\vdots & \vdots   & \ddots & \vdots
\\
-l_0 & 0 &\ldots  & l_{r} 
\end{array}
\right],
\qquad
l_i \ = \ (l_{i1}, \ldots, l_{in_i}) \ \in \ \ZZ_{\ge 1}^{n_i}
$$
such that the columns $v_{ij}$, $v_k$ of 
$P$ are pairwise different primitive vectors 
generating $\QQ^{r+s}$ as a cone. 
Consider the polynomial algebra 
$$
\KK[T_{ij},S_k]
\ := \ 
\KK[T_{ij},S_k; \; 0 \le i \le r, \, 1 \le j \le n_i, 1 \le k \le m]. 
$$
Denote by $\mathfrak{I}$ the set of 
all triples $I = (i_1,i_2,i_3)$ with 
$0 \le i_1 < i_2 < i_3 \le r$ 
and define for any $I \in \mathfrak{I}$ 
a trinomial 
$$
g_I
\ := \
g_{i_1,i_2,i_3}
\ := \
\det
\left[
\begin{array}{ccc}
T_{i_1}^{l_{i_1}} & T_{i_2}^{l_{i_2}} & T_{i_3}^{l_{i_3}}
\\
a_{i_1} & a_{i_2} & a_{i_3}
\end{array}
\right],
\qquad
T_i^{l_i} 
\  := \
T_{i1}^{l_{i1}} \cdots T_{in_i}^{l_{in_i}}.
$$
Consider the factor group 
$K := \ZZ^{n+m}/\rm{im}(P^*)$
and the projection $Q \colon \ZZ^{n+m} \to K$.
We define a $K$-grading on 
$\KK[T_{ij},S_k]$ by setting
$$ 
\deg(T_{ij}) 
\ := \ 
w_{ij}
\ := \ 
Q(e_{ij}),
\qquad
\deg(S_{k}) 
 \ := \
w_k
\ := \  
Q(e_{k}).
$$
Then the trinomials $g_I$ just introduced 
are $K$-homogeneous, all of the same degree.
In particular, we obtain a $K$-graded 
factor algebra  
$$
R(A,P)
\ := \
\KK[T_{ij},S_k] 
\ / \
\bangle{g_I; \ I \in \mathfrak{I}}.
$$
\end{construction}

The ring $R(A,P)$ just constructed is a normal 
complete intersection ring and its ideal of 
relations is, for example, generated by 
$g_{i,i+1,i+2}$, where $i = 0, \ldots, r-2$.
The varieties $X$ with torus action of 
complexity one are constructed as quotients
of $\Spec \, R(A,P)$ by the quasitorus 
$H = \Spec \, \KK[K]$. Each of them comes 
embedded into a toric variety.

\begin{construction}
\label{constr:RAPdown}
Situation in Construction~\ref{constr:RAP}.
Consider the common zero set of the 
defining relations of $R(A,P)$:
$$ 
\bar{X} 
\ := \ 
V(g_I; \ I \in \mathfrak{I}) 
\ \subseteq \ 
\bar{Z} 
\ := \ 
\KK^{n+m}
$$
Let $\Sigma$ be any fan in the lattice $N = \ZZ^{r+s}$
having the columns of $P$ as the primitive 
generators of its rays.
Construction~\ref{constr:coxtoric} leads to 
a commutative diagram
$$ 
\xymatrix@R=1.5em{
{\bar{X}}
\ar@{}[r]|\subseteq
\ar@{}[d]|{\rotatebox[origin=c]{90}{$\scriptstyle\subseteq$}}
&
{\bar{Z}}
\ar@{}[d]|{\rotatebox[origin=c]{90}{$\scriptstyle\subseteq$}}
\\
{\hat{X}}
\ar[r]
\ar[d]_{\quot H}^p
&
{\hat{Z}}
\ar[d]^{\quot H}_p
\\
X
\ar[r]
&
Z}
$$
with a variety $X = X(A,P,\Sigma)$ embedded into 
the toric variety $Z$ associated with~$\Sigma$.
Dimension, divisor class group
and Cox ring of $X$ are given~by 
$$ 
\dim(X) = s+1,
\qquad
\Cl(X) \ \cong \ K,
\qquad
\mathcal{R}(X) \ \cong \ R(A,P).
$$
The subtorus $T \subseteq \TT^{r+s}$ of the 
acting torus of $Z$ associated with the  
sublattice $\ZZ^{s} \subseteq \ZZ^{r+s}$
leaves $X$ invariant and the induced $T$-action
on $X$ is of complexity one.
\end{construction}

\begin{remark}
\label{rem:tc-space}
In Construction~\ref{constr:RAPdown},
the group $H \cong \Spec \, \KK[\Cl(X)]$ 
is the characteristic quasitorus 
and $\bar X \cong \Spec \, \mathcal{R}(X)$ 
is the total coordinate space of $X$.
Moreover, $p \colon \rq{X} \to X$ is the 
characteristic space over $X$.
\end{remark}

\begin{remark}
As in the toric case, 
Construction~\ref{constr:RAPdown} yields
\emph{Cox coordinates} for the points 
of $X = X(A,P,\Sigma)$.
Every $x \in X \subseteq Z$
can be written as $x = p(z)$,
where $z \in \hat X \subseteq \hat Z$ 
is a point with closed $H$-orbit
in $\hat X$ and this 
presentation is unique up to multiplication 
by elements of $H$.
\end{remark}

\begin{remark}
We say that the matrix $P$ from Construction~\ref{constr:RAP}
is \emph{irredundant} if we have $l_{i1}n_i \ge 2$ for
$i = 0, \ldots, r$.
In Construction~\ref{constr:RAPdown}, we may assume
without loss of generality that $P$ is irredundant.
An $X(A,P,\Sigma)$ with irredundant $P$ is a toric
variety if and only if $r = 1$ holds.
\end{remark}

The results of~\cites{ArDeHaLa,HaHe,HaSu} 
tell us in particular the following;
see also~\cite{HaHiWr} for a generalization
to higher complexity.

\begin{theorem}
Every normal rational projective variety 
with a torus action of complexity one is
equivariantly isomorphic to some $X(A,P,\Sigma)$.
\end{theorem}

%% file: demazure-roots.tex
\section{Demazure roots and automorphisms}
\label{sec:demroots}

Here we present the necessary general background 
and facts on automorphisms of toric varieties  
and rational varieties with a torus action of
complexity one.

The approach to automorphisms via Demazure roots 
involves locally nilpotent derivations.
Let us briefly recall some basics from~\cite{Freu}.
A \emph{derivation} on an integral 
affine $\KK$-algebra $R$ is a $\KK$-linear map
$\delta \colon R \to R$ satisfying the Leibniz 
rule
$$
\delta(fg)
\ = \
\delta(f)g  +  f\delta(g).
$$
A derivation $\delta \colon R \to R$ is
\emph{locally nilpotent} if every
$f \in R$ admits an $n \in \NN$ with
$\delta^n(f) = 0$.
Any locally nilpotent derivation
$\delta \colon R \to R$ defines 
a representation
$$
\bar \lambda^\star_\delta \colon \KK \to \Aut(R),
\qquad
\bar \lambda^\star_\delta(s)(f)
:= 
\exp(s \delta)(f)
:= 
\sum_{k = 0}^{\infty} \frac{s^k}{k!} \delta^k(f).
$$
In fact this yields a bijection between the
locally nilpotent derivations of $R$
and the rational representations
of $\KK$ by automorphisms of $R$.
Consequently,
$$ 
\bar \lambda_\delta \colon \KK \ \to \ \Aut(\Spec \, R),
\qquad
s \ \mapsto \ \Spec(\bar \lambda^\star_\delta(s))
$$
is a group homomorphism and, by construction, each
of the automorphisms $\bar \lambda_\delta(s)$ of $\Spec(R)$ 
has
$\bar \lambda_\delta(s)^* = \bar \lambda^\star_\delta(s)$
as its comorphism.

As for any complete rational variety,
the automorphism group of a toric 
variety is an affine algebraic group.
Its structure has been studied by 
Demazure~\cite{Dem} and Cox~\cite{Cox}.
The following is a key concept.

\begin{definition}
\label{def:toricdemazureroot}
Notation as in~\ref{constr:coxtoric}.
A \emph{Demazure root} at the primitive generator 
$v_i \in N$ of $\Sigma$ is an integral linear 
form $u \in M$ satisfying the conditions
$$ 
\bangle{u,v_i} \ = \ -1,
\qquad \qquad 
\bangle{u,v_j} \ \ge \ 0
\text{ for all } j \ne i.
$$
\end{definition}

\begin{construction}
\label{constr:toriclnd}
Notation as in~\ref{constr:coxtoric}.
Let $u \in M$ be a Demazure root at 
the primitive generator $v_i \in N$ of $\Sigma$.
The associated locally nilpotent 
derivation $\delta_u$ on 
$\KK[T_1,\ldots,T_r]$ is defined by its 
values on the variables:
$$ 
\delta_u(T_j) 
\ :=  \ 
\begin{cases}
T_iT^{P^*(u)}, & j = i,
\\
0,           & j \ne i,
\end{cases}
\qquad
\text{ where} \quad
T^{P^*(u)} = T_1^{\bangle{u,v_1}} \cdots T_r^{\bangle{u,v_r}}.
$$
Observe that $\delta_u^2 = 0$ holds.
Moreover, we have $Q(P^*(u)) = 0$ and thus
$\delta_u$ preserves 
the $K$-grading of $\KK[T_1,\ldots,T_r]$.
The corresponding rational representation of $\KK$ 
on $\KK[T_1,\ldots,T_n]$ is given by  
$$
\bar \lambda_u^\star(s)(T_j)
 \ = \ 
\begin{cases}
T_i + sT_iT^{P^*(u)}, & j = i,
\\
T_j,           & j \ne i,
\end{cases}
$$
Now, if $\bar \lambda_u \colon \KK \to \Aut(\bar Z)$ 
leaves $\hat Z$ invariant, for instance, 
if $Z$ is complete, then~$\bar \lambda_u$ descends 
to a homomorphism $\lambda_u \colon \KK \to \Aut(Z)$,
the \emph{root group} associated with the 
Demazure root $u$.
In Cox coordinates, we have
$$ 
\lambda_u(s)(z) 
\ = \ 
z + s z_iz^{P^*(u)} e_i.
$$
\end{construction}

\begin{theorem}
\label{thm:toricaut}
See~\cite[Cor.~4.7]{Cox}.
Let $Z$ be a complete toric variety
arising from a fan~$\Sigma$ in a 
lattice~$N$.
Then $\Aut(Z)^0$ is generated as 
group by the acting torus $T_Z$ of $Z$ 
and the images $\lambda_u(\KK)$,
where $u$ runs through the Demazure 
roots of $\Sigma$.
\end{theorem}

The concept of Demazure roots was extended
in~\cite{ArHaHeLi} to the case of normal 
rational varieties $X$ with an effective 
torus action $T \times X \to X$ of
complexity one. 
Let us recall the basic notions and 
facts.

\begin{definition}
\label{def:Pdemroot}
See~\cite[Def.~5.2]{ArHaHeLi}.
Let $P$ be a matrix as in
Construction~\ref{constr:RAP}.
Consider the columns 
$v_{ij}, v_k \in N = \ZZ^{r+s}$
of $P$ and the dual lattice 
$M$ of $N$.
\begin{enumerate}
\item
A {\em vertical Demazure $P$-root\/} is
a tuple $(u,k_0)$ with a linear form $u \in M$
and an index $1 \le k_0 \le m$ satisfying
\begin{eqnarray*}
\bangle{u,v_{ij}}
& \ge  &
0
\qquad \text{ for all } i,j,
\\
\bangle{u,v_{k}}
& \ge  &
0
\qquad \text{ for all } k \ne k_0,
\\
\bangle{u,v_{k_0}}
& =  &
-1.
\end{eqnarray*}
\item
A {\em horizontal Demazure $P$-root\/}
is a tuple $(u,i_0,i_1,C)$, where
$u \in M$ is a linear form,
$i_0 \ne i_1$ are indices
with $0 \le i_0, i_1 \le r$,
and $C = (c_0,\ldots,c_r)$ is a sequence
with $1 \le c_i \le n_i$ such that
\begin{eqnarray*}
l_{ic_i}
& = &
1 \qquad \text{ for all } i \ne i_0,i_1,
\\
\bangle{u,v_{ic_i}}
& = &
\begin{cases}
0,        &\qquad i \ne i_0,i_1,
\\
-1,        &\qquad i=i_1,
\end{cases}
\\
\bangle{u,v_{ij}}
& \ge &
\begin{cases}
l_{ij}, &\qquad i \ne i_0,i_1, \quad j \ne c_i,
\\
0,      &\qquad i = i_0,i_1, \quad j \ne c_i,
\\
0,       &\qquad i=i_0,\hphantom{i_1,} \quad j = c_i,
\end{cases}
\\
\bangle{u,v_k}
& \ge &
0 \qquad \text{ for all } k.
\end{eqnarray*}
\end{enumerate}
\end{definition}

\begin{example}
\label{ex:running-example-2}
Consider the defining matrix $P$ of the 
$\KK^*$-surface discussed before in
Example~\ref{ex:running-example-1},
that means
$$
P
\ = \ 
\left[  
\begin{array}{rrrr}
-7 & 2 & 0 & 0
\\
-7 & 0 & 1 & 1
\\
 -4 & 1 & 1 & 0 
\end{array}
\right].
$$
As $m = 0$, there are no vertical Demazure $P$-roots,
but we have a horizontal Demazure $P$-root
$(u,i_0,i_1,C)$ given by 
$$
u \ = \ (-1,0,1),
\qquad
i_0 \ = \ 0,
\qquad
i_1 \ = \ 1,
\qquad
C \ = \ (1,1,2).
$$
\end{example}

\begin{construction}
\label{constr:DEMLND}
See~\cite[Constr.~3.4 and~5.7]{ArHaHeLi}.
Let $A$ and $P$ be as in
Construction~\ref{constr:RAP}.
Given $i_0 \ne i_1$ with $0 \le i_0,i_1 \le r$ and
$C = (c_0,\ldots,c_r)$ with $1 \le c_i \le n_i$
we define $\zeta = \zeta(i_0,i_1,C) = (\zeta_{ij},\zeta_k) \in \ZZ^{n+m}$
by
$$
\qquad\qquad
\zeta_{ij}
\ := \
\begin{cases}
l_{ij}, & i \ne i_0,i_1, \ j \ne c_i,
\\
-1,     & i = i_1,\hphantom{i_0,} \ j = c_{i_1},
\\
0       & \text{else},
\end{cases}
\qquad\qquad
\zeta_k
\ := \
0,
k = 1, \ldots , m.
$$
Moreover, to $u \in M$ and the lattice vectors
$\zeta \in \ZZ^{n+m}$ just introduced 
we assign the following monomials
$$
h^u
\ = \
\prod_{i,j} T_{ij}^{\bangle{u,v_{ij}}}
\prod_{k} S_k^{\bangle{u,v_k}},
\qquad\qquad
h^\zeta
\ := \
\prod_{i,j} T_{ij}^{\zeta_{ij}}
\prod_{k} S_k^{\zeta_k}.
$$
Every Demazure $P$-root $\kappa$ defines 
a locally nilpotent derivation $\delta_\kappa$ 
on $\KK[T_{ij},S_k]$.
If $\kappa = (u,k_0)$ is vertical,
then one sets
$$
\delta_\kappa(T_{ij}) \ := \ 0  \text{ for all } i,j,
\qquad
\delta_\kappa(S_k)
\ := \
\begin{cases}
h^uS_{k_0}, & k = k_0,
\\
0,          & k \ne k_0.
\end{cases}
$$
If $\kappa = (u,i_0,i_1,C)$ is horizontal,
then there is a unique vector $\beta = \beta(A,i_0,i_1)$
in the row space of~$A$ with
$\beta_{i_0}=0$, $\beta_{i_1}=1$
and one sets
$$
\delta_{\kappa}(T_{ij}) 
\ := \
\begin{cases}
\beta_i 
\frac{h^u}{h^\zeta}
\prod_{k \ne i,i_0} \frac{\partial T_k^{l_k}}{\partial T_{k c_k}},
& \quad
j = c_i,
\\
0,
& \quad
j \ne c_i,
\end{cases}
\qquad
\delta_{\kappa}(S_k)
\ := \
0,
\ k = 1, \ldots, m.
$$
In both cases, the derivation $\delta_\kappa$ repects 
the $K$-grading. 
Moreover, $\delta_\kappa$ leaves the ideal of defining 
relations of $R(A,P)$ invariant  and thus induces  
a locally nilpotent derivation on $R(A,P)$.
This gives us 
$$ 
\bar \lambda_\kappa 
:= 
\bar \lambda_{\delta_\kappa} 
\colon \KK \ \to \ \Aut(\bar X),
$$
where $\bar \lambda_\kappa$ is the additive 
one-parameter group associated with the 
locally nilpotent derivation $\delta_\kappa$
of $R(A,P)$.
If $\hat X$ is invariant under $\bar \lambda_\kappa$,
for example, if $X$ is complete,
then the \emph{root group} associated with $\kappa$ is
$$
\lambda_\kappa 
:= 
\lambda_{\delta_\kappa} 
\colon \KK \ \to \ \Aut(X).
$$
\end{construction}

\begin{example}
\label{ex:running-example-3}
We continue the discussion started 
in~\ref{ex:running-example-1}
and~\ref{ex:running-example-2}.
Recall that the $\KK^*$-surface 
$X$ comes embedded into 
a weighted projective space~$Z$ 
via
$$
X \ = \ V(T_{01}^7 + T_{11}^2 + T_{21}T_{22}) 
\ \subseteq \
\PP_{2,7,1,13} \ = \ Z.
$$
We determine the root 
group $\lambda_\kappa \colon \KK \to \Aut(X)$
arising from the horizontal Demazure-$P$
root $\kappa = (u,i_0,i_1,C)$ given by 
$$
u \ = \ (-1,0,1),
\qquad
i_0 \ = \ 0,
\qquad
i_1 \ = \ 1,
\qquad
C \ = \ (1,1,2).
$$
First we have to write down the monomials $h^u$ and $h^\zeta$ 
and the vector $\beta$ from Construction~\ref{constr:DEMLND}.
These are
$$ 
h^u 
\ = \
T_{01}^{3}T_{11}^{-1}T_{21},
\qquad
h^\zeta 
\ = \ 
T_{11}^{-1} T_{21},
\qquad 
\frac{h^u}{h^\zeta} 
\ = \ 
T_{01}^3,
\qquad 
\beta 
\ = \ 
(0,1,-1).
$$
Next we describe the derivation $\delta_\kappa \colon R(A,P) \to R(A,P)$.
It is determined by its values on the variables~$T_{ij}$,
which in turn are given as
$$
\delta_\kappa(T_{01}) 
\ = \ 
0,
\qquad\qquad
\delta_\kappa(T_{11}) 
\ = \ 
T_{01}^3 \frac{\partial T_2^{l_2}}{\partial T_{2  2}} 
\ = \ 
T_{01}^3 T_{21},
$$
$$
\delta_\kappa(T_{21}) \ = \ 0,
\qquad\qquad
\delta_\kappa(T_{22}) \ = \ 
-T_{01}^3 \frac{\partial T_1^{l_1}}{\partial T_{1 1}}
\ = \ -2T_{01}^3 T_{11}.
$$
For computing the exponential map, we have to evaluate
the powers of $\delta_\kappa$ on the variables. 
For $T_{11}$ and $T_{22}$ this needs further computation:
$$
\begin{array}{ccccccc}
\delta_{\kappa}^2(T_{11}) 
& = &
\delta_{\kappa}(T_{01}^3T_{21}) 
&  = &
T_{01}^3 \delta_{\kappa}(T_{21}) 
&  = & 
0,
\\[1ex]
\delta_{\kappa}^2(T_{22}) 
& = &
\delta_{\kappa}(-T_{01}^3T_{11})
& = &
-2T_{01}^3\delta_{\kappa}(T_{11}) 
& = &
-2T_{01}^6T_{21},
\\[1ex]
\delta_{\kappa}^3(T_{22}) 
& = &
\delta_{\kappa}(-2T_{01}^6T_{21})
& = &
-2T_{01}^6\delta_{\kappa}(T_{21}) 
& = &
0.
\end{array}
$$
Using this, we directly obtain the comorphism 
$\bar{\lambda}_\kappa(s)^* = \exp(s\delta_\kappa)$.
On the variables~$T_{ij}$ it is given by
\begin{eqnarray*}
T_{01} & \mapsto & T_{01},
\\
T_{11} & \mapsto & T_{11} + s T_{01}^3T_{21},
\\
T_{21} & \mapsto & T_{21},
\\
T_{22} & \mapsto & T_{22} - 2s T_{01}^3T_{11} - s^2 T_{01}^6T_{21}.
\end{eqnarray*}
Consequently, we can represent each automorphism 
$\lambda_\kappa(s) \colon X \to X$, where $s \in \KK$, 
explicitly in Cox coordinates as
$$
[z_{01},\, z_{11},\, z_{21},\, z_{22}] 
\ \mapsto \
[z_{01},\, z_{11} + s z_{01}^3z_{21},\, z_{21},\, 
z_{22} - 2s z_{01}^3z_{11} - s^2 z_{01}^6z_{21}].
$$
\end{example}

One of the central ingredients of our article is the following 
result on automorphisms of varieties with a torus 
action of complecity one.

\begin{theorem}
\label{thm:cpl1aut}
See~\cite[Thm.~5.5 and Cor.~5.11]{ArHaHeLi}.
Let $X = X(A,P,\Sigma)$  
be a complete variety arising 
from Construction~\ref{constr:RAPdown}.
Then $\Aut(X)^0$ is generated as 
a group by the acting torus $T_X$ 
of $X$ and the root groups associated
with the Demazure-$P$ roots.
\end{theorem}

Every Demazure $P$-root in the sense of
Definition~\ref{def:Pdemroot} also
hosts a Demazure root in the sense of
Definition~\ref{def:toricdemazureroot}.
We spend a few words on the relations
among the associated automorphisms.

\begin{remark}
\label{rem:DemPandDem}
Consider a complete variety
$X = X(A,P,\Sigma)$ and its
ambient toric variety $Z$ 
as in Construction~\ref{constr:RAPdown}.
\begin{enumerate}
\item
Let $\kappa = (u,k_0)$ be a vertical
Demazure $P$-root.
Then $u$ is a Demazure root at $v_{k_0}$,
each $\lambda_u(s) \in \Aut(Z)$ 
leaves $X \subseteq Z$ invariant and the
restriction of $\lambda_u(s)$ to $X$ equals
$\lambda_\kappa(s) \in \Aut(X)$.
\item
Let $\kappa = (u,i_0,i_1,C)$ be a horizontal
Demazure $P$-root.
Then $u$ is a Demazure root at $v_{i_1 c_{i_1}}$.
In general, the automorphism
$\lambda_u(s) \in \Aut(Z)$
does not leave $X \subseteq Z$ invariant.
\end{enumerate}
\end{remark}

\begin{example}
\label{ex:running-example-4}
Consider once more the $\KK^*$-surface 
$X \subseteq Z = \PP_{2,7,1,13}$
from~\ref{ex:running-example-1}.
As seen in~\ref{ex:running-example-2},
we have the horizontal Demazure-$P$
root $\kappa = (u,i_0,i_1,C)$, where 
$$
u \ = \ (-1,0,1),
\qquad
i_0 \ = \ 0,
\qquad
i_1 \ = \ 1,
\qquad
C \ = \ (1,1,2).
$$
In Example~\ref{ex:running-example-3}, we computed 
the associated root group.
In Cox coordinates the automorphisms
$\lambda_\kappa(s)$ are given by
$$
[z_{01},\, z_{11},\, z_{21},\, z_{22}] 
\ \mapsto \
[z_{01},\, z_{11} + s z_{01}^3z_{21},\, z_{21},\, 
z_{22} - 2s z_{01}^3z_{11} - s^2 z_{01}^6z_{21}].
$$
The linear form $u$ also defines 
a Demazure root at $v_{11}$ for $Z$ in the 
sense of Definition~\ref{def:toricdemazureroot}. 
By Construction~\ref{constr:toriclnd}
the corresponding automorphisms are
$$
\lambda_{u}(s)\colon Z \ \to \ Z,
\qquad
[z_{01},\, z_{11},\, z_{21},\, z_{22}] 
\ \mapsto \  
[z_{01},\, z_{11}+sT_{01}^3T_{21},\, z_{21},\, z_{22}]
$$
Observe that these ambient automorphisms do not 
leave the surface $X \subseteq Z$ invariant.
For instance, we have 
$$
x \ = \ [1,0,-1,1] \ \in \ X, 
\qquad
\lambda_u(1)(x) \ = \ [1,1,-1,1] \ \not\in \ X.
$$ 
The toric variety $Z = \PP_{2,7,1,13}$ admits 
two further Demazure roots, each at the primitive ray 
generator $v_{22}$, namely
$$
u' \ := \ (0,-1,1), 
\qquad\qquad
u'' \ := \ (-1,-1,2).
$$
The corresponding derivations vanish at all variables $T_{ij}$
except for $T_{22}$,
and on $T_{22}$ the evaluations are given as
$$
\delta_{u'} (T_{22}) \ = \ T_{01}^3 T_{11},
\qquad\qquad
\delta_{u''} (T_{22}) \ = \ T_{01}^6 T_{21}.
$$
With the associated automorphisms of $Z$, 
we can represent $\lambda_\kappa(s)$
on $X$ as a composition of ambient automorphisms
$$
\lambda_\kappa(s)
\ = \
\lambda_u(s) 
\circ 
\lambda_{u'}(-2s)
\circ 
\lambda_{u''}(-s^2) \vert_X
$$
\end{example}

%% file: toric-aut.tex
\section{Automorphisms of complete toric surfaces}
\label{sec:toric-surf-aut}

We consider the toric variety arising
from a complete fan and investigate
groups of automorphisms generated
by the acting torus and root groups
arising from Demazure roots at one
or two generators.
The results are given in
Propositions~\ref{prop:comm-der}
and~\ref{prop:rel-tor-roots}.
As an independent application
we present in Proposition~\ref{prop:tor-surf-aut}
the automorphism groups of
the projective toric surfaces.

\begin{reminder}
\label{rem:semidir-prod}
Let $G,H$ be groups and
$\varphi \colon H \to \mathrm{Aut}(G)$
a homomorphism.
The \emph{semidirect product} is 
the set $G \times H$ together with the
group law
$$
(g,h) \cdot (g',h')
\ := \
(g \varphi(h)(g'),hh').
$$
The notation for the semidirect product is
$G \rtimes_\varphi H$ and we call $\varphi$
the \emph{twisting homomorphism}.
Observe that $G = G \times\{e_H\}$ is a
normal subgroup in $G \rtimes_\varphi H$.
\end{reminder}

\begin{proposition}
\label{prop:comm-der}
Let $\Sigma$ be a complete fan in $\ZZ^n$,
denote by $P = [v_1,\ldots,v_r]$ the 
generator matrix and let~$Z$ be the
associated toric variety.
Moreover, fix $0 \le i_0 \le r$ and
let $u_1, \ldots, u_\rho$ be 
pairwise distinct Demazure roots
at $v_{i_0}$.
\begin{enumerate}
\item
Let $1 \le j \le k \le \rho$. Then 
$\delta_{u_j} \delta_{u_k} = \delta_{u_k} \delta_{u_j} = 0$
holds.
In particular, for any two $s_j,s_k \in \KK$, we have
$$
\bar \lambda_{u_j}(s_j)^* 
\bar \lambda_{u_k}(s_k)^*
\  = \
\bar \lambda_{u_k}(s_k)^* 
\bar \lambda_{u_j}(s_j)^*
$$
\item
Let $U \subseteq \Aut(Z)$ be the subgroup generated
by the root groups $\lambda_{u_1}, \ldots, \lambda_{u_\rho}$.
Then we have an isomorphism of algebraic groups
$$
\qquad
\Theta
\colon
\KK^\rho
\ \to \
U,
\qquad
(s_1,\ldots,s_\rho)
\ \mapsto \ 
\lambda_{u_1}(s_1) \cdots \lambda_{u_\rho}(s_\rho).
$$
\item
The acting torus $\TT^n \subseteq \Aut(Z)$
normalizes $U \subseteq \Aut(Z)$ and we have
an isomorphism of algebraic groups
$$
\Psi \colon \KK^\rho \rtimes_\psi \TT^n \ \to \ U\TT^n,
\qquad
(s,t) \ \mapsto \ \Theta(s)t,
$$
where the twisting homomorphism
$\psi \colon \TT^n \to \Aut(\KK^\rho)$
sends $t \in \TT^n$ to the diagonal matrix
$\diag(\chi^{u_1}(t), \ldots,\chi^{u_\varrho}(t))$.
\end{enumerate}
\end{proposition}

In the proof and also later,
we will make use of the following fact;
see for instance~\cite[Lemma~1]{Dz}.

\begin{lemma}
\label{lem:lambda-u-normal}
Consider a complete fan $\Sigma$ in $\ZZ^n$,
the associated toric variety $Z$ 
and a Demazure root $u$ of $\Sigma$.
Then, for all $t \in \TT^n$ and $s \in \KK$,
the root group $\lambda_u$ associated with
$u$ satisfies
$t^{-1} \lambda_u(s) t = \lambda_u (\chi^u(t)s)$.
\end{lemma}

\begin{proof}[Proof of Proposition~\ref{prop:comm-der}]
We prove~(i). Due to the definition of
$\delta_{u_j}$ and $\delta_{u_k}$  
given in Construction~\ref{constr:toriclnd},
we have
$\delta_{u_k}(T_i) = \delta_{u_j}(T_i) = 0$
for all $i \ne i_0$.
We conclude
$$
\delta_{u_j} \delta_{u_k} (T_i)
\ = \
0
\ = \
\delta_{u_k} \delta_{u_j} (T_i),
$$
whenever $i \ne i_0$. Moreover,
using  the Leibniz rule, we
compare the evaluations at~$T_{i_0}$
and obtain
$$
\delta_{u_j} \delta_{u_k} (T_{i_0})
\ = \
\delta_{u_j}
\prod_{i \ne i_0} T_i^{\bangle{u_k,v_i}}
\ = \
0
\ = \
\delta_{u_k}
\prod_{i \ne i_0} T_i^{\bangle{u_j,v_i}}
\ = \
\delta_{u_k} \delta_{u_j} (T_{i_0}).
$$
As $\delta_j$ and $\delta_k$ commute,
we can apply the homomorphism
property of the exponential map which
yields
$$
\bar \lambda_{u_j}(s_j)^* 
\bar \lambda_{u_k}(s_k)^*
 =
\exp(s_j \delta_{u_j} + s_k \delta_{u_k})
=
\exp(s_k \delta_{u_k} + s_j \delta_{u_j})
=
\bar \lambda_{u_k}(s_k)^* 
\bar \lambda_{u_j}(s_j)^*.
$$
This proves~(i). As a consequence,
$\lambda_{u_j}$ and $\lambda_{u_k}$ commute.
Thus, the map $\Theta$ from~(ii) is a homomorphism.
Let us see why $\Theta$ is injective. Assume
$$
\Theta(s_1,\ldots,s_\rho)
\ = \
\lambda_{u_1}(s_1) \cdots \lambda_{u_\rho}(s_\rho)
\ = \
\id_Z.
$$
The task is to show $s_1 = \ldots = s_\rho =0$.
In Cox coordinates, the automorphism $\vartheta := \Theta(s_1,\ldots,s_\rho)$
of $Z$ is given by 
$$
\bar \vartheta \colon
z
\ \mapsto \
\tilde z,
\qquad
\tilde z_i
\ = \
\begin{cases}
z_{i_0} + s_1 z^{P^*u_1}z_{i_0}  + \ldots + s_\rho z^{P^*u_\rho} z_{i_0},
&
i = i_0,
\\
z_i
& i \ne i_0.
\end{cases}
$$
Consider the set $\hat Z_0 \subseteq \hat Z$
obtained by removing all $V(T_i,T_j)$ from $\bar Z = \KK^r$,
where $i \ne j$.
Then $\vartheta = \id_Z$ implies that
there is a morphism $h \colon \hat Z_0 \to H = \ker(p)$
with
$$
\bar \vartheta (z)
\ = \
h(z) \cdot z
\qquad
\text{for all }
z \in \hat Z_0.
$$
As $H$ is a quasitorus, $h$ must be constant.
Since $H$ acts freely on $V(T_{i_0}) \cap \hat Z_0$,
we obtain $h(z) = e_H$ for all $z \in \hat Z_0$.
Consequently,
$$
s_1 z^{P^*u_1} + \ldots + s_\rho z^{P^*u_\rho}
\ = \
0
\qquad
\text{for all }
z \in \TT^r .
$$
Since $v_1, \ldots, v_r$ generate $\QQ^n$
as a vector space, the dual map $P^*$ is
injective.
Thus, as $u_1, \ldots, u_\rho$ are pairwise distinct,
we can conclude
$$
s_1 \ = \ \ldots \ = \ s_\rho \ = \ 0.
$$
We verify~(iii).
Using~(ii) and Lemma~\ref{lem:lambda-u-normal},
we see that $\TT^n$ normalizes $U$.
In particular, $U\TT^n \subseteq \Aut(Z)$
is a closed subgroup.
By the definition of $\Theta$ and
applying again Lemma~\ref{lem:lambda-u-normal},
we obtain
$$
\Theta(s)\Theta(\varphi(t)(s'))
\ = \
\Theta(s)\Theta(\chi^{u_1}(t)s_1', \ldots,\chi^{u_\varrho}(t)s_\varrho')
\ = \
\Theta(s)t^{-1}\Theta(s')t.
$$
We conclude that $\Psi$ is a group homomorphism.
Since $U$ is unipotent and every $t \in \TT$
is semisimple, we have $U \cap \TT^n= \{\id_{Z}\}$.
Thus, $\Psi$ injective.
Using~(ii) again, we see that $\Psi$ is surjective.
\end{proof}

\begin{proposition}
\label{prop:rel-tor-roots}
Let $\Sigma$ be a complete fan 
with primitive generators $v_1,\ldots,v_r$
and~$Z$ the associated toric variety.
For $i_0 \ne i_1$ and $\varepsilon \ge 0$,
let $u$, $u_\varepsilon$ be Demazure roots
at $v_{i_0}$, $v_{i_1}$ respectively,
such that $\bangle{u,v_{i_1}} = 0$
and $\bangle{u_\varepsilon,v_{i_0}} = \varepsilon$.
\begin{enumerate}
\item
For $\mu = 0, \ldots, \varepsilon$, set
$u_\mu := u_\varepsilon + (\varepsilon - \mu) u$.
Then each $u_\mu$ is a Demazure root at $v_{i_1}$
with $\bangle{u_\mu,v_{i_0}} = \mu$ 
and for every $0 \le \alpha \le \varepsilon$
we have
$$
\lambda_{u_\alpha}(s) 
\lambda_{u}(q)
\ = \
\lambda_{u}(q) \prod_{\mu=0}^{\alpha}
\lambda_{u_\mu}
\left(
s {\alpha \choose \mu} q^{\alpha-\mu}
\right).
$$
\item
Let $U,V \subseteq \Aut(Z)$ be the subgroups generated
by $\lambda_{u_0}, \ldots, \lambda_{u_\varepsilon}$
and by $\lambda_u$, respectively.
Then $V$ normalizes $U$ and 
$$
\qquad\qquad
\Phi \colon \KK^{\varepsilon +1} \rtimes_{\varphi} \KK
\ \to \
UV,
\quad
(s_0, \ldots, s_\varepsilon, q)
\ \mapsto \
\lambda_{u_0}(s_0) \cdots \lambda_{u_\varepsilon}(s_\varepsilon) \lambda_u(q)
$$
is an isomorphism of algebraic groups,
where the twisting homomorphism 
$\varphi \colon \KK \to \Aut(\KK^{\varepsilon+1})$ 
is given by the matrix valued map
$$
\qquad\qquad
q \mapsto A(q) = (a_{\mu \alpha}(q)),
\text{ where } 
a_{\mu \alpha}(q)
\ = \ 
\begin{cases}
{\alpha-1 \choose \mu-1} q^{\alpha-\mu},
& 
\alpha \ge \mu, 
\\
0, 
& 
\alpha < \mu.
\end{cases}
$$
\item
The acting torus $\TT^n \subseteq \Aut(Z)$
normalizes $UV \subseteq \Aut(Z)$ and we have
an isomorphism of algebraic groups
$$
\qquad\qquad
\Psi \colon (\KK^{\varepsilon +1} \rtimes_{\varphi} \KK) \rtimes_{\psi} \TT^n
\ \to \ UV\TT^n,
\qquad
(s,q,t) \ \mapsto \ \Phi(s,q)t,
$$
where the twisting homomorphism
$\psi \colon \TT^n \to \Aut(\KK^{\varepsilon +1} \rtimes_{\varphi} \KK)$
sends $t \in \TT^n$ to the diagonal matrix
$\diag(\chi^{u_0}(t), \ldots,\chi^{u_\varepsilon}(t),\chi^u(t))$.
\end{enumerate}
\end{proposition}

\begin{proof}
We show~(i).
The fact that each $u_\mu$ is a Demazure root
as claimed is directly verified.
Now, write $P = [v_1,\ldots,v_r]$.
Then we compute in Cox coordinates:
\begin{eqnarray*}
\lambda_{u}(q)(z)^{P^*(u_\alpha)}
& = & 
\left(z+qz_{i_0} z^{P^*(u)}e_{i_0} \right)^{P^*(u_\alpha)}
\\
& = &
\left(
z_{i_0}+qz_{i_0} \prod_{i=0}^r z_i^{\bangle{u,v_i}}
\right)^\alpha
\prod_{i \ne i_0}
z_i^{\bangle{u_\alpha,v_i}}
\\
& = & 
\left(\sum_{\mu = 0}^\alpha
{\alpha \choose \mu}
q^{\alpha-\mu}
z_{i_0}^\alpha
\prod_{i=0}^r z_{i}^{(\alpha-\mu)\bangle{u,v_i}}
\right)
\prod_{i \neq i_0}
z_i^{\bangle{u_\alpha,v_i}}
\\
& = & 
\sum_{\mu = 0}^\alpha
{\alpha \choose \mu}
q^{\alpha-\mu}
\prod_{i=0}^r 
z_i^{\bangle{u_\alpha+(\alpha-\mu)u,v_i}}
\\
& = & 
\sum_{\mu = 0}^\alpha
{\alpha \choose \mu}
q^{\alpha-\mu}
z^{P^*(u_\mu)} .
\end{eqnarray*}
Next, using $\bangle{u,v_{i_1}} = 0$,
we see that the monomial $z^{P^*(u)}$
does not depend on $z_{i_1}$ and thus,
for any $a \in \KK$, obtain
\begin{eqnarray*}
\lambda_{u}(q)(z + ae_{i_1})
& = &        
z + ae_{i_1} + q(z + ae_{i_1})^{P^*(u)}
\\
& = &
z +  qz^{P^*(u)} + ae_{i_1}
\\
& = & 
\lambda_{u}(q)(z) + ae_{i_1}.
\end{eqnarray*}
Set for short
$t_\mu := s {\alpha \choose \mu} q^{\alpha-\mu}$.
Then, applying the two computations just performed, we can
verify the displayed formula as follows:
\begin{eqnarray*}
\lambda_{u_\alpha}(s) \lambda_{u}(q)(z)
& = & 
\lambda_{u}(q)(z)
 +
s z_{i_1} \lambda_{u}(q)(z)^{P^*(u_\alpha)}e_{i_1}
\\
& = & 
\lambda_{u}(q)(z)
+
\sum_{\mu = 0}^\alpha
t_\mu z_{i_1}z^{P^*(u_\mu)}
e_{i_1}
\\
& = & 
\lambda_{u}(q)(\lambda_{u_0}(t_0)(z))
+
\sum_{\mu = 1}^\alpha
t_\mu z_{i_1}z^{P^*(u_\mu)}
e_{i_1}
\\
& \vdots & 
\\
& = & 
\lambda_{u}(q)
\lambda_{u_\alpha}(t_\alpha)
\cdots
\lambda_{u_0}(t_0)
(z),
\end{eqnarray*}
where we used that all the $u_\mu$ are Demazure roots at  
$v_{i_1}$ and hence for every $\mu = 0, \ldots , \alpha-1$
we have 
$$z_{i_1}z^{P^*(u_{\mu+1})}
\ = \
\lambda(t_\mu)(z)_{i_1}\lambda(t_\mu)(z)^{P^*(u_{\mu})}.
$$

We turn to~(ii).
First note that $V$ normalizes $U$,
since by the identity of~(i) it
normalizes each $\lambda_{u_\alpha}(\KK)$,
where $\alpha = 0, \ldots, \varepsilon$.
In particular, $UV$ is a closed subgroup
of $\Aut(Z)$.
Now, for $s \in \KK^{\varepsilon+1}$ and
$q \in \KK$, set
$$
\Psi(s) \ := \ \Psi(s,0)
\qquad\qquad
\Psi(r) \ := \ \Psi(0,q).
$$
By the nature of $\Psi$, we then have
$\Psi(s,q) = \Psi(s) \Psi(q)$.
Moreover, showing that~$\Psi$ respects 
the multiplication of $(s,q)$ and $(s',q')$
means to verify
$$
\Psi(q) \Psi(s') \ = \ \Psi(\varphi(q)(s')) \Psi(q).
$$
For this, write $s' = s'_0e_0 + \ldots + s'_\varepsilon e_\varepsilon$
with the canonical basis vectors $e_i \in \KK^{\varepsilon+1}$.
Then for each $s'_ie_i$, the above equality is
a direct consequence of the definition of $\varphi$
and the identity provided by~(i).
Thus, $\Psi$ is a homomorphism.

Proposition~\ref{prop:comm-der} tells us that
$\Psi$ maps $\KK^{\varepsilon +1}$ isomorphically
onto $U$ and $\KK$ isomorphically onto $V$.
In particular, $\Psi$ is surjective.
Now consider an element
$(s,q) \in \ker(\Psi)$ with $s \ne 0$ and
$q \ne 0$.
As in the proof of
Proposition~\ref{prop:comm-der},
we work in Cox coordinates.
The element $(s,q)$ restricted to the
identity on $V(T_{i_0}, T_{i_1}) \subseteq \bar Z$.
By our assumptions, $v_{i_0}$ and $v_{i_1}$ form
part of a lattice basis of $\ZZ^n$ and thus
$H$ acts freely on $V(T_{i_0}, T_{i_1})$;
see~\cite[Prop.~2.1.4.2]{ArDeHaLa}.
We conclude $s=0$ and $q=0$.

Finally we note that the verification of Assertion~(iii)
runs exactly as for the corresponding statement 
of Proposition~\ref{prop:comm-der}.
\end{proof}

We enter the surface case.
The first step towards our description
of the automorphism groups is
combinatorial: we specify the fans admitting
Demazure roots at two or more
primitive generators.

\begin{proposition}
\label{prop:roots-at-two-gen}
Let $N$ and $M$ be mutually dual 
two-dimensional lattices.
Consider distinct primitive 
vectors $v,v',w,w' \in N$ and 
$u,u' \in M$ such that 
$$
\begin{array}{ccccccc}
\bangle{u,v} =  -1,  
&
& 
\xi := \bangle{u,v'} \ge 0,
&
&
\bangle{u,w} \ge 0,
&
&
\bangle{u,w'} \ge 0,
\\[1em]
\bangle{u',v'} =  -1,  
&
& 
\xi' := \bangle{u',v} \ge 0,
&
&
\bangle{u',w} \ge 0.
&
&
\bangle{u',w'} \ge 0.
\end{array}
$$
Assume $\xi = 0$ or $\xi' = 0$. 
Then we have $w \in \cone(-v,-v')$
and, choosing suitable $\ZZ$-linear 
coordinates on $N$, we achieve
$$
v \ = \ (1,0)
\qquad
v' \ = \ (0,1),
\qquad
u \ = \ (-1,\xi), 
\qquad
u' = (\xi',-1).
$$
Assume $\xi, \xi' > 0$, $w \not\in \cone(v,v')$
and $w' \ne w$.
Then each of $u,u'$ annihiliates~$w$ and~$w'$.
Choosing suitable $\ZZ$-linear
coordinates on $N$ and $b \in \ZZ_{\ge 1}$,
we have 
$$
\begin{array}{cccccc}
v =  (1,0),
&&
u  = (-1,1),
&&
w = (-1,-1),
\\[1em]
v' = (b-1,b),
&&
u' = (1,-1),
&&
w' = (1,1).
\end{array}
$$
\end{proposition}

\begin{proof}
By assumption $v$ and $v'$ generate $N_\QQ$
as a vector space and $w$ is not a multiple of
one of $v,v'$.
Thus, we can write $w = -\eta v - \eta'v'$
with $\eta,\eta' \in \QQ$.
Evaluating the linear forms~$u$ and $u'$ yields
$$
\eta - \xi \eta'
\ = \
\bangle{u,w}
\ \ge \
0,
\qquad\qquad 
\eta' - \xi' \eta
\ = \
\bangle{u',w}
\ \ge \ 
0.
$$
Assume $\xi\xi' = 0$.
Then $\eta, \eta' \ge 0$, hence
$w \in \cone(-v,-v')$.
Moreover, $(v,v')$ is a basis of
$N$ as it is sent via $(u,u')$
to a basis of $\ZZ^2$.
Clearly, $(v,v')$ provides the desired
coordinates.
Now assume $\xi, \xi' > 0$ and $w \not\in \cone(v_1,v_2)$.
Then
$$
\eta \ \ge \ \xi \eta' \ \ge \ \xi \xi' \eta \ > \ 0,
\qquad
\eta' \ \ge \ \xi' \eta \ \ge \ \xi' \xi \eta' \ > \ 0.
$$
We conclude $\xi = \xi'= 1$
and $\eta = \eta'$.
This implies $u' = -u$.
Consequently, each of the linear forms $u$ and $u'$
annihilates~$w$ and $w'$.
With respect to suitable $\ZZ$-linear
coordinates on~$N$, we have
$$ 
v \ = \ (1,0),
\qquad
v' \ = \ (a,b), \text{ where } 0 \le a < b.
$$ 
Then $\bangle{u,v} = -1$ implies $u = (-1,x)$ 
with $x \in \ZZ$.
Moreover, $\bangle{u,v'} = 1$ 
gives us $bx = a + 1$.
Because of $b > a$, we must have $x=1$.
Consequently, $b = a+1 \ge 1$ holds.
We conclude
$$
u = (-1,1),
\quad
u' = (1,-1),
\quad
w = (-1,-1),
\quad
w' = (1,1).
$$
\end{proof}

\begin{corollary}
\label{cor:toric-surf-roots}
Let $\Sigma$ be a complete fan in $\ZZ^2$
and $P = [v_1,\ldots, v_r]$ the generator matrix
of~$\Sigma$.
Let Demazure roots $m_1$ at $v_1$ and $m_2$ at $v_2$
be given.
\begin{enumerate}
\item 
Assume that $\bangle{m_1,v_2} = 0$ or $\bangle{m_2,v_1} = 0$
holds.
Then, with respect to suitable $\ZZ$-linear coordinates,
we have
$$
\qquad\qquad
v_1 = (1,0), 
\qquad
v_2 = (0,1),
\qquad
v_i \in \cone(-v_1,-v_2), \ i = 3, \ldots, r
$$
and, denoting by $\xi(\Sigma)$ the minimum
of the slopes of the lines $\QQ v_3, \ldots, \QQ v_r$,
the Demazure roots of $\Sigma$ at $v_1$ and $v_2$
are given as
$$
\qquad
u = (-1,0),
\qquad
u_\xi = (\xi,-1),
\quad
0 \le \xi \le \xi(\Sigma).
$$
Assume in addition that some $v_i$ with $i \ge 3$
admits a Demazure root $u'$.
Then, according to the value of $\xi(\Sigma)$
and suitably renumbering, we have
$$
\qquad\qquad
\begin{array}{ll}
\xi(\Sigma) = 0:
&
v_3 = (-1,0),
\quad
u' = (1,0),
\quad
v_4 = (0,-1),
\quad
u'' = (0,1),
\\[1em]
\xi(\Sigma) > 0:
&
v_3 = (-1,-\xi(\Sigma)),
\quad
u' = (1,0),
\quad
v_4 = (0,-1)
\end{array}
$$
as the lists of the remaining primitive generators
and the remaining Demazure roots of the fan $\Sigma$,
where $v_4$ is optional in the second case.
\item
Assume that $\bangle{m_1,v_2} > 0$ and $\bangle{m_2,v_1} > 0$
hold.
Then $r = 3$ or $r = 4$ and
with respect to suitable $\ZZ$-linear
coordinates, we have
$$
\qquad\qquad
v_1 = (1,0), 
\quad
v_2 = (b-1,b),
\quad
v_3 = (-1,-1),
\quad
v_4 = (1,1),
$$
where $b \ge 1$.
The possible Demazure roots of $\Sigma$
are $u=(-1,1)$ at $v_1$,
next $u' = (1,-1)$ at $v_2$ and,
$u_\xi = (\xi,1-\xi)$ at $v_3$, where $0 \le \xi \le b$.
\end{enumerate}
\end{corollary}

\begin{proof}
We show~(i).
We may assume $\bangle{m_1,v_2} = 0$.
In this setting, the first part of
Proposition~\ref{prop:roots-at-two-gen}
shows $v_i \in \cone(-v_1,-v_2)$
for $i = 3, \ldots, r$ and provides
us with the desired coordinates.
Moreover, $0 \le \xi \le \xi(\Sigma)$
holds because $\bangle{u_\xi}$ evaluates
non-negatively on~$v_i$ for $i \ge 3$.
Now, let $v_i$ with $i \ge 3$ admit
a Demazure root $u'$.
We may assume $i=3$ and then have
$$
v_3 \ = \ (-a,-b),
\quad a,b \in \ZZ_{\ge 0},
\qquad\qquad
u' \ = \ (c,d),
\quad
c,d \in \ZZ_{\ge 0}.
$$
So, $\bangle{u',v_3} = -1$ means $-ac-bd = -1$.
Consequently, $ac=0$ or $bd = 0$.
Consider the case $bd = 0$.
Then $a = c = 1$.
According to $b=0$ and $d=0$, we arrive at
$$
u' \ = \ (1,0),
\quad
v_3 \ = \ (-1,0),
\qquad\qquad
u' \ = \ (1,0),
\quad
v_3 \ = \ (-1,-b).
$$
Observe $b = \xi(\Sigma)$ and that we
must add $v_4 = (0,-1)$ if $b=0$ and
may add it if $b>0$.
The case $ac=0$ delivers nothing new.
We turn to~(ii).
Here, we may assume $v_3 \not\in \cone(v_1,v_2)$.
Then the second part of
Proposition~\ref{prop:roots-at-two-gen}
brings us to the setting of~(ii)
and, similarly as above, the statement
on the Demazure roots is obtained via
straightforward calculation.
\end{proof}

\begin{remark}
The toric surfaces behind the fans
of Corollary~\ref{cor:toric-surf-roots}
provided we have at least three primitive
allowing a Demazure root are the following.
In~(i), we have
$\PP_1 \times \PP_1$ if $\xi(\Sigma)=0$
and for $\xi(\Sigma)>0$ we obtain the weighted
projective plane $\PP_{1,1,b}$ for $r=3$ and
the Hirzebruch surfaces $Z_{b}$, where $b \in \ZZ_{\ge 1}$,
for $r=4$.
The latter two are also the ones showing up
in~(ii).
\end{remark}

Now we begin to invesigate of the structure
of the automorphism group in the surface case.

\begin{proposition}
\label{prop:Aut-P11b}
Situation as in 
Corollary~\ref{cor:toric-surf-roots}~(ii).
Then for $r=3$ and $r=4$ the following
holds.
For $0 \le \xi \le b$, the root subgroups
given by $u$, $u'$ and~$u_\xi$ satisfy
$$
\lambda_{u_\xi}(r) 
\lambda_{u}(s) \lambda_{u'}(s')
\ = \
\lambda_{u}(s) \lambda_{u'}(s')
\prod_{\nu = 0}^b
{\lambda}_{u_{\nu}}
\left(
r
\sum_{\mu=0}^{\mathrm{min}(\nu,\xi)}
{\xi \choose \mu}
{b-\mu \choose b-\nu}
s^{\xi-\mu}
(s')^{\nu-\mu}
\right).
$$
Let $U \subseteq \Aut(X)$ be 
the subgroup generated by 
$\lambda_{u_0},\ldots, \lambda_{u_b},\lambda_u$
and $\lambda_{u'}$.
Then we have an isomorphism 
of algebraic groups
\begin{eqnarray*}
\Psi_2\colon \KK^{b+1} \rtimes_{\varphi_2} \KK^2,
& \to &
U,
\\
\quad 
(r_0,\dots,r_{b},s,s')
& \mapsto & 
\lambda_{u}(s)\circ \lambda_{u'}(s')
\lambda_{u_0}(r_0)
\cdots 
\lambda_{u_b} (r_b),
\end{eqnarray*}
Here, the twisting homomorphism 
$\varphi_2\colon \KK^2 \to \Aut(\KK^{b+1})$
is given by the matrix valued map
$(s,s') \mapsto B(s,s') \ = \ (b_{ji}(s,s'))$,
where $b_{ji}(s,s') = 0$ for $i < j$ and 
$$
b_{ji}(s,s')
\ = \ 
\sum_{\mu = 0}^{\mathrm{min}(j-1,i-1)}
{i-1 \choose \mu}
{b-\mu \choose b-j+1}
s^{i-1-\mu}
(s')^{j-1-\mu}
\qquad
\text{ for } i \ge j.
$$ 
\end{proposition}

\begin{proof}
First recall that both $u$ and $u'$ evaluate
to zero on $v_3$. 
Furthermore, one directly computes
$$
\bangle{u_\xi,v_1} = \xi,
\quad
u_\xi + (\xi-\mu)u  = u_{\mu},
\quad
\bangle{u_\mu,v_2} = b-\mu,
\quad 
u_\mu + (b - \mu -\nu)u' = u_{b-\nu}.
$$
Thus, applying Proposition~\ref{prop:rel-tor-roots}
to the pairs roots $u, u_\xi$ and $u',u_\xi$,
we can verify the first assertion:
\begin{eqnarray*}
{\lambda}_{u_\xi}(r)
{\lambda}_{u}(s){\lambda}_{u'}(s')
& = & 
{\lambda}_{u}(s)
\prod_{\mu=0}^{\xi}
\lambda_{u_\mu}
\left(
r {\xi \choose \mu} s^{\xi-\mu}
\right)
{\lambda}_{u'}(s')
\\
& = & 
{\lambda}_{u}(s){\lambda}_{u'}(s')
\prod_{\mu = 0}^\xi
\prod_{\nu = 0}^{b-\mu}
{\lambda}_{b-\nu} 
\left( 
r {\xi \choose \mu} s^{\xi-\mu} 
{{b-\mu \choose \nu} (s')^{b-\mu-\nu}}
\right)
\\
& = & 
{\lambda}_{u}(s){\lambda}_{u'}(s')
\prod_{\nu = 0}^b
\prod_{\mu = 0}^{\mathrm{min}(\nu,\xi)}
{\lambda}_{u_{\nu}}
\left( 
r {\xi \choose \mu} s^{\xi-\mu} 
{{b-\mu \choose b-\nu} (s')^{\nu-\mu}}
\right)
\\
& = & 
{\lambda}_{u}(s){\lambda}_{u'}(s')
\prod_{\nu = 0}^b
{\lambda}_{u_{\nu}}
\left(
r
\sum_{\mu=0}^{\mathrm{min}(\nu,\xi)}
{\xi \choose \mu}
{b-\mu \choose b-\nu}
s^{\xi-\mu}
(s')^{\nu-\mu}
\right)
\end{eqnarray*}
For second assertion, we proceed similarly
as in the proof of 
Proposition~\ref{prop:rel-tor-roots}.
The property that $\Psi_2$ is a group homomorphism
reduces to the following, whih is a consequnece
of the first assertion.
$$
\Psi_2(r) \Psi_2(s,s')
\ = \
\Psi_2(\varphi_2(s,s')) \Psi_2(r).
$$
By definition, $\Psi_2$ is surjective.
To see injectivity, take 
$(s,s',r)\in \ker(\Psi_2)$.
Then, working Cox coordinates,
we find an $h \in H$ with
$\Psi_2(s,s',r)(z) = h \cdot z$.
we conclude $s=0$, $s'=0$ and
$r = 0$, using the fact that the
$H$-action is explicitly known
in our situation.
\end{proof}

\begin{proposition}
\label{prop:tor-surf-aut}
Let $\Sigma$ be a complete fan in $\ZZ^2$,
denote by $\ell(\Sigma)$ the number
of primitive ray generators admitting
a Demazure root and by $\rho(\Sigma)$
the total number of Demazure roots
of $\Sigma$.
Then the unit component of the automorphism
group of the toric surface $Z$ defined
by $\Sigma$ is given as follows:
\begin{center}
\begin{tabular}{c|c|c|c}
  $Z$
  & $\ell(\Sigma)$
  & $\mathrm{Aut}^0(Z)$
  & $\rho(\Sigma)$
  \\
  \hline
  $-$
  &
  $0$
  &
  $\TT^2$
  &
  0
  \\
  $-$
  &
  $1$
  &
  $\KK^\rho \rtimes_{\psi_1} \TT^2$
  &
  $\rho$
  \\
  $-$
  &
  $2$
  &
  $(\KK^\rho \rtimes_\varphi\KK) \rtimes_{\psi_2} \TT^2$
  &
  $\rho + 1$
  \\
  $\PP_{1,1,b}, \, b \ge 2$
  &
  $3$
  &
  $(\KK^{b+1} \rtimes_{\varphi'} \KK^2) \rtimes_{\psi_3} \TT^2$
  &
  $b+3$
  \\
  $Z_b, \, b \ge 1$
  &
  $3$
  &
  $(\KK^{b+1} \rtimes_{\varphi'} \KK^2) \rtimes_{\psi_3} \TT^2$
  &
  $b+3$
  \\
  $\PP_2$
  & $3$
  &
  $\mathrm{PGL}_3(\KK)$
  &
  $6$
  \\ 
  $\PP_1 \times \PP_1$
  &
  $4$
  & 
  $\mathrm{PGL}_2(\KK) \times \mathrm{PGL}_2(\KK)$
  &
  $4$
  \\
\end{tabular}
\end{center}
The twisting homomorphisms
$\varphi$ and $\varphi'$
are given by upper triangular matrices
$\varphi(s) = A  =  (a_{j i})$
and $\varphi'(s,s') = B = (b_{j i})$,
where for $i \ge j$ we have
$$
a_{j i} 
\ = \ {i-1\choose j-1}s^{i-j},
\qquad
b_{j i}
\ = \
\sum_{\mu = 0}^{\mathrm{min}(j-1,i-1)}
{i-1 \choose \mu}
{b-\mu \choose b-j+1}
s^{i-1-\mu}
(s')^{j-1-\mu}.
$$
Furthermore the twisting homomorphism
for the toric factors are given by the following
diagonal matrices:
\begin{eqnarray*}
\psi_1(t_1,t_2)
& = &  
\diag(t_1^{-1}, \ldots, t_1^{-1})
\\
\psi_2(t_1,t_2)
& = &  
\diag(t_2^{-1},
t_1 t_2^{-1}, \ldots, t_1^{\rho-1}t_2^{-1},t_1^{-1}),
\\
\psi_3(t_1,t_2)
& = &  
\diag(t_2,t_1,t_1^2t_2^{-1}, \ldots, t_1^{\rho}t_2^{1-\rho},t_1t_2^{-1},t_1^{-1}t_2).
\end{eqnarray*}
\end{proposition}

\begin{proof}
The cases of $\PP_2$ and $\PP_1\times \PP_1$
are well known.
So, let~$Z$ arise from a complete fan
admitting Demazure roots and denote by
$U \subseteq Z$ the subgroup generated
generated by all the root groups.
If only one primitive generator of
$\Sigma$ allows Demazure roots, then
Proposition~\ref{prop:comm-der}
yields $U \cong \KK^\rho$.
If there are roots at exactly
two primitive ray generators,
Proposition~\ref{prop:rel-tor-roots}
shows that $U \cong \KK^\rho \rtimes_\varphi \KK$
is as claimed.
In the cases $Z = \PP_{1,1,b}$ or $Z = Z_b$,
Proposition~\ref{prop:Aut-P11b} tells us
that $U \cong K^\rho \rtimes_{\varphi'} \KK^2$
is as in the assertion.
Finally, using Lemma~\ref{lem:lambda-u-normal}
we see that also the twisting
homomorphisms~$\psi_1$, $\psi_2$ and $\psi_3$
are the right ones.
\end{proof}

%% file: ambient-morph.tex
\section{Representation via toric
ambient automorphisms}

The main result of this section, 
Theorem~\ref{thm:restrrootauts}, 
is an important ingredient for the explicit 
handling of automorphisms in the subsequent 
sections; it represents automorphisms of a 
variety with torus action of complexity one 
as restrictions of explicitly accessible 
automorphisms of its ambient toric variety.

\begin{construction}
\label{constr:uvi}
Situation as in Construction~\ref{constr:RAP}.
Let $\kappa = (u,i_0,i_1,C)$ be a horizontal
Demazure $P$-root.
Define linear forms
$$ 
u_{\nu, \iota} \ := \ \nu u + e_{i_1}' - e_\iota' \ \in \ M ,
\qquad 
\nu = 1, \ldots, l_{i_1c_{i_1}}, 
\quad
\iota \ne i_0,i_1,
$$
where $e_0' := 0$ and $e_1', \ldots, e_r' \in M$
are the first $r$ canonical basis vectors.
So, the values of $u_{\nu, \iota}$ on the columns 
$v_{ij}$ and $v_k$ of $P$ are given by 
$$ 
\bangle{u_{\nu,\iota},v_{ij}} 
\ = \ 
\begin{cases} 
\nu \bangle{u,v_{ij}} + l_{ij}, & i = i_1, 
\\
\nu \bangle{u,v_{ij}} - l_{ij} & i = \iota, 
\\ 
\nu \bangle{u,v_{ij}} & i \ne i_1, \iota,
\end{cases} 
\qquad 
\bangle{u_{\nu,\iota},v_{k}} = \nu \bangle{u,v_{k}},
\
k = 1, \ldots, m.
$$ 
\end{construction}

\begin{lemma}
\label{lem:deriv}
Let $\delta$ and $\delta'$ be derivations 
on the polynomial ring $\KK[T_1, \ldots, T_n]$.
\begin{enumerate}
\item
We have $\delta(a) = 0$ for every $a \in \KK$.
\item
If $\delta(T_i) = \delta(0)$ holds 
for $i = 1, \ldots, n$, then 
$\delta = 0$ holds.
\item
If $\delta'(T_i) = \delta(T_i)$ holds 
for $i = 1, \ldots, n$, then 
$\delta' = \delta$ holds.
\item
If $\delta'\delta(T_i) = 0$ 
holds for $i = 1, \ldots, n$, then 
$\delta'\delta = 0$ holds.
\item
If $\delta'\delta(T_i) = \delta \delta'(T_i)$ 
holds for $i = 1, \ldots, n$, then 
$\delta'\delta = \delta \delta'$ holds.
\end{enumerate}
\end{lemma}

\begin{proof}
The statements are directly verified by using
the defining properties of derivations and fact that
$\KK[T_1, \ldots, T_n]$ is generated as a $\KK$-algebra
by the variables $T_1, \ldots, T_n$.
\end{proof}
  
\begin{lemma}
\label{lem:linformsunuiota}
Consider linear forms $u_{\nu,\iota}$
and $u_{\nu',\iota'}$ 
as in Construction~\ref{constr:uvi}
and the associated derivations
$\delta_{u_{\nu,\iota}}$ and $\delta_{u_{\nu',\iota'}}$
on $\KK[T_{ij},S_k]$ as 
provided by Construction~\ref{constr:toriclnd}.
\begin{enumerate}
\item
The linear form $u_{\nu,\iota} \in M$ 
is a Demazure root at $v_{\iota c_\iota} \in N$ 
in the sense of 
Definition~\ref{def:toricdemazureroot}. 
\item
The derivation $\delta_{u_{\nu,\iota}}$ annihilates 
all variables $T_{ij}$ and $S_k$ except $T_{\iota c_\iota}$,
where we have 
$$
\delta_{u_{\nu,\iota}} (T_{\iota c_\iota})
\ = \
f_{u,\nu,\iota} T_{i_1c_{i_1}}^{l_{i_1c_{i_1}-\nu}}
$$
with a monomial $f_{u,\nu,\iota}$ in the variables
$T_{ij}$, $S_k$ but not depending on 
any~$T_{ic_i}$ with $i \ne i_0$.
\item
We have 
$
\delta_{u_{\nu',\iota'}} \delta_{u_{\nu,\iota}} 
= 
\delta_{u_{\nu,\iota}} \delta_{u_{\nu',\iota'}}
= 
0
$.
In particular, the derivations $\delta_{u_{\nu,\iota}}$ and 
$\delta_{u_{\nu',\iota'}}$ commute.
\item
For any two $s,s' \in \KK$, the automorphisms 
$\bar \lambda_{u_{\nu,\iota}}(s)^*$ and 
$\bar \lambda_{u_{\nu',\iota'}}(s')^*$ 
of $\KK[T_1,\ldots,T_n]$ satisfy
$$ 
\bar \lambda_{u_{\nu',\iota'}}(s')^* 
\circ 
\bar \lambda_{u_{\nu,\iota}}(s)^*
\ = \ 
\exp(s'\delta_{u_{\nu',\iota'}} + s\delta_{u_{\nu,\iota}}).
$$
In particular, $\bar \lambda_{u_{\nu,\iota}}(s)^*$ 
and 
$\bar \lambda_{u_{\nu',\iota'}}(s')^*$  
as well as the associated automorphisms 
$\bar \lambda_{u_{\nu,\iota}}(s)$ and 
$\bar \lambda_{u_{\nu',\iota'}}(s')$ of $\bar Z$ commute.
\end{enumerate}
\end{lemma}

\begin{proof}
For~(i), we need 
that $u_{\nu,\iota}$ evaluates to $-1$ 
on $v_{\iota c_\iota}$ and  is non-negative 
on all other columns of~$P$.
By Definition~\ref{def:Pdemroot},
the latter is clear for all $v_k$ 
and $v_{ij}$ with $(i,j) \ne (i_1,c_{i_1})$
or $i \ne \iota$.
For the open cases, we compute
$$ 
\bangle{u_{\nu,\iota},v_{i_1c_{i_1}}} 
\ = \ 
\nu \bangle{u,v_{i_1c_{i_1}}} + l_{i_1c_{i_1}}
\ = \ 
l_{i_1c_{i_1}} - \nu,
$$
$$ 
\bangle{u_{\nu,\iota},v_{\iota j}}
\ = \ 
\nu \bangle{u,v_{\iota j}} - l_{\iota j}
\  \ 
\begin{cases}
\ge \ (\nu-1) l_{\iota j}, & j \ne c_\iota,
\\
= \ -1, & j = c_\iota.
\end{cases}
$$

We turn to~(ii). 
We infer directly from 
Construction~\ref{constr:toriclnd}
that $\delta_{u_{\nu,\iota}}$
annihilates all variables 
$T_{ij}$, $S_k$ except $T_{\iota c_\iota}$
and satisfies
$$ 
\delta_{u_{\nu,\iota}} (T_{\iota c_\iota})
\ = \ 
T_{\iota c_\iota} 
\prod T_{ij}^{ \bangle {u_{\nu,\iota}, v_{ij}} }
\prod S_k^{ \bangle {u_{\nu,\iota}, v_k} }
\ = \ 
f_{u,\nu,\iota} T_{i_1c_{i_1}}^{l_{i_1c_{i_1}-\nu}},
$$
where by the compution proving~(i)
the monomial $f_{u,\nu,\iota}$ 
depends neither on $T_{i_1,c_{i_1}}$ nor on 
$T_{\iota,c_\iota}$ 
and the $T_{ic_i}$ with $i \ne i_0,i_1,\iota$
have by Definition~\ref{def:Pdemroot} 
the exponent
$$ 
\bangle{u_{\nu,\iota}, v_{ic_i}} 
\ = \ 
\nu \bangle{u, v_{ic_i}} 
\ = \ 
0.
$$
We show~(iii).
From~(ii) we infer that 
$\delta_{u_{\nu',\iota'}} \delta_{u_{\nu,\iota}}$
as well as $\delta_{u_{\nu,\iota}} \delta_{u_{\nu',\iota'}}$
annihilate all variables $T_{ij}$ and $S_k$;
use Lemma~\ref{lem:deriv}~(ii) and~(iii).
Thus, Lemma~\ref{lem:deriv}~(v) gives
$$
\delta_{u_{\nu',\iota'}} \delta_{u_{\nu,\iota}} 
\ = \
\delta_{u_{\nu,\iota}} \delta_{u_{\nu',\iota'}}
\ = \
0.
$$
To obtain~(iv), first note that due to~(iii),
the linear endomorphisms $s\delta_{u_{\nu,\iota}}$ 
and $s'\delta_{u_{\nu',\iota'}}$ of 
$\KK[T_1,\ldots,T_n]$ commute.
Thus, the assertion follows from 
the definition 
$\bar \lambda_\delta(s)^* := \exp(s\delta)$ 
and the homomorphism property of the 
exponential series.
\end{proof}

\begin{theorem}
\label{thm:restrrootauts}
Let $X = X(A,P,\Sigma)$ be as in 
Construction~\ref{constr:RAP} with
$\Sigma$ complete and 
$\kappa = (u,i_0,i_1,C)$ a 
horizontal Demazure $P$-root.
For $s \in \KK$ set
$$ 
\alpha(s,\nu,\iota)
\ := \
\beta_\iota{l_{i_1 c_{i_1}} \choose \nu} s^\nu,
\qquad 
\iota = 0, \ldots, r,
\quad
\iota \ne i_0,i_1,
\quad 
\nu = 1, \ldots, l_{i_1 c_{i_1}},
$$
where $\beta$ is the unique vector 
in the row space of $A$ with 
$\beta_{i_0} = 0$ and $\beta_{i_1} = 1$.
Consider the linear forms $u_{\nu,\iota}$
from Construction~\ref{constr:uvi} and
the automorphisms  
$$ 
\varphi_{u}(s) 
\ := \
\prod_{\iota \ne i_0,i_1}
\prod_{\nu = 1}^{l_{i_1 c_{i_1}}} \lambda_{u_{\nu,\iota}} 
(\alpha(s,\nu,\iota))
\ \in \ 
\Aut(Z).
$$
Then the automorphism $\lambda_\kappa(s)$ of $X$ 
can be presented as the restriction of an automorphism 
of $Z$ as follows:
$$ 
\lambda_\kappa(s) \ = \   \lambda_{u}(s) \circ \varphi_u(s)  \vert_X.
$$
More explicitly, the comorphism satisfies
$\bar \lambda_\kappa(s)^*(S_k) = S_k$ and 
$\bar \lambda_\kappa(s)^* (T_{ij}) = T_{ij}$, 
whenever $i = i_0$ or $j \ne c_i$ and 
in the remaining cases
\begin{eqnarray*}
\bar \lambda_\kappa(s)^*(T_{i_1 c_{i_1}}) 
& = &
T_{i_1 c_{i_1}} 
+ 
s\delta_u(T_{i_1 c_{i_1}}),
\\
\bar \lambda_\kappa(s)^*(T_{\iota c_{\iota}}) 
& = &
T_{\iota c_{\iota}}
+
\sum_{\nu = 1}^{l_{i_1 c_{i_1}}} \alpha(s,\nu,\iota) 
\delta_{u_{\nu,\iota}}(T_{\iota c_{\iota}}).
\end{eqnarray*}
\end{theorem}

\begin{proof}
By definition, the comorphism of $\bar \lambda_\kappa(s)$ 
equals $\exp(s\delta_\kappa)$. 
In a first step, we compute the powers $\delta_\kappa^\nu$
occuring in the exponential series.
We will make repeated use of the fact
$$ 
\delta_{u_{\nu,\iota}} (T_{\iota c_\iota})
\ = \ 
T_{\iota c_\iota} h^{\nu u} 
\frac{T_{i_1}^{l_{i_1}}}{T_\iota^{l_\iota}}
\ = \
f_{u,\nu,\iota} T_{i_1c_{i_1}}^{l_{i_1c_{i_1}-\nu}},
$$
where $h^{\nu u}$ is as in 
Construction~\ref{constr:DEMLND}
and  $f_{u,\nu,\iota}$ is a monomial in 
the variables~$T_{ij}$ and~$S_k$ but not depending on 
any~$T_{ic_i}$ with $i \ne i_0$;
see Lemma~\ref{lem:linformsunuiota}~(ii).
Now, recall from Construction~\ref{constr:DEMLND}
that, apart from the $T_{i c_i}$,
all variables $T_{ij}$ and $S_k$ are annihilated 
by $\delta_\kappa$. 
Moreover, we have 
$$ 
\delta_\kappa(T_{i_0c_{i_0}})  
=  
\beta_{i_0}  
=  
0,
\qquad
\delta_\kappa (T_{i_1c_{i_1}}) 
= 
\beta_{i_1} \frac{h ^u}{h^\zeta} 
\prod_{i \ne i_0, i_1} \frac{\partial T_i^{l_i}}{\partial T_{ic_i}} 
= 
T_{i_1 c_{i1}} h^u 
=  
\delta_u(T_{i_1c_{i_1}}).
$$
Since $\delta_u(T_{i_1c_{i_1}})$ 
does not depend 
on any $T_{ic_i}$ with $i \ne i_0$, 
we conclude 
$\delta_\kappa^2(T_{i_1c_{i_1}}) = 0$.
Finally, for $\iota \ne i_0,i_1$, 
Construction~\ref{constr:DEMLND} 
and the above formula for $\nu = 1$
give us
$$ 
\delta_\kappa(T_{\iota c_\iota})
\ = \ 
\beta_\iota \frac{h ^u}{h^\zeta}
\prod_{i \ne i_0,\iota} \frac{\partial T_i^{l_i}}{\partial T_{ic_i}}
\ = \ 
\beta_\iota l_{i_1c_{i_1}} T_{\iota c_{\iota}} 
h^u \frac{T_{i_1}^{l_{i_1}}}{T_\iota^{l_\iota}}
\ = \ 
\beta_\iota l_{i_1c_{i_1}} \delta_{u_{1,\iota}}(T_{\iota c_\iota}).
$$
To evaluate higher powers of $\delta_\kappa$, 
we use the representation
$\delta_{u_{\nu,\iota}}(T_{\iota c_\iota}) = 
f_\nu T_{i_1c_{i_1}}^{l_{i_1c_{i_1}} -\nu}$
given above.
Applying the Leibniz rule yields
$$ 
\delta_{\kappa}(\delta_{u_{\nu,\iota}}(T_{\iota c_\iota}))
\  = \
(l_{i_1c_{i_1}} - \nu) \delta_{u_{\nu+1,\iota}}(T_{\iota c_\iota}).  
$$
Putting things together, we arrive at 
$$ 
\delta_\kappa^\nu (T_{\iota c_\iota})
\ = \ 
\beta_\iota \frac{l_{i_1c_{i_1}} !}{(l_{i_1c_{i_1}}-\nu) !} 
\delta_{u_{\nu,\iota}}(T_{\iota,c_\iota}).
$$
In the next step, we compute 
the values of 
$\bar \lambda_\kappa(s)^* = \exp(s\delta_\kappa)$
on the generators $T_{ij}$ and $S_k$.
From above, we infer 
$\bar \lambda_\kappa(s)^* (T_{ij}) = T_{ij}$, 
whenever $i = i_0$ or $j \ne c_i$.
Moreover, we have
$$ 
\bar \lambda_\kappa^*(s)(T_{i_1 c_{i_1}}) 
\ = \
T_{i_1 c_{i_1}} + s \delta_u(T_{i_1 c_{i1}}).
$$
Finally, for $\iota \ne i_0,i_1$,
plugging the above representations of 
the $\delta_\kappa^\nu (T_{\iota c_\iota})$
into the exponentional series
gives the remaining statements on 
comorphisms:
$$
\bar \lambda_\kappa^*(s)(T_{\iota c_{\iota}}) 
\ = \
T_{\iota c_{\iota}} + 
\sum_{\nu = 1}^{l_{i_1 c_{i_1}}} 
\beta_\iota {l_{i_1 c_{i_1}} \choose \nu} s^\nu
\delta_{u_{\nu,\iota}}(T_{\iota c_{\iota}})
\ = \ 
T_{\iota c_{\iota}}
+
\sum_{\nu = 1}^{l_{i_1 c_{i_1}}} \alpha(s,\nu,\iota) 
\delta_{u_{\nu,\iota}}(T_{\iota c_{\iota}}).
$$
We turn to
$\lambda_\kappa(s) = \lambda_u(s) \circ \varphi_u(s) \vert_X$.
We verify the corresponding identity on 
$\bar \lambda_\kappa$ and 
$\bar \lambda_u  \circ \bar \varphi_u$
by comparing the comorphisms.
First recall from Construction~\ref{constr:toriclnd}
that we have 
$$
\bar \lambda_u^*(s)
\ = \
\id  + s \delta_u,
$$
where $\delta_u$ annihilates all variables $T_{ij}$ and $S_k$ 
except $T_{i_1c_{i1}}$.
Next note that $\varphi_u(s)$ doesn't depend 
on the order of composition
due to Lemma~\ref{lem:deriv}~(v).
Moreover, Lemma~\ref{lem:deriv}~(iv)
allows to compute
$$ 
\bar \varphi_{u}(s)^*
\ = \
\prod_{\iota \ne i_0,i_1}
\prod_{\nu = 1}^{l_{i_1 c_{i_1}}} 
\bar \lambda_{u_{\nu,\iota}} 
(\alpha(s,\nu,\iota))^*
\ = \ 
\id
+
\sum_{\iota \ne i_0,i_1}
\sum_{\nu = 1}^{l_{i_1 c_{i_1}}} 
\alpha(s,\nu,\iota) 
\delta_{u_{\nu,\iota}}.
$$
Now, we explicitly evaluate 
$\bar \varphi_{u}(s)^* \circ \bar \lambda_u(s)^*$
on the generators $T_{ij}$ and~$S_k$ of 
the polynomial ring $\KK[T_{ij},S_k]$.
Obviously, we have 
$$ 
\bar \varphi_{u}(s)^* \circ \bar \lambda_u(s)^* (S_k) 
\ = \ 
\bar \varphi_{u}(s)^* (S_k) 
\ = \ 
S_k,
\quad k = 1, \ldots, m,
$$
$$
\bar \varphi_{u}(s)^* \circ \bar \lambda_u(s)^* (T_{ij}) 
\ = \ 
\bar \varphi_{u}(s)^* (T_{ij}) 
\ = \ 
T_{ij},
\quad
i = i_0 \text{ or } j \ne c_i.
$$
Using $\delta_u(T_{i_1c_{i_1}}) = h^uT_{i_1c_{i_1}}$, 
where the latter monomial doesn't depend on 
any $T_{ic_i}$ with $i \ne i_0$, we compute
$$ 
\bar \varphi_{u}(s)^* \circ \bar \lambda_u(s)^* (T_{i_1c_{i_1}}) 
= 
\bar \varphi_{u}(s)^* (T_{i_1c_{i_1}}) 
+ 
\bar \varphi_{u}(s)^*(h^uT_{i_1c_{i_1}}) 
= 
T_{i_1}c_{i_1} + \delta_u(T_{i_1c_{i_1}}).
$$
Finally, for any $\iota \ne i_0,i_1$, we obtain
$$
\bar \varphi_{u}(s)^* \circ \bar \lambda_u(s)^* (T_{\iota c_\iota}) 
\ = \ 
\bar \varphi_{u}(s)^* (T_{\iota c_\iota}) 
\ = \ 
T_{\iota c_\iota}
+
\sum_{\nu = 1}^{l_{i_1 c_{i_1}}} 
\alpha(s,\nu,\iota) 
\delta_{u_{\nu,\iota}}(T_{\iota c_\iota}).
$$
Thus, comparing with the previously obtained 
values of $\bar \lambda_\kappa(s)^*$ on the generators, 
we arrive at the identity 
$\bar \lambda_\kappa(s)^* = 
\bar \varphi_{u}(s)^* \circ \bar \lambda_u(s)^*$
of comorphisms,
which in turn induces the desired equation
$\lambda_\kappa(s) =  \lambda_u(s) \circ \varphi_u(s)$
on $Z$ and hence $X$.
\end{proof}

%% file: kstar-surfaces.tex
\section{Rational projective $\KK^*$-surfaces}
\label{sec:ratprojkstarsurf}

We will use the approach provided by 
Constructions~\ref{constr:RAP}
and~\ref{constr:RAPdown}
producing all rational projective
varieties with torus action of complexity
one as $X = X(A,P,\Sigma)$.
Recall that the defining
$(r+s) \times (n+m)$ block matrix $P$
is of the form
$$
P
\ = \
\left[
\begin{array}{cc}
L & 0
\\
d & d'
\end{array}      
\right]
\ = \ 
[v_{01}, \ldots, v_{0n_0},
\ldots,
v_{r1}, \ldots, v_{rn_r},
v_1,\ldots, v_m],
$$
where the columns $v_{ij}$ and $v_k$ are
pairwise distinct primitive integral
vectors generating $\QQ^{r+s}$ as a vector
space.
In the case of a $\KK^*$-surface $X$,
several aspects simplify.
First, we have $s=1$. Thus,
the lower part $[d,d']$ of $P$ is just
one row and $m \le 2$ holds.
Observe
$$
v_{0j} = (-l_{0j}, \ldots, -l_{0j},d_{0j}),
\qquad
v_{ij} = (0 , \ldots, 0, l_{ij}, 0 \ldots, 0, d_{ij} ),
\quad i = 1, \ldots, r,
$$   
where $l_{ij}$ sits at the $i$-th place for
$i = 1, \ldots, r$ and we always have
$\gcd(l_{ij},d_{ij}) = 1$.
Moreover, we arrange $P$ to be
\emph{slope ordered}, that means that
for each $0 \le i \le r$, we order
the block $v_{i1}, \ldots, v_{in_i}$ of
columns in such a way that 
$$
m_{i1} \ > \ \ldots \ > \ m_{ini},
\qquad
\text{where }
m_{ij} \ := \ \frac{d_{ij}}{l_{ij}}.
$$
Finally, in the surface case the defining
fan $\Sigma$ of the ambient toric variety~$Z$
is basically unique and needs no extra specification.
More precisely, the rays of $\Sigma$ are the
cones over the columns of $P$ and we always have
the maximal cones
$$
\tau_{ij}
\ := \
\cone(v_{ij}, v_{ij+1})
\ \in \ \Sigma,
\qquad
i = 0, \ldots, r, \ j = 1, \ldots, n_i-1.
$$
Writing $v^+ := v_1 = (0, \ldots,0,1)$ and
$v^- := v_2 = (0, \ldots,0,-1)$ for the
columns of~$P$ that arise for $m = 1,2$,
the collection of maximal cones of
$\Sigma$ is complemented depending on
the value of $m$ as follows

\def\pp{
\begin{tikzpicture}[scale=.5]
%Referenzpunkte
\coordinate (v) at (0,0);
\coordinate (v+) at (0,.9);
\coordinate (v-) at (0,-.8);
\coordinate (v01) at (-.75,.75);
\coordinate (v02) at (-1,0);
\coordinate (v0n) at (-.75,-.5);
\coordinate (v11) at (1,.5);
\coordinate (s11) at ($(v)!.7!(v11)$); 
\coordinate (v12) at (1.1,0);
\coordinate (v1n) at (1,-.75);
\coordinate (s1n) at ($(v)!.6!(v1n)$);
\coordinate (v21) at (.5,.85);
\coordinate (s21) at ($(v)!.71!(v21)$);  
\coordinate (v2n) at (.5,-.1);
\coordinate (s2n) at ($(v)!1.5!(v2n)$);  
\coordinate (r2n) at ($(v)!1.2!(v2n)$);  
% Kegel
\fill[gray!50,opacity=0.3] (v+)--(v01)--(v02)--(v0n)--(v-);
\fill[gray!50,opacity=0.3] (v+)--(s11)--(s1n)--(v-);
\fill[gray!50,opacity=0.3] (v+)--(v21)--(s2n)--(v-);
% Strahlen
\draw[] (v)--(v+);
\draw[] (v)--(v-);
\draw[] (v)--(v01);
\draw[] (v)--(v02);
\draw[] (v)--(v0n);
\draw[] (v)--(s11);
\draw[] (v)--(s1n);
\draw[dotted] (v)--(s21);
\draw[] (s21)--(v21);
\draw[dotted] (v)--(r2n);
\draw[] (r2n)--(s2n);
\end{tikzpicture}   
}

\def\pe{
\begin{tikzpicture}[scale=.5]
\coordinate (v) at (0,0);
\coordinate (v+) at (0,.9);
\coordinate (v-) at (0,-.8);
\coordinate (v01) at (-.75,.75);
\coordinate (v02) at (-1,0);
\coordinate (v0n) at (-.75,-.5);
\coordinate (v11) at (1,.5);
\coordinate (s11) at ($(v)!.7!(v11)$); 
\coordinate (v12) at (1.1,0);
\coordinate (v1n) at (1,-.75);
\coordinate (s1n) at ($(v)!.6!(v1n)$);
\coordinate (v21) at (.5,.85);
\coordinate (s21) at ($(v)!.71!(v21)$);  
\coordinate (v2n) at (.5,-.1);
\coordinate (s2n) at ($(v)!1.5!(v2n)$);  
\coordinate (r2n) at ($(v)!1.2!(v2n)$);  
% Kegel
\fill[gray!50,opacity=0.3] (v)--(v+)--(v21)--(s2n)--cycle;
\fill[gray!50,opacity=0.3] (v)--(v0n)--(s1n)--(s2n)--cycle;
\fill[gray!50,opacity=0.3] (v)--(v+)--(v01)--(v02)--(v0n)--cycle;
\fill[gray!50,opacity=0.3] (v)--(v+)--(s11)--(s1n)--cycle;
% Strahlen
\draw[] (v)--(v+);
\draw[] (v)--(v01);
\draw[] (v)--(v02);
\draw[] (v)--(v0n);
\draw[] (v)--(s11);
\draw[] (v)--(s1n);
\draw[dotted] (v)--(s21);
\draw[] (s21)--(v21);
\draw[dotted] (v)--(r2n);
\draw[] (r2n)--(s2n);
\draw[] (v0n)--(s1n)--(s2n);
\end{tikzpicture}   
}

\def\ep{
\begin{tikzpicture}[scale=.5]
%Referenzpunkte
\coordinate (v) at (0,0);
\coordinate (v+) at (0,.9);
\coordinate (v-) at (0,-.8);
\coordinate (v01) at (-.75,.75);
\coordinate (v02) at (-1,0);
\coordinate (v0n) at (-.75,-.5);
\coordinate (v11) at (1,.5);
\coordinate (s11) at ($(v)!.7!(v11)$); 
\coordinate (v12) at (1.1,0);
\coordinate (v1n) at (1,-.75);
\coordinate (s1n) at ($(v)!.6!(v1n)$);
\coordinate (v21) at (.5,.85);
\coordinate (s21) at ($(v)!.54!(v21)$);  
\coordinate (v2n) at (.5,-.1);
\coordinate (s2n) at ($(v)!1.5!(v2n)$);
\coordinate (r2n) at ($(v)!1.2!(v2n)$);  
% Kegel
\fill[gray!50,opacity=0.3] (v)--(v21)--(s2n)--(v-);
\fill[gray!50,opacity=0.3] (v)--(v01)--(v02)--(v0n)--(v-);
\fill[gray!50,opacity=0.3] (v)--(s11)--(s1n)--(v-);
\fill[gray!50,opacity=0.3] (v)--(v01)--(s11)--cycle;
\fill[gray!50,opacity=0.3] (v)--(s11)--(v21)--cycle;
\fill[gray!50,opacity=0.3] (v)--(v01)--(v21)--cycle;
% Strahlen
\draw[] (v)--(v-);
\draw[] (v)--(v01);
\draw[] (v)--(v02);
\draw[] (v)--(v0n);
\draw[] (v)--(s11);
\draw[] (v)--(s1n);
\draw[dotted] (v)--(s21);
\draw[] (s21)--(v21);
\draw[dotted] (v)--(r2n);
\draw[] (r2n)--(s2n);
\draw[] (v01)--(s11)--(v21)--(v01);
\end{tikzpicture}   
}

\def\ee{
\begin{tikzpicture}[scale=.5]
%Referenzpunkte
\coordinate (v) at (0,0);
\coordinate (v+) at (0,.9);
\coordinate (v-) at (0,-.8);
\coordinate (v01) at (-.75,.75);
\coordinate (v02) at (-1,0);
\coordinate (v0n) at (-.75,-.5);
\coordinate (v11) at (1,.5);
\coordinate (s11) at ($(v)!.7!(v11)$); 
\coordinate (v12) at (1.1,0);
\coordinate (v1n) at (1,-.75);
\coordinate (s1n) at ($(v)!.6!(v1n)$);
\coordinate (v21) at (.5,.85);
\coordinate (s21) at ($(v)!.54!(v21)$);  
\coordinate (v2n) at (.5,-.1);
\coordinate (s2n) at ($(v)!1.5!(v2n)$);
\coordinate (r2n) at ($(v)!1.2!(v2n)$);  
% Kegel
\fill[gray!50,opacity=0.3] (v)--(v21)--(s2n)--cycle;
\fill[gray!50,opacity=0.3] (v)--(v01)--(v02)--(v0n)--cycle;
\fill[gray!50,opacity=0.3] (v)--(s11)--(s1n)--cycle;
\fill[gray!50,opacity=0.3] (v)--(v01)--(s11)--cycle;
\fill[gray!50,opacity=0.3] (v)--(s11)--(v21)--cycle;
\fill[gray!50,opacity=0.3] (v)--(v01)--(v21)--cycle;
\fill[gray!50,opacity=0.3] (v)--(v0n)--(s1n)--(s2n)--cycle;
% Strahlen
\draw[] (v)--(v01);
\draw[] (v)--(v02);
\draw[] (v)--(v0n);
\draw[] (v)--(s11);
\draw[] (v)--(s1n);
\draw[dotted] (v)--(s21);
\draw[] (s21)--(v21);
\draw[dotted] (v)--(r2n);
\draw[] (r2n)--(s2n);
\draw[] (v0n)--(s1n)--(s2n);
\draw[] (v01)--(s11)--(v21)--(v01);
\end{tikzpicture}   
}

$$
\begin{array}{lll}
\qquad &
\begin{array}{lllcl}
m = 2: \quad
&  
\text{(p-p)}
&
\tau_i^+ & := & \cone(v^+,v_{i1})
\\                                    
&& \tau_i^- & := & \cone(v_{in_i},v^-)
\end{array}
&
\vcenter{ \pp }
\\[15pt]
\qquad &
\begin{array}{lllcl}
m = 1: \quad  & \text{(p-e)} 
& \tau_i^+ & := & \cone(v^+,v_{i1})
\\        
&& \sigma^- & :=  & \cone(v_{0n_0}, \ldots, v_{rn_r})
\end{array}
&                    
\vcenter{ \pe }
\\[15pt]
\qquad &
\begin{array}{lllcl}
\hphantom{m = 1: \quad}  & \text{(e-p)}
&\sigma^+ & :=  & \cone(v_{01}, \ldots, v_{r1})
\\
&& \tau_i^- & := & \cone(v_{in_i},v^-)
\end{array}
&
\vcenter{ \ep }
\\[15pt]
\qquad &
\begin{array}{lllcl}
m = 0: \quad  & \text{(e-e)}
& \sigma^+ & :=  & \cone(v_{01}, \ldots, v_{r1})
\\
&& \sigma^- & :=  & \cone(v_{0n_0}, \ldots, v_{rn_r})
\end{array}
&
\vcenter{ \ee }
\\[5pt]
\qquad &
\end{array}
$$

\noindent
In particular, the $\KK^*$-surfaces delivered by
Construction~\ref{constr:RAPdown}
only depend on the matrices~$A$ and~$P$,
which allows us to denote them 
as $X = X(A,P)$.
The $\KK^*$-action on $X$ is given on 
the torus $\TT^{r+1} \subseteq Z$ by
$t \cdot z = (z_1, \ldots, z_r, tz_{r+1})$.

\begin{remark}
Let $X = X(A,P)$ be a $\KK^*$-surface
as above.
Then the fan $\Sigma$ of the ambient toric
variety $Z$ of $X$ reflects the geometry
of the $\KK^*$-action on $X$, as outlined in
the introduction, in the following way.
\begin{enumerate}
\item
If $P$ has a column~$v^+$ or $v^-$,
then the toric prime divisor 
on $Z$ corresponding to 
$\varrho^+ = \cone(v^+)$,
or $\varrho^- = \cone(v^-)$,
cuts out a parabolic fixed point
curve forming source or sink:
$$
D^+ \ = \ (B^+ \cap X) \cup \{x_0^+\} \cup \ldots \cup \{x_r^+\},
$$
$$
D^- \ = \ (B^- \cap X) \cup \{x_0^-\} \cup \ldots \cup \{x_r^-\}.
$$
Here $B^+, B^- \subseteq Z$ denote 
the toric orbits corresponding to 
$\varrho^+, \varrho^- \in \Sigma$ 
and $x_i^+ \in B_i^+ \cap X$, 
$x_i^- \in B_i^- \cap X$
are the unique points    
in the intersections 
with the toric orbits
$B_i^+,  B_i^- \subseteq Z$ 
corresponding to $\tau_i^+, \tau_i^- \in \Sigma$.
\item
If we have a cone $\sigma^+$, resp.~$\sigma^-$,
in $\Sigma$, then the associated toric
fixed point~$x^+$, resp.~$x^-$, of $Z$ is an
elliptic fixed point of the $\KK^*$-action on~$X$
forming the source, resp. the sink.
\item
The toric prime divisor of $Z$
corresponding to the ray
$\varrho_{ij} = \cone(v_{ij})$ of~$\Sigma$
cuts out the closure $D_{ij} \subseteq X$
of an orbit $\KK^* \cdot x_{ij} \subseteq X$.
If $P$ is irredundant, then the arms of $X$
are precisely
$$
\mathcal{A}_i
\ = \ 
D_{i1} \cup \ldots \cup D_{in_i},
\qquad
i \ = \ 0, \ldots, r.
$$
The order of the isotropy group $\KK^*_{x_{ij}}$
equals $l_{ij}$.
The hyperbolic fixed point forming
$D_{ij} \cap D_{ij+1}$ is the point
cut out from $X$ by the toric orbit
of $Z$ corresponding to the cone
$\tau_{ij} \in \Sigma$.
\end{enumerate}
\end{remark}

\begin{definition}
For any rational projective $\KK^*$-surface
$X = X(A,P)$, we define the numbers 
$$ 
l^+ \ := \ l_{01} \cdots l_{r1},
\qquad
m^+ \ := \ m_{01} + \ldots + m_{r1},
$$
$$
l^-  \ := \ l_{0n_0} \cdots l_{rn_r},
\qquad
m^- \ := \ m_{0n_0} + \ldots + m_{rn_r}.
$$
\end{definition}

\begin{remark}
\label{rem:m+m-}
Let $X = X(A,P)$.
We have $l^+m^+ \in \ZZ$ and if
there is an elliptic fixed 
point $x^+ \in X$, then
$$
\det(\sigma^+)
\ := \
\det(v_{01}, \ldots, v_{r1})
\ = \ 
l^+m^+
\ > \
0.
$$
Similarly, $l^-m^- \in \ZZ$ holds
and if there is an elliptic 
fixed point $x^- \in X$, then
we obtain
$$
\det(\sigma^-)
\ := \
\det(v_{0n_0}, \ldots, v_{rn_r})
\ = \ 
l^-m^-
\ < \
0.
$$
\end{remark}

Rational projective $\KK^*$-surfaces 
turn out to be always $\QQ$-factorial.
That means in particular that intersection
numbers are well defined.
Let us recall~\cite[Cor.~5.4.4.2]{ArDeHaLa}.

\begin{remark}
\label{rem:selfintKstar}
For $X = X(A,P)$,
the self intersection numbers of
the orbit closures $D_{ij} \subseteq X$ 
and possible parabolic fixed point curves
$D^+, D^- \subseteq X$ are given by
$$
\begin{array}{lclr}
D_{i1}^2
& = &
\begin{cases}
\frac{1}{l_{i1}^2}\bigl(\frac{1}{m^+}-\frac{1}{m^-}\bigr),
&
\text{(e-e),}
\\
0
&
\text{(p-p),}
\\
\frac{1}{l_{i1}^2 m^+},
& 
\text{(e-p),}
\\
\frac{-1}{l_{i1}^2 m^-}, 
& 
\text{(p-e),}
\end{cases}
&
\text{for } n_i = 1,                            
\\
\\                            
D_{ij}^2
& = &
\begin{cases}
\frac{1}{l_{ij}^2}
\bigl(\frac{1}{m^+} - \frac{1}{m_{ij}-m_{ij+1}}\bigr),
&
j = 1 \text{ and (e-e) or (e-p),}
\\
\frac{-1}{l_{ij}^2(m_{ij}-m_{ij+1})},
&
j = 1 \text{ and (p-p) or (p-e),}
\\
\frac{-(m_{ij-1} - m_{ij+1})}{l_{ij}^2(m_{ij-1} - m_{ij})(m_{ij} - m_{ij+1})},
& 
1<j<n_i,
\\
\frac{-1}{l_{ij}^2(m_{ij-1}-m_{ij})},
&
j = n_i \text{ and (p-p) or (e-p),}
\\
\frac{1}{l_{ij}^2}
\bigl(-\frac{1}{m^-} + \frac{1}{m_{ij}-m_{ij-1}}\bigr),
&
j = n_i \text{ and (e-e) or (p-e),}
\end{cases}
&
\text{for } n_i > 1,        
\\
\\  
(D^+)^2
& = &
-m^+,
&      
\\[5pt]
%\\  
(D^-)^2
& = &
m^-.
&
\\    
\end{array}
$$
\end{remark}

Recall that an irreducible curve $D$ on a normal projective 
surface $X$ is called~\emph{contractible} if there is a 
morphism $\pi \colon X \to X'$ onto a normal surface $X'$ 
mapping $D$ to a point $x' \in X'$ and inducing an 
isomorphism from $X \setminus D$ onto $X' \setminus \{x'\}$.

\begin{remark}
\label{rem:parselfint}
Consider $X = X(A,P)$. Then, provided that they are
present, the vectors $v^+$ and $v^-$ satisfy
the identities
$$
l_{01}^{-1}v_{01} + \ldots + l_{r1}^{-1}v_{r1} \ = \ m^+ v^+,
\qquad
l_{0n_0}^{-1}v_{0n_0} + \ldots + l_{rn_r}^{-1}v_{rn_r} \ = \ - m^- v^-.
$$
Combining this with Remark~\ref{rem:selfintKstar},
we rediscover that $D^+$ and $D^-$ are contractible
if and only if they are of negative self intersection.
\end{remark}

More generally, contractibility of invariant curves 
on rational normal projective $\KK^*$-surfaces is
characterized as follows.

\begin{remark}
\label{rem:conctractibility}
Consider $X = X(A,P)$.
Given a column $v$ of $P$, let $D \subseteq X$
be the corresponding curve and $P'$ the matix 
obtained from $P$ by removing $v$. 
Then the following statements are equivalent.
\begin{enumerate}
\item
The curve $D \subseteq X$ is contractible.
\item
The matrices $A$ and $P'$ define a $\KK^*$-surface $X' = X(A,P')$.
\item
We have $D^2 < 0$ for the self intersection number. 
\end{enumerate} 
If one of these statements holds, then $D$ is contracted
by the $\KK^*$-equivariant morphism $X \to X'$ induced by the
map of fans $\Sigma \to \Sigma'$ and there is a unique 
cone $\sigma' \in \Sigma'$ containing~$v$ in its relative 
interior.
\end{remark}

We turn to singularities of $\KK^*$-surfaces
$X = X(A,P)$.
Note that due to normality of the surfaces,
every singularity is a fixed point.

\begin{definition}
\label{def:qsmooth}
Let $X$ be a rational projective
$\KK^*$-surface
and $p \colon \hat X \to X$ its 
characteristic space.
A point $x \in X$ is \emph{quasismooth} 
if $x = p(z)$ holds for a smooth point 
$z \in \hat X$.
\end{definition}

We characterize quasismoothness and
smoothness of parabolic, hyperpbolic
and elliptic fixed points in terms of
the defining matrix $P$ of $X = X(A,P)$.
All the statements are direct consequences of 
the general (quasi-)smoothness
criterion~\cite[Cor.~7.16]{HaHiWr}.

\begin{proposition}
\label{prop:smooth-par-li1}
Consider $X = X(A,P)$. 
Then we have the following 
statements on (quasi-)smoothness
of possible parabolic fixed points.
\begin{enumerate}
\item
All points of $B^+ \subseteq D^+$ 
are smooth and all points 
$x_i^+ \in D^+$ are quasismooth.
Moreover, $x_i^+ \in D^+$ is smooth
if and only if $l_{i1} = 1$ holds.
\item
All points of $B^- \subseteq D^-$ 
are smooth and all points 
$x_i^- \in D^-$ are quasismooth.
Moreover, $x_i^- \in D^-$ is smooth
if and only if $l_{in_i} = 1$ holds.
\end{enumerate}
\end{proposition}

\begin{proposition}
\label{prop:hyp-q-smooth}
Consider $X = X(A,P)$. 
Then every hyperbolic
fixed point of~$X$ is 
quasismooth.
Moreover, the hyperbolic
fixed point corresponding to 
$\tau_{ij} \in \Sigma$ is smooth
if and only if 
$l_{ij+1}d_{ij} - l_{ij}d_{ij+1} = 1$ 
holds.
\end{proposition}

\begin{proposition}
\label{prop:ell-q-smooth}
Assume that the $\KK^*$-surface $X = X(A,P)$ 
has an elliptic fixed point $x \in X$.
\begin{enumerate}
\item 
If $x = x^+$, then $x$ is quasismooth 
if and only if there are 
$0 \le \iota_0,\iota_1 \le r$
with $l_{i1} = 1$ for every $i \ne \iota_0, \iota_1$.
\item 
If $x = x^+$, then $x$ is smooth 
if and only if there are 
$0 \le \iota_0,\iota_1 \le r$
with $l_{i1} = 1$ for every $i \ne \iota_0, \iota_1$
and 
$$
\qquad
\det(\sigma^+) 
\ = \
l^+m^+
\ = \
l_{\iota_0 1}d_{\iota_1 1} + l_{\iota_1 1}d_{\iota_0 1}
\ = \ 
1.
$$
\item 
If $x = x^-$, then $x$ is quasismooth 
if and only if there are 
$0 \le \iota_0,\iota_1 \le r$
with $l_{in_i} = 1$ for every $i \ne \iota_0, \iota_1$.
\item 
If $x = x^-$, then $x$ is smooth 
if and only if there are 
$0 \le \iota_0,\iota_1 \le r$
with $l_{in_i} = 1$ for every $i \ne \iota_0, \iota_1$
and 
$$
\qquad
\det(\sigma^-) 
\ = \ 
l^-m^-
\ = \
l_{\iota_0 n_{\iota_0}}d_{\iota_1 n_{\iota_1}} + l_{\iota_1 n_{\iota_1}}d_{\iota_0 n_{\iota_0}}
\ = \ 
-1.
$$
\end{enumerate}
\end{proposition}

\begin{definition}
\label{def: leading indices}
Given an elliptic fixed point 
$x \in X = X(A,P)$, 
we call the numbers
$0 \le \iota_0,\iota_1 \le r$
from Proposition~\ref{prop:ell-q-smooth}
\emph{leading indices} 
for $x$.
\end{definition}

As a consequence
of Proposition~\ref{prop:ell-q-smooth}, 
we obtain the following characterization
of quasismoothness of
$\KK^*$-surface~$X$.
We say that a singularity
$x \in X$ is a \emph{toric
surface singularity}
if there is a $\KK^*$-invariant
open neighbourhood
$x \in U \subseteq X$
such that
$U$ is a toric surface.
Recall from~\cite{CoLiSc} that 
toric surface singularities are
quotients of $\KK^2$ by finite
cyclic groups.
 
\begin{corollary}
\label{cor:quasismoothchar}
A rational projective $\KK^*$-surface
is quasismooth if and only if   
it has at most toric surface
singularities.
\end{corollary}

\begin{proof}
We may assume that our $\KK^*$-surface
is given as $X = X(A,P)$.
By normality, any singular point of $X$
is a $\KK^*$-fixed point.
The parabolic and hyperbolic fixed
point are toric surfaces singularities
due to~\cite[Prop.~3.4.4.6]{ArDeHaLa}.

Thus, we are left with discussing
quasismooth elliptic fixed points.
It suffices to consider $x^- \in X$.
Let $0 \le \iota_0,\iota_1 \le r$
be leading indices for $x^-$.
The cone $\sigma^- \in \Sigma$ defines
affine open subsets
$$
Z^- \ \subseteq \ Z,
\qquad\qquad
X^- \ := \ X \cap Z^- \ \subseteq \ X.
$$ 
Recall that $x^-$ is the toric fixed
point of $Z^-$.
Then $X^-$ is the affine $\KK^*$-surface
given by the data $A$ and
$P^- = [v_{0n_0}, \ldots, v_{rn_r}]$
in the sense
of~\cite[Constr.~1.6 and Cor.~1.9]{HaWr}.
Consider the defining relations 
$$
g_{i_1,i_2,i_3}
\ := \
\det
\left[
\begin{array}{ccc}
T_{i_1n_{i_1}}^{l_{i_1n_{i_1}}}
&
T_{i_2n_{i_2}}^{l_{i_2n_{i_2}}}
&
T_{i_3n_{i_3}}^{l_{i_3n_{i_3}}}
\\
a_{i_1} & a_{i_2} & a_{i_3}
\end{array}
\right]
$$
of the Cox ring $\mathcal{R}(X^-)$ of $X^-$.
By Proposition~\ref{prop:ell-q-smooth}
the point $x^-$ is quasismooth
if and only if $l_{in_i} = 1$ for all 
$i \ne \iota_0,\iota_1$.
The latter is equivalent to the
fact that $\mathcal{R}(X^-)$
is a polynomial ring.
This in turn holds if and only if
$X^-$ is an affine toric surface.
\end{proof}

\begin{remark}
\label{rem:surfsingres}
Consider $X = X(A,P)$.
The \emph{canonical resolution} of
singularities $X'' \to X$
from~\cite[Constr.~5.4.3.2]{ArDeHaLa}
is obtained by the following
two step procedure.
\begin{enumerate}
\item
Enlarge $P$ to a matrix $P'$ by adding $e_{r+1}$ 
and $-e_{r+1}$, if not already present. 
Then the surface $X' := X(A,P')$ is quasismooth
and there is a canonical morphism $X' \to X$.
\item 
Let $P''$ be the slope ordered matrix having
the primitive generators of the regular
subdivision of $\Sigma(P')$ as its columns.
Then $X'' := X(A,P'')$ is smooth and 
there is a canonical morphism $X'' \to X'$.
\end{enumerate}
Contracting all $(-1)$-curves inside
the smooth locus that lie over singularities
of $X$ gives $X'' \to \tilde X \to X$,
where $\tilde X = X(A,\tilde P)$ is the
\emph{minimial resolution} of $X$.
\end{remark}

%% file: cont-frac.tex
\section{Self intersection numbers
and continued fractions}
\label{sec:cfrac}

This section presents some variations
on~\cite[Thm.~2.5]{OrWa3}~(iii) and~(iv)
on continued fractions over the numbers
$-D_{i1}^2, \ldots, -D_{n_i}^2$ given
by an arm of a smooth $\KK^*$-surface
with two parabolic fixed point curves.
Proposition~\ref{prop:par-fp-curve}
shows how to express the entries
$l_{ij}$ and $d_{ij}$ of $P$ for smooth
$X(A,P)$ of types (p-p), (p-e) and~(e-e)
via continued fractions over partial
arms.
For convenience, we present the full
proofs.

\begin{definition}
\label{def:adpated}
Consider the defining matrix $P$ of a smooth
rational projective $\KK^*$-surface $X(A,P)$.
By our assumptions, $P$ is irredundant and slope 
ordered.
\begin{enumerate}
\item
We call $P$ \emph{adapted to the source} if 
it satisfies
\begin{enumerate}
\item 
$-l_{i1} < d_{i1} \le 0$ for $i = 1, \ldots, r$, 
\item
$l_{01}, l_{11} \ge l_{21} \ge\ldots \ge l_{r1}$.
\vspace{3pt}
\end{enumerate} 
\item
We call $P$ \emph{adapted to the sink} if 
it satisfies
\begin{enumerate}
\item 
$0 \le d_{in_i} < l_{in_i}$ for $i = 1, \ldots, r$, 
\item
$l_{01}, l_{11} \ge l_{21} \ge \ldots \ge l_{rn_r}$.
\end{enumerate}
\end{enumerate}
\end{definition}

\begin{definition}
\label{def:lij-dij-beyond}
Consider a smooth rational projective 
$\KK^*$-surface $X = X(A,P)$
and the entries $l_{ij}$ and $d_{ij}$
of the defining matrix $P$.
\begin{enumerate}
\item  
Assume that $X$ has a parabolic fixed point 
curve $D^+$ and let $P$ be adapted to the 
source.
Set
$$
\qquad\qquad
\begin{array}{lll}
\text{for } i = 0, \ldots, r:
&
&
l_{i0} := 0,
\quad
d_{i 0} := 1,
\quad
D_{i0}^2 := (D^+)^2,
\\[.3em]
\text{for } i = 0:
&
&
l_{i(-1)} := -l_{1 1},
\quad 
d_{i(-1)} := d_{11},
\\[.3em]
\text{for } i = 1, \ldots, r:
&
&
l_{i(-1)} := -l_{0 1},
\quad
 d_{i(-1)} := d_{01}.
\end{array}
$$
\item
Assume that $X$ has an elliptic fixed point 
$x^-$ and let $P$ be adapted to the 
sink. Set
$$
\qquad\qquad
\begin{array}{lll}
\text{for } i = 0:
&
&
l_{0 n_0+1} := -l_{1 n_1},
\quad
d_{0 n_0+1} := d_{1 n_1},
\\[.3em]
\text{for } i = 1:
&
&
l_{1 n_1+1} := -l_{0 n_0},
\quad
d_{1 n_1+1} := d_{0 n_0},
\\[.3em]
\text{for } i = 2, \ldots, r:
&
&
l_{i n_i+1} := -l_{0 n_0} l_{1 n_1},
\quad
d_{i n_i+1}  :=  -1.
\end{array}
$$
\end{enumerate}
\end{definition}

\begin{lemma}
\label{lem:intnosDij}
Consider a smooth rational projective 
$\KK^*$-surface $X = X(A,P)$
and the curves $D_{ij}$ in the arms of $X$.
\begin{enumerate}
\item  
Assume that $X$ has a parabolic fixed point 
curve $D^+$ and let $P$ be adapted to the 
source. 
Then, for all $i = 0, \ldots, r$ 
and $j = 0, \ldots, n_i-1$,
we have
$$
\qquad
-l_{ij}D_{ij}^2 \ = \ l_{ij-1} + l_{ij+1},
\qquad
-d_{ij}D_{ij}^2 
\ = \ d_{ij-1} + d_{ij+1}.
$$
\item
Assume that $X$ has an elliptic fixed point 
curve $x^-$ and let $P$ be adapted to the 
sink. 
Then, for all $i = 0, \ldots, r$ 
and $j = 2, \ldots, n_i$, we have
$$
\qquad
-l_{ij}D_{ij}^2 \ = \ l_{ij-1} + l_{ij+1},
\qquad
-d_{ij}D_{ij}^2 \ = \ d_{ij-1} + d_{ij+1}.
$$
\end{enumerate}
\end{lemma}

\begin{proof}
The statements follow directly from the
computation of self intersection numbers 
of $X = X(A,P)$ in terms of the entries
of~$P$ given in Remark~\ref{rem:selfintKstar}.
\end{proof}

\begin{reminder}
Given any finite sequence $a_1,\ldots,a_k$ of
rational numbers, consider the process
$$
\mathrm{CF}_1(a_1) = a_1,
\quad
\mathrm{CF}_2(a_1,a_2) = a_1 - \frac{1}{a_2},
\quad
\mathrm{CF}_3(a_1,a_2,a_3) =
a_1 - \frac{1}{a_2-\frac{1}{a_3}}
\quad \ldots
$$
Provided there is no division by zero,
these numbers are called \emph{continued fractions}.
The formal definition runs inductively:
$$
\mathrm{CF}_1(a_1)
\ := \
a_1,
\qquad
\mathrm{CF}_k(a_1,\dots,a_k)
\ := \ 
a_1 
- \frac{1}{\mathrm{CF}_{k-1}(a_2,\dots,a_k)}.
$$
\end{reminder}

\begin{proposition}
\label{prop:par-fp-curve}
Consider a smooth rational projective 
$\KK^*$-surface $X = X(A,P)$
and the curves $D_{ij}$ in the arms of $X$.
\begin{enumerate}
\item  
Assume that $X$ has a parabolic fixed point curve $D^+$
and $P$ is adapted to the source.
Fix $0 \le i \le r$ and $1 \le j \le n_i$
and for $k=1,\ldots,j-1$ set 
$$
f_{ijk} 
\ := \ 
\mathrm{CF}_k(-D_{ij-k}^2,\dots,-D_{ij-1}^2).
$$
Then the entries $l_{ij}$ and $d_{ij}$ of the matrix $P$
can be expressed in terms of the above continued
fractions $f_{ijk}$ as
$$
\qquad\qquad\quad
l_{ij}
\ = \ 
\prod_{k=1}^{j-1} f_{ijk},
\qquad
d_{ij}
\ = \
\begin{cases}
\displaystyle
-\left((D^+)^2 f_{0 j j-1} + 1\right) \prod_{k=1}^{j-2} f_{0jk}
&
i = 0,
\\[1.5em]
\displaystyle
- \prod_{k=1}^{j-2} f_{ijk},
&
i \ne 0.
\end{cases}
$$
\item
Assume that $X$ has an elliptic fixed point $x^-$
and $P$ is adapted to the sink.
Fix $0 \le i \le r$ and  $1 \le j \le n_i$ and
for $k=1,\ldots,j-1$ set 
$$
h_{ijk}
\ := \
\mathrm{CF}_k(-D^2_{i n_i-j+k},\dots,-D^2_{i n_ i}).
$$
Then the entries $l_{i n_i-j}$ and $d_{i n_i-j}$
of the matrix $P$ can be expressed
in terms of the above continued fractions as
\begin{eqnarray*}
l_{n_i-j}
& = &
\left(\prod_{k=1}^{j} h_{ijk}\right) l_{i n_i}
-
\left(\prod_{k=1}^{j-1} h_{ijk}\right) l_{i n_i+1},
\\[1.5em]
d_{n_i-j}
& = &
\left(\prod_{k=1}^{j} h_{ijk}\right) d_{i n_i}
-
\left(\prod_{k=1}^{j-1} h_{ijk}\right) d_{i n_i+1}.
\end{eqnarray*}
\end{enumerate}
\end{proposition}

\begin{proof}
For both assertions, the proof 
relies on partially solving
tridigonal systems of 
linear equations of the following
form:
$$
\left[
\begin{array}{ccccc}
a_1 
& -1 
&  
& 
&
0
\\
-1 
& 
a_2 
& 
-1 
&  
&  
\\
& 
& 
\ddots 
&  
&  
\\
&  
& 
-1 
& 
a_{n-1}
& 
-1 
\\
0 
& 
& 
& 
-1 
& 
a_n 
\end{array}
\right]
\cdot 
\left[
\begin{array}{c}
x_{1}
\\
x_{2}
\\
\vdots
\\
x_{n-1}
\\
x_{n}
\end{array}
\right]
\quad = \quad
\left[
\begin{array}{c}
b_1
\\
0
\\
\vdots
\\
0
\\
b_n
\end{array}
\right]
$$
In~\cite[Thm.~15]{Ki}, the solutions are explicitly
computed via continued fractions in the entries.
In particular, with
$f_{k} \ := \ \mathrm{CF}_k(a_k,\dots,a_1)$,
it gives us
$$
b_1
\ = \
\left(\prod_{k=1}^n f_k \right) x_n
- \left(\prod_{k=1}^{n-1} f_k \right) b_n.
$$
We verify~(i).
Due to smoothness, Proposition~\ref{prop:smooth-par-li1}~(i) yields
$l_{i1} =1$.
Now, the relations among the $l_{ij}$ 
provided by Lemma~\ref{lem:intnosDij}~(i) can
be written as follows
$$
\left[
\begin{array}{ccccc}
-D_{i j-1}^2 
& -1 
&  
& 
&
0
\\
-1 
& 
-D_{i j-2}^2 
& 
-1 
&  
&  
\\
& 
& 
\ddots 
&  
&  
\\
&  
& 
-1 
& 
-D_{i 2}^2 
& 
-1 
\\
0 
& 
& 
& 
-1 
& 
-D_{i 1}^2 
\end{array}
\right]
\cdot 
\left[
\begin{array}{c}
l_{i j-1}
\\
l_{i j-2}
\\
\vdots
\\
l_{i2}
\\
l_{i1}
\end{array}
\right]
\quad = \quad
\left[
\begin{array}{c}
l_{i j}
\\
0
\\
\vdots
\\
0
\\
l_{i0}
\end{array}
\right]
$$
Thus, the above formula for $b_1$
gives the desired presentation of~$l_{ij}$.
With the $d_{ij}$, we proceed analogously.
In order to verify~(ii), look at
$$
\left[
\begin{array}{ccccc}
-D_{i n_i - j +1}^2 
& -1 
&  
& 
&
0
\\
-1 
& 
-D_{i n_i - j +2}^2 
& 
-1 
&  
&  
\\
& 
& 
\ddots 
&  
&  
\\
&  
& 
-1 
& 
-D_{i n_i -1}^2 
& 
-1 
\\
0 
& 
& 
& 
-1 
& 
-D_{i n_i}^2 
\end{array}
\right]
\cdot 
\left[
\begin{array}{c}
l_{i n_i -j + 1}
\\
l_{i n_i - j+2}
\\
\vdots
\\
l_{in_i-1}
\\
l_{in_i}
\end{array}
\right]
 = 
\left[
\begin{array}{c}
l_{i n_i - j}
\\
0
\\
\vdots
\\
0
\\
l_{i n_i +1}
\end{array}
\right]
$$
encoding the relations among the $l_{ij}$ 
from Lemma~\ref{lem:intnosDij}~(ii)
and apply the above presentation of $b_1$.
Again the $d_{ij}$ are settled analogously.
\end{proof}

\begin{corollary}
\label{cor:intno-cfrac}
Consider a smooth rational projective
$\KK^*$-surface $X = X(A,P)$.
\begin{enumerate}
\item
Assume that there is a fixed point curve $D^+ \subseteq X$
and that $P$ is adapted to the source. 
Fix any choice of indices $1 \le j_i \le n_i$,
where $i = 0, \ldots, r$.
Then we have
$$
\qquad
(D^+)^2
\ = \
- \sum_{i = 0}^r m_{i j_i}
- \sum_{i = 0}^r \mathrm{CF}_{j_i-1} (-D_{i1}^2,\ldots,-D_{i j_i-1}^2)^{-1}.
$$
\item
Assume that there is a fixed point $x^- \in X$
and that $P$ is adapted to the sink. 
Fix $0 \le j \le n_0-1$ and set 
$\bar{\sigma}^- := \cone(v_{0n_0-j},v_{1 n_1}, \ldots,v_{r n_r})$.
Then we have:
$$
\frac{l_{0 n_0-j}}{\det(\bar{\sigma}^-)}
\ = \
\mathrm{CF}_{j}(-D_{i n_i}^2,\ldots,-D_{i n_i-j+1}^2)^{-1}l_{1 n_1}.
$$
\end{enumerate}
\end{corollary}

\begin{proof}
We prove~(i).
Proposition~\ref{prop:par-fp-curve}~(i) allows us 
to express the slopes $m_{i j_i}$ in the 
following way:
$$
m_{0 j_0}
\ = \
-(D^+)^2 - f_{0 j_0 j_0 -1}^{-1},
\qquad\qquad
m_{i j_i}
\ = \
-f_{i j_i j_i -1}^{-1},
\quad
i = 1, \ldots, r.
$$
By the definition of the $f_{ijk}$,
this directly leads to the desired
representation of
the self intersection number:
$$
(D^+)^2
\ = \
-m_{0 j_0} - f_{0 j_0 j_0 -1}^{-1} 
\ = \
-\sum_{i = 0}^r m_{i j_i}
-
\sum_{i = 0}^r f_{i j_i j_i -1}^{-1}.
$$
We turn to~(i). Since $X$ is smooth,
Proposition~\ref{prop:ell-q-smooth}~(iv)
tells us $\det(\sigma^-)= -1$.
Thus, setting $h_0 := h_{0j1} \cdots h_{0jj}$,
we have
\begin{eqnarray*}
-1
& = &
\det(\sigma^-)
\\
& = &
l_{0 n_0}d_{1 n_1} + l_{1 n_1}d_{0 n_0}
\\
& = & 
(h_{0}^{-1} d_{0 n_0-j}  +  h_{0 j j}^{-1} d_{1 n_1})l_{1 n_1} 
 +
d_{1 n_1} (h_{0}^{-1} l_{0 n_0-j} - h_{0 j j}^{-1} l_{1 n_1})
\\
& = & 
h_{0}^{-1} (d_{0 n_0-j}l_{1 n_1} + d_{1 n_1}l_{0 n_0-j})
\\
& = & 
h_{0}^{-1} \det(\bar{\sigma}^-),
\end{eqnarray*}
as is seen by a direct computation.
Using the representation of $l_{0 n_0}$ 
provided by Proposition~\ref{prop:par-fp-curve}~(ii),
we obtain
$$
\frac{l_{0 n_0-j}}{\det(\bar{\sigma}^-)}
- l_{0 n_0}
\ = \ 
- h_{0}^{-1} l_{0 n_0-j} - l_{0 n_0}
\ = \
h_{0 j j}^{-1} l_{1 n_1}.
$$
\end{proof}

%% file: qsm-str-ell-fps.tex
\section{Quasismooth simple elliptic fixed points}

The aim of this section is to establish
Theorem~\ref{thm:qs-str-ell-fp} which
specifies obstructions to the existence of
quasismooth simple elliptic fixed points. 
This is a first step towards the proof
of Theorem~\ref{thm:main}, but also has
some general applications to the
geometry
of rational projective $\KK^*$-surfaces,
see Corollaries~\ref{cor:qsm1}
to~\ref{cor:qsmlast}.
Let us recall our definition of
a simple elliptic fixed point.

\begin{definition}
We say that an elliptic fixed point $x$
of a rational projective
$\KK^*$-surface $X$ is \emph{simple}
if for the minimal resolution
$\pi \colon \tilde X \to X$
of singularities,
the fiber $\pi^{-1}(x)$ is 
contained in an arm of $\tilde X$.
\end{definition}

Note that a simple elliptic
fixed point has in particular
no parabolic fixed point curve
in its fiber of the minimal
resolution of singularities.
We discuss two examples of
simple elliptic fixed points, making
use of the resolution of singularities
provided by Remark~\ref{rem:surfsingres}
and the formulae for the self intersection
numbers given in Remark~\ref{rem:selfintKstar}.

\begin{example}
The following matrices 
$P$ and $\tilde P$ define the minimal
resolutions $\tilde X \to X$ of
non-toric $\KK^*$-surfaces $X$,
each of them with a
simple elliptic fixed point~$x^- \in X$.
\begin{enumerate}
\item
Here $x^-$ is quasismooth but singular
and the fiber over~$x^-$ equals  
the curve $\tilde D_{04} \subseteq \tilde X$
in the $0$-th arm:
\begin{eqnarray*}
P
& = &
\left[
\begin{array}{rrrrrrrr}
-1 & -2 & -3 & 1 & 1 & 0 & 0 & 0
\\                        
-1 & -2 & -3 & 0 & 0 & 1 & 1 & 0
\\
 0 & -1 & -2 & 1 & 0 & 1 & 0 & 1
\end{array}
\right],
\\[5pt]
\tilde P 
& = &
\left[
\begin{array}{rrrrrrrrr}
-1 & -2 & -3 & -1 & 1 & 1 & 0 & 0 & 0
\\                        
-1 & -2 & -3 & -1 & 0 & 0 & 1 & 1 & 0
\\
 0 & -1 & -2 & -1 & 1 & 0 & 1 & 0 & 1
\end{array}
\right].
\end{eqnarray*}
The curve $\tilde D_{04}$ is isomorphic to
a projective line and has
self intersection number equal to $-2$.
In other words, $x^- \in X$ is
an $A_2$-singularity.
\item
Here $x^-$ is not quasismooth and 
the fiber over $x^-$ equals the curve 
$\tilde D_{08} \subseteq \tilde X$
in the $0$-th arm:
\begin{eqnarray*}
\qquad\qquad
P 
& = &  
\left[
\begin{array}{rrrrrrrrrrrrr}
-1 & -2 & -3 & -4 & -5 & -6 & -7 & 1 & 2 & 3 & 0 & 0 & 0
\\ 
-1 & -2 & -3 & -4 & -5 & -6 & -7 & 0 & 0 & 0 & 1 & 2 & 0
\\
0 &  -1 & -2 & -3 & -4 & -5 & -6 & 1 & 1 & 1 & 1 & 1 & 1
\end{array}
\right],
\\[5pt]
\tilde P 
& = &
\left[
\setlength{\arraycolsep}{4.5pt}        
\begin{array}{rrrrrrrrrrrrrrr}
-1 & -2 & -3 & -4 & -5 & -6 & -7 & -1 & 1 & 2 & 3 & 0 & 0 & 0
\\ 
-1 & -2 & -3 & -4 & -5 & -6 & -7 & -1 & 0 & 0 & 0 & 1 & 2 & 0
\\
0 &  -1 & -2 & -3 & -4 & -5 & -6 & -1 & 1 & 1 & 1 & 1 & 1 & 1
\end{array}
\right].
\\
\end{eqnarray*}
The curve $\tilde D_{08}$ is of intersection number
$-1$ and has a cusp singularity.
The singularity $x^- \in X$ is isomorphic to the
Brieskorn-Pham singularity
$$
0
\ \in \ 
V(T_1^7 + T_2^3 + T_3^2)
\ \subseteq \
\KK^3,
$$
where we gain this presentation by looking
at $X \cap Z_{\sigma^-}$ for the  (smooth)
affine toric chart $Z_{\sigma^-} \subseteq Z$;
compare also~\cite[No.~2.5 on p.~72]{Kol}.
\end{enumerate}
\end{example}

\begin{definition}
We say that a parabolic
fixed point curve $D \subseteq X$
is of a rational projective
$\KK^*$-surface is \emph{gentle}
if there is an arm
$\mathcal{A}_i = D_{i1} \cup \ldots \cup D_{in_i}$
such that the (unique) point
$x \in D \cap \mathcal{A}_i$
is a smooth point of $X$.
\end{definition}

\begin{theorem}
\label{thm:qs-str-ell-fp}
Let $X$ be a non-toric rational projective
$\KK^*$-surface $X$ with a
quasismooth simple elliptic fixed
point $x \in X$.
\begin{enumerate}
\item 
There is no gentle non-negative 
parabolic fixed point curve
in $X$.
\item
There is no other quasismooth
simple elliptic fixed point
in $X$.
\end{enumerate}
\end{theorem}

\begin{remark}
The assumption that $X$ is non-toric
is essential in
Theorem~\ref{thm:qs-str-ell-fp}.
A cheap smooth toric counterexample 
is given by the projective plane $\PP_2$:
Consider the two $\KK^*$-actions 
given by
$$ 
t \cdot [z]
\ = \ 
[z_0,z_1,tz_2],
\qquad\qquad
t \cdot [z]
\ = \ 
[z_0,tz_1,t^2z_2].
$$
For the first one, $[0,0,1]$ is an 
elliptic fixed point and 
$V(T_2)$ a parabolic fixed point 
curve of self intersection one.
The second one has $[1,0,0]$ and 
$[0,0,1]$ as elliptic fixed 
points.
\end{remark}

\begin{corollary}
\label{cor:qsm1}
Every rational projective $\KK^*$-surface
with two quasismooth simple elliptic fixed
points is a toric surface.
\end{corollary}

\begin{corollary}
\label{cor:qsm2}
Every quasismooth non-toric rational projective
$\KK^*$-surface with a simple elliptic fixed
point has a fixed point curve.
\end{corollary}

The latter says that, when considering
quasismooth non-toric rational projective
$\KK^*$-surfaces~$X$, we always may assume
that there is a curve $D^+ \subseteq X$.
For smooth $X = X(A,P)$, this allows us to 
complement~\cite[Thm.~2.5]{OrWa3}
by showing that the defining matrix $P$
is basically determined by the self
intersection numbers of invariant curves.
More precisely, we obtain the following.

\begin{corollary}
\label{cor:qsmlast}
Let $X = X(A,P)$ be smooth, non-toric
with $P$ adapted to $D^+ \subseteq X$.
Then all entries $l_{ij}$ and $d_{ij}$
of $P$ can be expressed via
self intersection numbers according to
Corollary~\ref{prop:par-fp-curve}.
\end{corollary}

\begin{definition}
Let $X = X(A,P)$ have a simple
elliptic fixed point $x \in X$.
We call $0 \le i \le r$ an
\emph{exceptional index} of $x$
if $\pi^{-1}(x)$ is contained
in the $i$-th arm of
$\tilde X = X(A,\tilde P)$,
where $\pi \colon \tilde X \to X$ is
the minimal resolution of singularities.
\end{definition}

Note that for any singular simple elliptic
fixed point the exceptional index is
unique.
The following characterization
of simple quasismooth elliptic
fixed points
is an important ingredient for the proof
of Theorem~\ref{thm:qs-str-ell-fp}.

\begin{proposition}
\label{prop:qsm-str-ell-char}
Let $x \in X = X(A,P)$ be a
quasismooth elliptic fixed point
with leading indices 
$\iota_0, \iota_1$.
\begin{enumerate}
\item
Assume $x = x^+$.
Then $x$ is simple 
with exceptional index $\iota_0$ 
if and only if
there exists a vector 
$u \in \ZZ^{r} \times \ZZ_{<0}$ 
such that
$$
\qquad
\bangle{u,v_{\iota_1 1}} = -1,
\quad
\bangle{u,v_{i 1}} =  0, \ i \ne \iota_0,\iota_1,
\quad
\bangle{u, v_{i j}} \ge 0, \ i \ne \iota_1. 
$$
We have $l_{i1} = 1$ whenever 
$i \ne  \iota_0, \iota_1$. 
Moreover, if
$u \in \ZZ^{r} \times \ZZ_{<0}$
is a vector as above, then 
the following holds:
$$
\qquad\qquad
0
\ < \
m^+
\ \le \
- u_{r+1} m^+
\ \le \
\frac{1}{l_{\iota_1 1}}.
$$
\item
Assume $x = x^-$.
Then $x$ is simple 
with exceptional index $\iota_0$ 
if and only if
there exists a vector 
$u \in \ZZ^{r} \times \ZZ_{<0}$ 
such that
$$
\qquad
\bangle{u,v_{\iota_1 n_{\iota_1}}} = -1,
\quad
\bangle{u,v_{i n_i}} =  0, \ i \ne \iota_0,\iota_1,
\quad
\bangle{u, v_{i j}} \ge 0, \ i \ne \iota_1. 
$$
We have $l_{in_i} = 1$ whenever 
$i \ne  \iota_0, \iota_1$. 
Moreover, if
$u \in \ZZ^{r} \times \ZZ_{>0}$
is a vector as above, then 
the following holds:
$$
\qquad\qquad
0
\ > \
m^-
\ \ge \
u_{r+1} m^-
\ \ge \
-\frac{1}{l_{\iota_1 n_{\iota_1}}}.
$$
\end{enumerate}
\end{proposition}

\begin{lemma}
\label{lem:iterate-vectors}
Consider a sequence of vectors $v_0, \ldots, v_k \in \QQ^2$
such that there are $c_1, \ldots, c_{k-1} \in \ZZ_{\ge 2}$
with
$$
v_{j+1} \ = \ c_jv_j - v_{j-1},
\qquad
j = 1 ,\ldots, k-1.
$$
Then, for $k \ge 2$, the difference $v_k - v_{k-1}$
lies in $\tau := \cone(v_1,v_1-v_0)$ and 
the vector $v_k$ lies in the
shifted cone $\tau + v_1$.
\end{lemma}

\begin{proof}
Clearly, $v_1 \in \tau + v_1$
and $v_1-v_0 \in \tau$.
We proceed inductively.
For $j \ge 1$, assume $v_j \in \tau + v_1$
and $v_{j}-v_{j-1} \in \tau$.
Write $v_j = v_j' + v_1$
with $v_j' \in \tau$.
Then, using $c_j \ge 2$ we see
$$
v_{j+1}
\ = \
c_jv_j - v_{j-1}
\ = \
(c_j-1)v_j' +(c_j-2)v_1 + (v_j- v_{j-1}) + v_1
\ \in \
\tau + v_1.
$$
\end{proof}

\begin{lemma}
\label{lem:half-plane}
Consider 
$H := \{(x,y) \in \QQ^2; \ x - y \ge 1 \}$.
Given $v_0 = (a,b)$ and $v_1 = (c,d)$ in $H$
with $a < 0$ and $c > 0$,
there is no $u \in \ZZ \times \ZZ_{> 0}$
satisfying
$$
\text{\rm (i)} \quad
\bangle{u,v_0} 
\ = \
u_1a + u_2 b
\ \ge \ 
0,
\qquad\qquad
\text{\rm (ii)} \quad
\bangle{u,v_1} 
\ = \
u_1c + u_2 d
\ = \ 
-1.
$$
\end{lemma}

\begin{proof}
Since $b < a < 0$ and $u_2 > 0$ hold, 
we infer $u_1 \le -2$ from~(i).
Then~(ii) tells us $d > 0$.
Now, plugging $u_1 = -(u_2d+1)/c$
into~(i) leads to a contradiction:
$$
u_2
\ \le \  
\frac{a}{bc-ad}
\ \le \ 
\frac{b+1}{bc-ad}
\ \le \
\frac{b+1}{bc}
\ \le \
1 + \frac{1}{bc}
\ < \ 1.
$$
\end{proof}

\begin{lemma}
\label{lem:4vectors}
Consider four vectors $\xi_1, \xi_2$ and $\eta_1, \eta_2$ 
in $\ZZ^2$ satisfying the following conditions: 
$$ 
\det(\xi_2,\xi_1) 
\ =  \
\det(\xi_1,\eta_1)
\ = \ 
\det(\eta_1,\eta_2)
\ = \ 
1,
\qquad \qquad
\det(\xi_2,\eta_2)
\ \ge \ 
1.  
$$
Then $\xi_1 = a \eta_1 - \eta_2$ and 
$\xi_2 = b \eta_1 - c \eta_2$, where
$a,b,c > 0$ and $c = ab -1$. 
In particular, 
$$ 
a = 1 \ \Rightarrow \ \det(\xi_2,\eta_2) = 1,
\qquad 
a \ge 2 \ \Rightarrow \ c - b \ge 1.
$$
\end{lemma}

\begin{proof}
In suitable linear coordinates,
we have
$\eta_1 = (0,-1)$ and $\eta_2 = (1,0)$
and moreover $\xi_1, \xi_2 \in \ZZ_{<0}^2$.
In this situation, the assertion 
can be directly verified.
\end{proof}  

\begin{proof}[Proof of Proposition~\ref{prop:qsm-str-ell-char}]
First observe that multiplying the last
row of $P$ by $-1$ interchanges
source and sink and thus it suffices
to prove Assertion~(ii).
For this we may assume that $P$ is
adapted to the sink.
Consider the minimal resolution
$\pi \colon \tilde X \to X$,
where $\tilde X = X(A, \tilde P)$,
as provided by Remark~\ref{rem:surfsingres}.
Then, for every $i = 0, \ldots, r$, the columns 
$v_{i1}, \ldots, v_{in_i}$ of $P$ occur 
among the columns 
$\tilde v_{i1}, \ldots, \tilde v_{i \tilde n_i}$
of $\tilde P$.
Moreover, Proposition~\ref{prop:ell-q-smooth} 
yields 
$$
\tilde l_{i \tilde n_i} 
\ = \ 
l_{in_i} 
\ = \ 
1,
\qquad
\tilde d_{i \tilde n_i} 
\ = \ 
d_{in_i} 
\ = \ 
0
\qquad
\text{for } i = 2, \ldots, r.
$$

First suppose that $x^- \in X$ is simple.
We may assume that the exceptional index 
of $x^-$ is $0$ that means that the 
divisors inside~$\pi^-(x)$
are located in the $0$-th arm of $\tilde X$.
Then $\tilde l_{1 \tilde n_1} = l_{1n_1}$
and $\tilde d_{1 \tilde n_1} = d_{1n_1}$ hold.
Define $u \in \ZZ^{r+1}$
by $u_1 := \tilde d_{0 \tilde n_0}$ and
$u_i := 0$ for $i = 1,\ldots,r$
and $u_{r+1} := \tilde l_{0 \tilde n_0}$.
Then Proposition~\ref{prop:ell-q-smooth} 
yields
$$
\bangle{u,v_{1n_1}}
\ = \ 
\tilde d_{0 n_0} l_{1n_1} + \tilde l_{0 \tilde n_0} d_{1 n_1} 
\ = \
\tilde l_{0 \tilde n_0} \tilde d_{1 \tilde n_1} 
+ 
\tilde l_{1 \tilde n_1} \tilde d_{0 \tilde n_0}
\ = \
- 1.
$$
According to the definition of $u$, we have 
$\bangle{u,v_{in_i}} = 0$ for $i = 2, \ldots, r$.
Moreover, for $j = 0, \ldots, \tilde n_0$, we use 
slope orderedness of $\tilde P$ to see
$$ 
\bangle{u,\tilde v_{0j}}
\ = \ 
-
\tilde d_{0 \tilde n_0} \tilde l_{0j}
+
\tilde l_{0 \tilde n_0} \tilde d_{0j}
\ = \ 
\tilde l_{0 \tilde n_0} \tilde l_{0j}
\left(
\frac{\tilde d_{0j}}{\tilde l_{0j}}
-
\frac{\tilde d_{0 \tilde n_0}}{\tilde l_{0 \tilde n_0}}
\right)
\ \ge \ 
0.
$$
This means in particular $\bangle{u,v_{0j}} \ge 0$ 
for $j = 1, \ldots, n_0$.
Moreover, since $P$ is slope ordered and adapted 
to the sink, we have $d_{ij} \ge 0$ for all 
$i \ge 1$ and thus
$$
\bangle{u,v_{ij}}
\ = \   
\tilde l_{0\tilde n_0} d_{ij}
\ \ge \ 
0,
\qquad 
i = 2, \ldots, r,
\quad
j = 1, \ldots, n_i.
$$
Let us care about the estimate for the
slope sum $m^-$.
Evaluating $u$ at the vectors $v_{0n_0}$
and $v_{1n_1}$ gives us 
$$
-u_1 l_{0n_0} + u_{r+1} d_{0n_0}
\ \ge \
0,
\qquad\qquad 
u_1 l_{1n_1} + u_{r+1} d_{1n_1}
\ = \
-1.
$$
Solving the second condition for $u_1$ and
plugging the result into the first one,
gives us the estimate
$$
u_{r+1} (l_{0n_0}d_{1n_1} + l_{1n_1}d_{0n_0})
\ \ge \
- l_{0n_0}.
$$
The expression
$l_{0n_0}d_{1n_1} + l_{1n_1}d_{0n_0}$
equals $l_{0n_0}l_{1n_1}m^-$ and
is negative due to
Proposition~\ref{prop:ell-q-smooth}.
We conclude
$$
\qquad
1
\ \le \
u_{r+1}
\ \le \
- \frac{1}{l_{1n_1}m^-}.
$$
This directly yields the desired
lower bound for $u_{r+1}m^-$.
The upper bound $0 > m^-$ is 
guaranteed by
Remark~\ref{rem:m+m-}.

Now suppose that there is 
a vector $u \in \ZZ^r \times \ZZ_{>0}$ 
as in the proposition.
Suitably arranging $P$ and adapting~$u$,
we achieve $\iota_0 = 0$ and
$\iota_1 = 1$.
Note that for each $i = 2, \ldots, r$,
we have $u_i = 0$ due to $d_{in_i} = 0$.
We assume that $x^-$ is not simple
and show that this leads to a contradiction.
For this, it suffices to verify 
$$ 
(-l_{0n_0},d_{0n_0}), \, (l_{1n_1},d_{1n_1})
\ \in \ 
H 
\ := \ 
\{(x,y) \in \QQ^2; \ x-y \ge 1\},
$$
because then Lemma~\ref{lem:half-plane}
implies that $u$ cannot evaluate 
non-negatively on $v_{0n_0}$ and 
to $-1$ on $v_{1n_1}$.
Since $P$ is slope ordered,
$(l_{1n_1},d_{1n_1})$ lies in $H$.
In order to see that also 
$(-l_{0n_0},d_{0n_0})$ belongs to $H$,
we use the assumption that 
$x^-$ is not simple and
thus we are in one of the
following two cases. 

\medskip
\noindent
\emph{Case 1:} The fiber $\pi^{-1}(x) \subseteq \tilde X$ 
contains a parabolic fixed point curve 
$\tilde D^-$.
Consider the sequence of divisors 
$
\tilde D^-, 
\tilde D_{0 \tilde n_0}, 
\ldots, 
\tilde D_{0 \tilde n_0 - k}
$  
connecting $\tilde D^-$ with 
the proper transform 
$\tilde D_{0 \tilde n_0 - k}$
of~$D_{0n_0} \subseteq X$.
This give us a sequence of 
pairs 
$$ 
(0,-1), \
(-1, \tilde d_{0 \tilde n_0}), \
\ldots, \
(-\tilde l_{0 \tilde n_0 - k}, \tilde d_{0 \tilde n_0 - k})
= 
(-l_{0n_0}, d_{0n_0}),
$$
where $\tilde l_{0 \tilde n_0} = 1$ due 
to Proposition~\ref{prop:smooth-par-li1}.
Since $\tilde X \to X$ is the minimal 
resolution, we have 
$(\tilde D^-)^2 \le -2$ and 
$\tilde D_{0j}^2 \le -2$
whenever $j > \tilde n_0 - k$.  
Thus, Remark~\ref{rem:selfintKstar} shows
$\tilde d_{0 \tilde n_0} \le -2$.
According to Lemma~\ref{lem:intnosDij}, 
the above sequence of pairs satisfies 
the assumptions of Lemma~\ref{lem:iterate-vectors}.
Applying the latter yields
$(-l_{0n_0},d_{0n_0}) \in H$.

\medskip
\noindent
\emph{Case 2:} The fiber $\pi^{-1}(x) \subseteq \tilde X$ 
contains an elliptic fixed point $\tilde x^-$ 
and curves from the arms $0$ and $1$ of $\tilde X$. 
The curves in $\pi^{-1}(x)$ are   
$
\tilde D_{i \tilde n_i}, 
\ldots, 
\tilde D_{i \tilde n_i - k_i}
$  
where $i = 0,1$ and 
$\tilde D_{i \tilde n_i - k_i-1}$
is the proper transform 
of~$D_{0n_0} \subseteq X$.
Consider the associated sequences 
of pairs
$$ 
(\tilde l_{0 \tilde n_0}, \tilde d_{0 \tilde n_0}), \
\ldots, \
(-\tilde l_{0 \tilde n_0 - k_0}, \tilde d_{0 \tilde n_0 - k_0})
= 
(-l_{0n_0}, d_{0n_0}),
$$
$$
(\tilde l_{1 \tilde n_1}, \tilde d_{1 \tilde n_1}), \
\ldots, \
(\tilde l_{1 \tilde n_1 - k_1}, \tilde d_{1 \tilde n_1 - k_1})
= 
(l_{1n_1}, d_{1n_1}).
$$
By slope orderedness of $\tilde P$, all
members of the second sequence lie in $H$.
Now, write $\xi_1$, $\xi_2$ for the first two 
pairs of the first sequence and
$\eta_1$, $\eta_2$ for the first two 
pairs of the second one.
Then Lemma~\ref{lem:4vectors}
provides us with integers $a,b,c > 1$,
where $c = ab-1$ such that
$$ 
\xi_1 \ = \ a \eta_1 - \eta_2, 
\qquad 
\xi_2 \ =  \ b \eta_1 - c \eta_2.
$$
Here $a \ge 2$ holds as otherwise
Proposition~\ref{prop:ell-q-smooth} shows
that $\tilde D_{0 \tilde n_0}$
and $\tilde D_{1 \tilde n_1}$  contract
smoothly which contradicts to minimality
of the resolution.
Thus,  $\xi_1, \xi_2 \in H$.
Again by minimality of the resolution,
all $\tilde D_{0 j} \subseteq \pi^{-1}(x)$
are of self intersection at most $-2$.
Using Lemmas~\ref{lem:intnosDij}
and ~\ref{lem:iterate-vectors},
we arrive at $(l_{1n_1}, d_{1n_1}) \in H$.
\end{proof}

\begin{lemma}
\label{lem:Padapted}
Consider the defining matrix $P$ of
a rational projective $\KK^*$-surface
$X(A,P)$. 
\begin{enumerate}
\item
If $m_{ij} = 0$ holds for 
$0 \le i \le r$ and $1 \le j \le n_i$,
then $d_{ij} = 0$ and $l_{ij} = 1$.
\item
If $P$ is adapted to the sink, then  
$0 \le m_{in_i} <  m_{in_i} + l_{in_i}^{-1} \le 1$ for $i = 1, \ldots, r$.
\item
If $P$ is irredundant and adapted
to the sink, then 
$m_{i1} > 0$ for $i = 1, \ldots, r$.
\end{enumerate}
\end{lemma}

\begin{proof}
We verify~(i).
If $m_{ij} = 0$ holds, then we must have
$d_{ij} = 0$ and thus
primitivity of the column $v_{ij}$ yields
$l_{ij} = 1$.
We turn to~(ii).
As $P$ is adapted to the sink, we have
$0 \le d_{in_i} < l_{in_i}$ whenever $i \ge 1$
and the desired estimate follows.
We prove~(iii).
By slope orderedness of $P$ and~(ii),
we have $m_{i1} \ge m_{in_1} \ge 0$.
We exclude $m_{i1} = 0$.
Otherwise, $m_{in_i} = 0$ holds.
Thus, $d_{i1} = d_{in_i} = 0$ 
and~(i) yields
$l_{i1} = l_{in_i} = 1$
and $n_i = 1$.
This is a contradiction to
irredundance of $P$.
\end{proof}

\begin{proof}[Proof of Theorem~\ref{thm:qs-str-ell-fp}]
We may assume $X = X(A,P)$ and that
the quasismooth simple elliptic fixed
point is $x^- \in X$.
Moreover, we may assume that $P$
is irredundant, adapted to the sink
and that the two leading indices from
Proposition~\ref{prop:qsm-str-ell-char}
are~$0$ and $1$.
Then we have 
$$
m_{01}
\ \ge \
m_{0n_0},
\qquad
m_{11}
\ \ge \
m_{1n_1}
\ \ge \
0,
\qquad
m_{i1}
\ > \
m_{in_i}
\ = \
0, \
i = 2, \ldots, r,
$$
by slope orderedness,
Proposition~\ref{prop:ell-q-smooth}
and Lemma~\ref{lem:Padapted}.
In particular, 
$m^-$ equals $m_{0n_0} + m_{n_1}$.
Using the estimate on $m^-$ from
Proposition~\ref{prop:qsm-str-ell-char}
and Lemma~\ref{lem:Padapted}~(ii),
we see
$$
0
\ \ge \
- m_{1n_1}
\ > \
m_{0n_0}
\ \ge \
- m_{1n_1} - \frac{1}{l_{1n_1}}
\ \ge \
-1.
$$

We prove~(i).
Let $D^+ \subseteq X$ be a
non-negative parabolic
fixed point curve.
We have to show that $D^+$ is
not gentle.
According to
Proposition~\ref{prop:smooth-par-li1},
this means to verify
$l_{i1} \ge 2$ for $i = 0, \ldots, r$.
Remark~\ref{rem:selfintKstar} 
yields
$$
0
\ \ge \
m^+
\ = \
m_{01} + m_{11} + m_{21} + \ldots + m_{r1},
$$
where $m_{11}, \ldots, m_{r1} > 0$
and $r \ge 2$,
hence $0 > m_{01}$.
Moreover, $m_{01} \ge m_{0n_0} \ge -1$
yields $0 < m_{i1} < 1$ for
$i = 1, \ldots, r$.   
This implies $l_{i1} \ge 2$ for $i = 1, \ldots, r$.   
We show $l_{01} \ge 2$.
Otherwise, $l_{01} = 1$
and thus $m_{01} = -1$.
Then $m_{0n_0} = -1$.
Consequently $n_0 = 1$ and
$l_{01} = -d_{01} = 1$ by
primitivity of~$v_{01}$.
This is a contradiction to
irredundance of $P$.

\goodbreak

We prove~(ii). 
Suppose that there is also a 
quasismooth simple elliptic 
fixed point $x^+ \in X$.
Let $0 \le \iota_0, \iota_1 \le r$ 
be the leading indices
as in Proposition~\ref{prop:qsm-str-ell-char}~(i).
Then $l_{i1} = 1$ holds
whenever $i \ne \iota_0,\iota_1$.
This allows us to assume $\iota_0,\iota_1 \le 3$.
The estimates from
Proposition~\ref{prop:qsm-str-ell-char}
yield
$$
m^+-m^-
\ = \
m_{01} - m_{0n_0} + m_{11} - m_{1n_1} + m_{21} + m_{31}
+ \sum_{i=4}^r d_{i1}
\ \le \
\frac{1}{l_{1n_1}} + \frac{1}{l_{\iota_1 1}}
\ \le \ 2.
$$

\medskip
\noindent
\emph{Case $\iota_0 \le 1$ and $\iota_1 \le 1$}.
Then we have $1 \le d_{21} = m_{21}$.
From the above estimate, we infer
$$ 
0
\ \le \ 
m^+ - m^- -1
\ \le \
\frac{1}{l_{1n_1}} + \frac{1}{l_{\iota_1 1}}
- 1
\ = \
\frac{l_{1n_1} + l_{\iota_1 1} -l_{1n_1}l_{\iota_1 1}}{l_{1n_1}l_{\iota_1 1}} .
$$
This leaves us with the following possibilities:
first $l_{1n_1} = l_{\iota_1 1} = 2$,
second $l_{1n_1}= 1$
and third  $l_{\iota_1 1} = 1$.
We go through these cases.

\medskip
\noindent
\emph{Let $l_{1n_1} = l_{\iota_1 1} = 2$.}
Then $m_{01} = m_{0n_0}$
and $m_{11} = m_{1n_1}$ as well as 
$d_{21} = 1$ hold.
Thus, $n_0 = n_1 = 1$.
Moreover, $l_{11} = 2$ implies $d_{11} = 1$.
Proposition~\ref{prop:qsm-str-ell-char} 
tells us
$$ 
-2\left(m_{01} + \frac{1}{2} \right) 
\ = \
-l_{1n_1}m^-
\ \le \
1,
\qquad \quad
2\left(m_{01} + \frac{1}{2} + 1 \right) 
\ \le \ 
l_{\iota_1 1}m^+
\ \le \
1.
$$
Thus $m_{01} \ge -1$ and $m_{01} \le -1$.
As $v_{01}$ is primitive, we arrive at 
$l_{01} = -d_{01} = 1$, which is a contradiction 
to the irredundance of $P$.

\medskip
\noindent
\emph{Let $l_{1n_1} = 1$}.
As $P$ is adapted to the sink,
$d_{1n_1} = 0$ holds.
Irredundance yields
$n_1 \ge 2$ and $d_{11} > 0$.
As noted before, $m_{01} \ge m_{0n_0} \ge -1$.
If $\iota_1 = 1$, then 
$$ 
0 
\ \le \
d_{11} 
+ 
l_{11} (m_{01} + 1)
=
l_{11} (m_{01} + m_{11} + 1)
\ \le  \
l_{11}m^+
\ \le \ 
1
$$
due to Proposition~\ref{prop:qsm-str-ell-char}.
This implies $d_{11} = 1$ and 
$m_{01} = m_{0n_0} = -1$. 
Thus, $n_0 = 1$ and $-d_{01} = l_{01} = 1$
holds; a contradiction.  
If $\iota_1 = 0$, then we have
$$ 
0 
\ \le \  
d_{01} + l_{01} 
\ < \  
l_{01} (m_{01} + m_{11} + 1)
\ \le  \
l_{01}m^+
\ \le \ 
1.
$$
We conclude $-d_{01}  = l_{01} = 1$, using 
primitivity of $v_{01}$.
Then $m_{01} \ge m_{0n_0} \ge -1$ implies
$n_{01} = 1$. 
A contradiction to irredundance of $P$.

\medskip
\noindent
\emph{Let $l_{\iota_1 1} = 1$}.
This case transforms into the preceding 
one by switching source and sink via 
multiplying the last row of $P$ by $-1$
and adapting to the new sink.

\medskip
\noindent
\emph{Case $\iota_0 \le 1$ and $\iota_1 \ge 2$}.
Then we may assume $\iota_1 = 2$.
If $\iota_0 = 0$, then we have
$l_{11} = 1$.
Thus, $n_1 \ge 2$ and $d_{11} > 0$.
Proposition~\ref{prop:qsm-str-ell-char}
and  Lemma~\ref{lem:Padapted}~(ii)
show
$$
d_{11} + \frac{d_{21}}{l_{21}}
\ = \ 
m_{11} + m_{21}
\ \le \
 m_{1n_1} + m^+ - m^-
\ \le \
m_{1n_1}
+
\frac{1}{l_{1n_1}}
+
\frac{1}{l_{21}}
\ \le \
2.
$$
Now, $d_{11} = 2$ yields $m_{21} = 0 = m_{2n_2}$,
which is impossible by irredundance of $P$.
Thus, $d_{11} = 1$.
The above inequality gives $d_{21} = 1$
and $d_{1n_1} = l_{1n_1} -1$.
Hence,
$$
m_{01}
\ \le  \
m^+ - m_{11} - m_{21}
\ \le \ 
\frac{1}{l_{21}}
-1 
- \frac{1}{l_{21}}
\ = \
-1.
$$
So, $m_{01} \ge m_{0n_0} \ge -1$ yields
$m_{0n_0} = m_{01} = -1$.
Thus, $n_{0} = 1$ and $l_{01} = 1$;
a contradiction.
If $\iota_0 = 1$, then $l_{01} = 1$.
Hence $m_{01} \in \ZZ$
and $n_0 \ge 2$.
Now
$$
-1
\ \le \ 
m_{0n_0}
\ < \ 
m_{01}
\ \le \
m^+ - m_{21} - m_{11}
\ \le \
\frac{1}{l_{21}} - \frac{d_{11}}{l_{11}} - \frac{d_{21}}{l_{21}}
$$
and $d_{21} > 0$ imply $d_{11} = 0$.
Then $0 = m_{11} \ge m_{1n_1} \ge  0$ yields
$n_1 = 1$ and $l_{11} = 1$.
A contradiction to irredundance of $P$.

\medskip
\noindent
\emph{Case $\iota_0 \ge 2$ and $\iota_1 \le 1$}.
We may assume $\iota_0 = 2$.
If $\iota_1 = 1$, then $l_{01} = 1$.
Hence $m_{01} \in \ZZ$ and $n_0 \ge 2$.
We derive $m_{1n_1} = m_{11} = d_{11} = 0$
and thus $l_{11} = 1$ from
$$
-1
\ \le \ 
m_{0n_0}
\ < \
m_{01}
\ \le \
m^+ - m_{11} - m_{21}
\ \le \
\frac{1}{l_{11}} - \frac{d_{11}}{l_{11}} - \frac{d_{21}}{l_{21}}.
$$
A contradiction to irredundance of $P$.
The case $\iota_1 = 0$ transforms to the
case $\iota_0 = 1$ and $\iota_2= 1$
settled before by switching
sink and source.

\medskip
\noindent
\emph{Case $\iota_0 \ge 2$ and $\iota_1 \ge 2$}.
We may assume $\iota_0 =2$ and $\iota_1 = 3$.
Then $l_{01} = l_{11}= 1$ and
$n_0, n_1 \ge 2$ hold.
In particular,
$d_{01} = m_{01} > m_{0n_0} \ge -1$.
Moreover, $d_{11}$, $d_{21}$
and $d_{31}$ are all strictly
positive.
This contradicts to the estimate
$$ 
d_{01} + d_{11}
+
\frac{d_{21}}{l_{21}}
+
\frac{d_{31}}{l_{31}}
\ \le \
m^+
\ \le \
\frac{1}{l_{31}}.
$$
\end{proof}

\begin{example}
We present $\KK^*$-surfaces $X = X(A,P)$
with a smooth elliptic fixed point $x^-$
and a positive parabolic fixed point
curve $D^+$. Consider  
$$
P
\ = \
\left[
\begin{array}{rcccc}
-l_{01} & l_{11} & 0 & 0      & 0
\\                             
-l_{01} & 0      & l_{22} & 1 & 0
\\
d_{01}     & d_{11} & 1 & 0      & 1                         
\end{array}    
\right] . 
$$
Now choose the entries $l_{ij}$ and $d_{ij}$
in such a way that $P$ is adapted
to the sink and we have
$$
l_{01}d_{11} + l_{11}d_{01} \ = \ -1,
\qquad
l_{21} \ > \ l_{01}l_{11}
$$
Then $x^-$ is smooth and $D^+$ has
self intersection number $l_{21} - l_{01}l_{11}$.
For example we can take
$$
P
\ = \
\left[
\begin{array}{rrrrr}
-2 & 3 & 0 & 0 & 0
\\                             
-2 & 0 & 7 & 1 & 0
\\
-1 & 1 & 1 & 0 & 1                         
\end{array}    
\right] . 
$$
Then $x^-$ is smooth and we have $m^+ = -1/42$.
Consequently, $D^+$ has self intersection
number $-m^+ = 1/42$.
Observe that $D^+$ is not gentle.
\end{example}

%% file: hor-ver-roots.tex
\section{Horizontal and vertical $P$-roots}
\label{sec:horverProots}

In section, we introduce horizontal and
vertical $P$-roots as adapted versions
of the general Demazure $P$-roots 
to the special case of rational projective 
$\KK^*$-surfaces $X = X(A,P)$.
This allows a less technical treatment. 
The main results of this section 
are Propositions~\ref{prop:hor-P-root-constraints},
\ref{prop:qsefp2novroots} and~\ref{prop:vroot+2novroot-}
showing geometric constraints to the existence 
of $P$-roots and 
Propositions~\ref{prop:hor-P-roots-gamma},
\ref{prop:vrootchar} which
identify the $P$-roots in terms of the defining
matrix~$P$.

\begin{definition}
\label{def:hor-P-root}
Consider a rational projective $\KK^*$-surface
$X = X(A,P)$ and assume that $P$ is irredundant.
\begin{enumerate}
\item
Let $x^+ \in X$ be an elliptic
fixed point and $0 \le i_0,i_1 \le r$.
A \emph{horizontal $P$-root} at $(x^+,i_0,i_1)$
is a vector $u \in \ZZ^r \times \ZZ_{<0}$ such that
$$
\qquad
\qquad
\bangle{u,v_{i_1 1}} = -1,
\quad
\bangle{u,v_{i1}} = 0, \, i \ne i_0,i_1,
\quad
l_{i1} = 1, \, i \ne i_0,i_1,
$$
$$
\qquad
\qquad
\bangle{u,v_{i_0 1}} \ge 0,
\
\bangle{u,v_{i_1 2}} \ge 0, \, n_{i_1} > 1,
\
\bangle{u,v_{i 2}} \ge l_{i 2}, \, i \ne i_0,i_1.
$$
\item
Let $x^- \in X$ be an elliptic
fixed point and $0 \le i_0,i_1 \le r$.
A \emph{horizontal $P$-root} at $(x^-,i_0,i_1)$ is
a vector $u \in \ZZ^r\times \ZZ_{>0}$ such that
$$
\qquad  
\qquad
\bangle{u,v_{i_1 n_{i_1}}} = -1,
\quad
\bangle{u,v_{i n_i}} = 0, \, i \ne i_0,i_1,
\quad
l_{in_i} = 1, \, i \ne i_0,i_1,
$$
$$
\qquad
\qquad
\bangle{u,v_{i_0 n_{i_0}}} \ge 0,
\
\bangle{u,v_{i_1 n_{i_1}-1}} \ge 0, \, n_{i_1} >1,
\
\bangle{u,v_{i n_i-1}} \ge l_{i n_i-1}, \, i \ne i_0,i_1.
$$
\end{enumerate}
We say that an elliptic fixed point $x \in X$
\emph{admits} a horizontal $P$-root if
there is a vector $u$ as in~(i) if $x = x^+$,
respectively a vector $u$ as in~(ii) if $x = x^-$.
\end{definition}

\begin{remark}
\label{rem:u0}
Given $u = (u_1,\ldots,u_{r+1}) \in \QQ^{r+1}$,
set $u_0 := -u_1 - \ldots - u_r$. 
For $i = 0, \ldots, r$, the linear form~$u$
evaluates at the colums $v_{ij}$ of 
a defining matrix~$P$ as
$$ 
\bangle{u,v_{ij}} 
\ = \ 
u_il_{ij} + u_{r+1} d_{ij}.
$$
This allows a unified treatment of the cases
$i=0$ and $i \ne 0$ and will be used frequently
in the sequel.
\end{remark}

The subsequent Propositions~\ref{prop:P-dem-to-P-root}
and~\ref{prop:P-root-to-P-dem}
together with Remark~\ref{rem:dem-P-vs-P}
give the precise relations between the horizontal
$P$-roots just defined and the 
horizontal Demazure $P$-roots recalled 
in Definition~\ref{def:Pdemroot}.

\begin{proposition}
\label{prop:P-dem-to-P-root}
Let $X = X(A,P)$ be non-toric,
$P$ irredundant and 
$(u,i_0,i_1,C)$ a horizontal
Demazure $P$-root.
Then precisely one of the following
statements holds:
\begin{enumerate}
\item
We have $u_{r+1} < 0$,  
there is an elliptic fixed
point $x^+ \in X$,
the vector $u$ is a horizontal $P$-root at 
$(x^+,i_0,i_1)$ and $c_i = 1$ holds for all 
$i \ne i_0$.
\item
We have $u_{r+1} > 0$,  
there is an elliptic fixed
point $x^- \in X$,
the vector $u$ is a horizontal $P$-root at 
$(x^-,i_0,i_1)$ and $c_i = n_i$ holds for all 
$i \ne i_0$.
\end{enumerate}
\end{proposition}

\begin{proof}
We show $u_{r+1} \ne 0$.
Otherwise, as $X$ is non-toric,
we find $0 \le i \le r$
different from $i_0,i_1$.
Then $\bangle{u,v_{ic_i}} = 0$
implies $u_i = 0$.
Since $P$ is irredundant
and $l_{ic_i} =1$ holds, there
is a $1 \le j \le n_i$ different
from $c_i$ and we have
$\bangle{u,v_{ij}} \ge l_{ij} > 0$.
This is impossible due to
$u_i = u_{r+1} = 0$.

Next we claim
$u_{r+1}m_{ic_i} \le u_{r+1}m_{ij}$
holds for every $0 \le i \le r$
with $i \ne i_0$ and every
$1 \le j \le n_i$.
Indeed, we infer
$$
u_i
\ = \
\begin{cases}
-u_{r+1}m_{ic_i}, & i \ne i_1,
\\
-\frac{1}{l_{ic_i}}-u_{r+1}m_{ic_i}, & i = i_1,
\end{cases}
\qquad
u_il_{ic_i} + u_{r+1} d_{ic_i}
\ \le \
u_il_{ij} + u_{r+1} d_{ij}
$$
from the conditions on 
$\bangle{u,v_{ij}}$ for 
$i \ne i_0$ stated in 
Definition~\ref{def:Pdemroot}.
Eliminating~$u_i$ in the 
above inequalities then directly 
yields the claim.

We show that for $u_{r+1} < 0$, we
arrive at~(i).
First, $P$ has no column $v^+$
by Definition~\ref{def:Pdemroot}
and $u_{r+1} < 0$.
Thus, there is an elliptic
fixed point $x^+ \in X$.
We have $m_{ic_i} \ge m_{ij}$
for $i \ne i_0$.
Hence, slope orderedness of $P$ forces
$c_i = 1$ for all $i \ne i_0$.
Clearly,~$u$ fulfills
the conditions of a horizontal
$P$-root at $(x^+,i_0,i_1)$.
Similarly, we see that $u_{r+1} > 0$
leads to~(ii).
\end{proof}

\begin{proposition}
\label{prop:P-root-to-P-dem}
Consider $X = X(A,P)$ and
assume that~$P$ is irredundant.
Define $(r+1)$-tuples
$C^+ := (1,\ldots,1)$ and
$C^- := (n_{0},\ldots,n_{r})$.
\begin{enumerate}
\item
Let $x^+ \in X$ be an
elliptic fixed point
and $u$ a horizontal $P$-root
at $(x^+,i_0,i_1)$.
Then $(u,i_0,i_1,C^+)$ is a Demazure $P$-root.
\item
Let $x^- \in X$ be an
elliptic fixed point
and $u$ a horizontal $P$-root
at $(x^-,i_0,i_1)$.
Then $(u,i_0,i_1,C^-)$ is a Demazure $P$-root.
\end{enumerate}
\end{proposition}

\begin{proof}
We exemplarily prove~(i). 
Recall that here we have $u_{r+1} < 0$.
Using the inequalities from 
Definition~\ref{def:hor-P-root}~(i)
and slope orderedness of~$P$, 
we obtain
$$
\frac{u_{i_0}}{u_{r+1}}
\ \ge \ 
m_{i_0 1}
\ \ge \ 
m_{i_0 j},
\qquad
j = 1, \ldots, n_{i_01} ,
$$
$$
-\frac{u_{i_1}}{u_{r+1}}
\ \ge \ 
m_{i_1 2} 
\ \ge \ 
m_{i_1 j},
\qquad
j = 2, \ldots, n_{i_11},
$$
$$
-\frac{u_{i}}{u_{r+1}}
\ \ge \ 
m_{i2} - \frac{1}{u_{r+1}}
\ \ge \ 
m_{ij}- \frac{1}{u_{r+1}},
\qquad
i \ne i_0,i_1,
\quad
j = 2, \ldots, n_i.
$$
Together with the two equations of 
Definition~\ref{def:hor-P-root}~(i),
this directly leads to the conditions 
of a Demazure $P$-root for 
$(u,i_0,i_1,C^+)$.
\end{proof}

\begin{remark}
\label{rem:dem-P-vs-P}
Let $X = X(A,P)$ be non-toric
and $P$ irredundant.
By Propositions~\ref{prop:P-dem-to-P-root}
and~\ref{prop:P-root-to-P-dem},
the horizontal Demazure $P$-roots
map surjectively to the
horizontal $P$-roots.
Here $\kappa$ and $\kappa'$ 
have the same image if and only if
$\kappa = (u,i_0,i_1,C)$ and
$\kappa' = (u,i_0,i_1,C')$,
where $C$ and $C'$ differ at most
in the $i_0$-th entry.
In this case, the locally nilpotent
derivations $\delta_\kappa$ and
$\delta_{\kappa'}$ on $\mathcal{R}(X) = R(A,P)$
coincide. 
\end{remark}

\goodbreak

\begin{proposition}
\label{prop:hor-P-root-constraints}
Consider a rational projective $\KK^*$-surface
$X = X(A,P)$, assume~$P$ to be irredundant
and let $0 \le i_0,i_1 \le r$.
\begin{enumerate}
\item
Let $X$ have an elliptic fixed point $x^+$.
If there is a horizontal $P$-root~$u$ 
at $(x^+,i_0,i_1)$,
then $x^+$ is simple, quasismooth with leading 
indices $i_0,i_1$ and
$$
\qquad\quad
0
\ < \
m^+
\ \le \
- u_{r+1} m^+
\ \le \
\frac{1}{l_{i_1 1}}.
$$
Additionally, the presence of a horizontal $P$-root~$u$
at $(x^+,i_0,i_1)$ forces $D_{i1}^2 \ge 0$ for all
$i = 0, \ldots, r$ with $i \ne i_0$. 
\item
Let $X$ have an elliptic fixed point $x^-$.
If there is a horizontal $P$-root at $(x^-,i_0,i_1)$,
then $x^-$ is simple, quasismooth with leading indices
$i_0,i_1$ and 
$$
\qquad\quad
0
\ > \
m^-
\ \ge \
u_{r+1} m^-
\ \ge \
-\frac{1}{l_{i_1 n_{i_1}}}.
$$
Additionally, the presence of a horizontal $P$-root~$u$
at $(x^-,i_0,i_1)$ forces $D_{in_i}^2 \ge 0$ for all
$i = 0, \ldots, r$ with $i \ne i_0$.
\end{enumerate}
Moreover, there exists at most one elliptic fixed point
in $X$ admitting a horizontal $P$-root.
\end{proposition}

\begin{proof}
The estimates can be treated at once,
writing $x = x^+,x^-$.
Proposition~\ref{prop:ell-q-smooth}
tells us that~$x$ is quasismooth
with leading indices $i_0,i_1$.
Moreover, according to 
Proposition~\ref{prop:P-root-to-P-dem},
the horizontal $P$-root $u$ satisfies 
the assumptions
of Proposition~\ref{prop:qsm-str-ell-char}.
Thus, $x$ is simple and we obtain the 
desired estimates.
For the self intersection numbers,
we exemplarily look at $x^-$.
For $n_i =1$, the claim directly
follows from 
Remark~\ref{rem:selfintKstar}.
For $n_i > 1$, assume $D_{in_i}^2 < 0$.
Then we infer from Remark~\ref{rem:conctractibility}
that~$v_{in_i}$
is a positive combination over $v_{in_i-1}$
and the $v_{kn_k}$ with $k \ne i$.
This contradicts to the definition of
a horizontal $P$-root $u$ at $(x^-,i_0,i_1)$.
The supplement is a consequence of
Theorem~\ref{thm:qs-str-ell-fp}.
\end{proof}

Our next step is to
identify the horizontal $P$-roots
as certain integers contained
in intervals 
$\Delta(\iota,\kappa) \subseteq \QQ_{\ge 0}$,
which in turn are extracted in the
following way from the defining matrix~$P$.

\begin{construction}
\label{constr:intervalls}
Consider the defining matrix $P$ of 
$X(A,P)$.
For $0 \le i,k \le r$, define
rational numbers
$$ 
\eta_k
\ := \ 
-\frac{1}{l_{k n_{k}}m^-},
\qquad\qquad
\xi_i
\ := \ 
\begin{cases}
0, & n_i = 1,
\\
\frac{1}{l_{i n_i}(m_{i n_i-1}-m_{i n_i})}, & n_i \ge 2. 
\end{cases}
$$
Then all $\xi_i$ and $\eta_k$ are non-negative.
Moreover, for $0 \le i,k \le r$ 
with $i \ne k$, consider the 
sets 
$$ 
[\xi_i, \eta_k]
\ = \ 
\{t \in \QQ; \ \xi_i \le t \le \eta_k\}
\ \subseteq \
\QQ_{\ge 0}.
$$
Note that $[\xi_i, \eta_k]$ may be empty.
Finally, for any two $0 \le \iota, \kappa \le r$, 
we have the intersections
$$
\Delta(\iota, \kappa)
\ = \ 
\bigcap_{i \ne \iota} \ [\xi_i, \eta_\kappa]
\ \subseteq \ 
\QQ_{\ge 0}.
$$ 
\end{construction}

\begin{remark}
\label{rem: intervalls self-intersection-no}
Using Remark~\ref{rem:selfintKstar},
we can express the length of the intervals 
$[\xi_i, \eta_k]$ from 
Construction~\ref{constr:intervalls}
via intersection numbers:
$$ 
\eta_k - \xi_i
\ = \ 
l_{i n_i}D_{i n_i}^2 
+ 
(l_{i n_i} - l_{k n_k})D_{i n_i}D_{k n_k}.
$$
Moreover, for any two $0 \le \iota, \kappa \le r$,
the (possibly empty) set $\Delta(\iota,\kappa)$
is explicitly given as 
$$
\Delta(\iota,\kappa)
\ = \
\left[\max
\left(
0,
\  
\frac{1}{l_{i n_i}(m_{i n_i-1}-m_{i n_i})},
\text{ where }
\ i \ne \iota, \ n_i \ge 2
\right),
\
-\frac{1}{l_{\kappa n_\kappa}m^-}
\right].
$$
\end{remark}

\begin{definition}
\label{def:ugamma}
Consider the defining matrix $P$ of
$X(A,P)$
and let $0 \le i_0,i_1 \le r$.
Set $e_0' := 0 \in \ZZ^{r+1}$
and $e_i' := e_i \in \ZZ^{r+1}$
for $i = 1, \ldots, r+1$.
Given $\gamma \in \QQ$, define
$$
u(i_0,i_1,\gamma)
\ := \
\gamma e_{r+1}'
- \frac{1}{l_{i_1n_{i_1}}} (e_{i_1}' - e_{i_0}')
- \gamma \sum_{i \ne i_0,r+1} m_{in_i} (e_{i}' - e_{i_0}')
\ \in \
\QQ^{r+1}.
$$
\end{definition}

\begin{lemma}
\label{lem:ugamma}
Let $u := u(i_0,i_1,\gamma)$ be as in
Definition~\ref{def:ugamma}.
Then the evaluation of~$u$ at a column
$v_{ij}$ of the matrix $P$ is given by
\begin{eqnarray*}
\bangle{u,v_{i_0j}}
& = &
l_{i_0j}
\left(
\gamma m_{i_0j}
+
\frac{1}{l_{i_1n_{i_1}}}
+
\gamma \sum_{i \ne i_0} m_{in_i}
\right),
\\[1ex]
\bangle{u,v_{i_1j}}
& = &
l_{i_1j}
\left(
\gamma m_{i_1j}
-
\frac{1}{l_{i_1n_{i_1}}}
-
\gamma m_{i_1n_{i_1}}
\right),
\\[1ex]
\bangle{u,v_{ij}}
& = &  
l_{ij}
\left(
\gamma m_{ij}
-
\gamma m_{in_i}
\right),
\quad i \ne i_0,i_1.      
\end{eqnarray*}
\end{lemma}

\begin{proof}
Set $e_0 := -e_1 - \ldots - e_{r} \in \ZZ^{r+1}$.
Then the evaluations of
$e_0' = 0 \in \ZZ^{r+1}$ and
$e_i' = e_i \in \ZZ^{r+1}$,
where $i = 0, \ldots, r+1$,
at $e_0, e_1, \ldots, e_{r+1}$ 
are
$$ 
\bangle{e_i',e_k}
\ = \ 
\begin{cases}
0,  & i = 0,
\\
-1, & 1 \le i \le r, \ k = 0,
\\
1,  & 1 \le i \le r, \ k = i,
\\
0,  & 1 \le i \le r, \ k \ne i,
\end{cases}
\qquad\qquad
\bangle{e_{r+1}', e_{k}}
\ = \
\delta_{r+1 \, k}.
$$
Here, as usual, we define
$\delta_{ik} := 1$ if $i = k$
and $\delta_{ik} := 0$ if $i \ne k$.
Consequently, for all
$0 \le i, k \le r$
with $i \ne i_0$
and $0 \le i, k \le r+1$
we obtain
$$
\bangle{e_i'-e_{i_0}', e_k}
\ = \
\begin{cases}
\delta_{ik}, & k \ne i_0,
\\
-1, & k = i_0.
\end{cases}
$$
This enables us to verify the assertion
by explicitly evaluating $u = u(i_0,i_1,\gamma)$
at the vectors $v_{ij} = l_{ij} e_i + d_{ij} e_{r+1}$.
\end{proof}

\begin{proposition}
\label{prop:hor-P-roots-gamma}
Assume that $X = X(A,P)$ has an
elliptic fixed point $x^- \in X$
and let $0 \le i_0,i_1 \le r$.
Then we have mutually inverse
bijections:
\begin{eqnarray*}
\left\{
\begin{array}{ll}  
\text{horizontal $P$-roots}
\\
\text{$u$ at $(x^-,i_0,i_1)$}
\end{array}  
\right\}
& \longleftrightarrow &
\left\{
\begin{array}{ll}  
\text{integers }
\gamma \in \Delta(i_0,i_1)
\text{ such}
\\ 
\text{that } 
\gamma d_{i_1n_{i_1}} \equiv -1 \, \mathrm{mod} \, l_{i_1n_{i_1}}
\end{array}
\right\}
\\
u & \mapsto & u_{r+1}
\\                
u(i_0,i_1,\gamma) & \mapsfrom & \gamma
\end{eqnarray*}
\end{proposition}

\begin{proof}
Given a horizontal $P$-root
$u$ at $(x^-,i_0,i_1)$, we
use Lemma~\ref{lem:ugamma}
to see $u = u(i_0,i_1,\gamma)$
for $\gamma := u_{r+1}$ by
comparing the values of $u$
and $u(i_0,i_1,\gamma)$
at the vectors $v_{in_i}$ for
$i \ne i_0$.
Now consider $0 \le i_0,i_1 \le r$
and any vector $u \in \ZZ^r \times \ZZ_{> 0}$.
Then we have 
$$
\bangle{u,v_{i_0 n_{i_0}}} \ge 0
\ \Leftrightarrow \
u_{r+1} \le \eta_{i_0}.
$$
Moreover, if $n_{i_1} > 1$, then
$$
\bangle{u,v_{i_1 n_{i_1}-1}} \ge 0
\ \Leftrightarrow \
u_{r+1} \ge \xi_{i_1}.
$$
Finally, if $i \ne i_0,i_1$, then
$$
\bangle{u,v_{i n_{i}-1}} \ge l_{in_i-1}
\ \Leftrightarrow \
u_{r+1} \ge \xi_{i}.
$$
So, the inequalities of
Definition~\ref{def:hor-P-root}~(ii)
are satisfied if and only if
$u_{r+1} \in \Delta(i_0,i_1)$ holds.
Thus, if $u$ is a horizontal
$P$-root $u$ at $(x^-,i_0,i_1)$,
then $u_{r+1} \in \Delta(i_0,i_1)$
and $u_{r+1} d_{i_1n_{i_1}} \equiv -1  \, \mathrm{mod} \, l_{i_1n_{i_1}}$
holds due to
$$
-1
\ = \
\bangle{u, v_{i_1n_{i_1}}}
\ = \
u_{i_1}l_{i_1n_{i_1}} + u_{r+1}d_{i_1n_{i_1}}.
$$
Conversely, given any 
$\gamma \in \Delta(i_0,i_1)$ with
$\gamma d_{i_1n_{i_1}} \equiv -1 \, \mathrm{mod} \, l_{i_1n_{i_1}}$,
we directly verify that the vector
$u(i_0,i_1,\gamma)$ is a horizontal $P$-root
at $(x^-,i_0,i_1)$
having $\gamma$ as its last coordinate.
\end{proof}

\begin{proposition}
\label{prop:roots-i0i1-switched}
Let $X = X(A,P)$ have an elliptic
fixed point $x^- \in X$ and
let~$u$ be a horizontal $P$-root
at $(x^-,i_0, i_1)$.
Then $x^- \in X$ is quasismooth
simple with leading indices
$i_0,i_1$.
\begin{enumerate}
\item
If $l_{i_0 n_{i_0}} \le l_{i_1 n_{i_1}}$
holds, then $x^-$ is smooth, we have
$\bangle{u,v_{i_0 n_{i_0}}} = 0$
and $u$ is the only horizontal $P$-root
at $(x^-,i_0, i_1)$.
\item
If the point $x^- \in X$ is singular,
then $l_{i_0 n_{i_0}} > l_{i_1 n_{i_1}}$ holds
and $i_0$ is the exceptional index
of $x^- \in X$.
\end{enumerate}
\end{proposition}

\begin{proof}
Proposition~\ref{prop:hor-P-root-constraints}
tells us that $x^- \in X$ is quasismooth
simple with leading indices
$i_0,i_1$.
Moreover, the second assertion is an immediate
consequence of the first one and
Proposition~\ref{prop:qsm-str-ell-char}.
Thus, we only have to prove the first assertion.
  
Suppose that there are two distinct
horizontal $P$-roots at $(x^-,i_0,i_1)$.
Then, by Proposition~\ref{prop:hor-P-roots-gamma}
they are given as $u(i_0,i_1,\gamma)$ and 
$u(i_0,i_1,\gamma')$ with positive integers
$\gamma, \gamma' \in \Delta(i_0,i_1)$
differing by a non-zero integral mutiple
of $l_{i_1n_{i_1}}$.
We conclude
$$
\eta_{i_1}
\ = \ 
-\frac{1}{l_{i_1n_{i_1}} m^-}
\ > \
l_{i_1n_{i_1}}
\ \ge \
l_{i_0n_{i_0}}.
$$
This implies $l^-m^- = l_{i_1n_{i_1}} l_{i_0n_{i_0}} m^- > -1$
which contradicts to Remark~\ref{rem:m+m-}.
Thus, there exists only one horizontal $P$-root
$u = u(i_0,i_1,\gamma)$ at $(x^-,i_0, i_1)$. 
We show $\bangle{u,v_{i_0 n_{i_0}}} = 0$.
Otherwise Lemma~\ref{lem:ugamma} yields 
$$
\bangle{u, v_{i_0 n_{i_0}}}
\ = \
l_{i_0 n_{i_0}}
\left( \gamma m^- + \frac{1}{l_{i_1 n_{i_1}}}
\right)
\ \ge \
1.
$$
This implies
$$
\gamma m^-
\ \ge \ 
\frac{1}{l_{i_0 n_{i_0}}} - \frac{1}{l_{i_1 n_{i_1}}}
\ \ge \
0.
$$
Again we arrive at a contradiction to Remark~\ref{rem:m+m-},
telling us $m^- < 0$.
Thus $\bangle{u,v_{i_0 n_{i_0}}} = 0$ holds.
According to Lemma~\ref{lem:ugamma} this forces
$$
\gamma
\ = \
- \frac{1}{l_{i_1 n_{i_1}}m^-}
\ = \
- \frac{l_{i_0 n_{i_0}}}{l^-m^-}.
$$
In particular, as $\gamma$ is an integer,
$l^-m^-$ divides $l_{i_0 n_{i_0}}$.
Moreover, making use of 
$\gamma d_{i_1n_{i_1}} \equiv -1 \, \mathrm{mod} \, l_{i_1n_{i_1}}$,
we obtain an integer 
$$
a
\ := \
\frac{1}{l_{i_1 n_{i_1}}} + \gamma m_{i_1 n_{i_1}}
\ = \ 
\frac{m^- -m_{i_1 n_{i_1}}}{l_{i_1 n_{i_1}}m^-}
\ = \ 
\frac{d_{i_0 n_{i_0}}}{l^- m^-}
+ \frac{l_{i_0 n_{i_0}}}{l^- m^-}
\left( \sum_{i\neq i_0,i_1} d_{i n_i} \right).
$$
Thus, $l^- m^-$ also divides $d_{i_0 n_{i_0}}$.
Since $l_{i_0 n_{i_0}}$ and $d_{i_0 n_{i_0}}$
are coprime, we arrive at $l^- m^- = -1$.
Proposition~\ref{prop:ell-q-smooth}
yields that $x^-$ is smooth.
\end{proof}

\begin{corollary}
\label{cor:all-hor-roots}
Consider $X = X(A,P)$ with a smooth elliptic
fixed point $x^-\in X$ and fix $0 \le i_0,i_1 \le r$
such that $l_{in_i} = 1$ for all $i \ne i_0,i_1$.
Let $\varepsilon \in \ZZ$ be maximal with
$$
l_{i_0n_{i_0}} - \varepsilon l_{i_1n_{i_1}}
\ \ge \
\frac{1}{l_{i n_i}(m_{i n_i-1}-m_{i n_i})}
\quad
\text{whenever } i \ne i_0 \text{ and } n_i \ge 2.
$$
For every integer $\mu$ with $0 \le \mu \le \varepsilon$,
set $u_\mu := u(i_0,i_1, l_{i_0n_{i_0}} - \mu l_{i_1n_{i_1}})$
according to Definition~\ref{def:ugamma}.
Then the following holds.
\begin{enumerate}
\item
For every $0 \le \mu \le \varepsilon$
the linear form $u_\mu$ 
is a horizontal $P$-root at $(x^-,i_0,i_1)$
and we have $\bangle{u_\mu,v_{i_0}} = \mu$.
\item
There exist horizontal $P$-roots at $(x^-,i_0,i_1)$
if and only if $\varepsilon \ge 0$ holds.
Moreover, $u_0, \ldots, u_\varepsilon$ are the only horizontal
$P$-roots at $(x^-,i_0,i_1)$.
\item
If $u$ is a horizontal $P$-root at $(x^-,i_0,i_1)$
then, for any two $0 \le \mu \le \alpha \le \varepsilon$,
we have $u_\mu = u_\alpha - (\alpha - \mu)u$.
\end{enumerate}
\end{corollary}

\begin{proof}
We check that
$\gamma := l_{i_0n_{i_0}} - \mu l_{i_1n_{i_1}}$
is as in Proposition~\ref{prop:hor-P-roots-gamma}
By the definition of $\varepsilon$,
we have $\gamma \ge \xi_i$
for all $i \ne i_0$.
Moreover, $l_{i_0 n_{i_0}} l_{i_1 n_{i_1}}m^- = -1$
by smoothness of $x^- \in X$ and
Proposition~\ref{prop:ell-q-smooth}~(iv).
Thus,
$$
\gamma
\ = \
l_{i_0n_{i_0}} - \mu l_{i_1n_{i_1}}
\ \le \
l_{i_0 n_{i_0}}
\ \le \
- \frac{1}{l_{i_1 n_{i_1}} m^-}
\ = \
\eta_{i_1}.
$$
Consequently, $\gamma \in \Delta(i_0,i_1)$.
Finally, also
$l_{i_0 n_{i_0}} d_{i_1 n_{i_1}} \equiv -1 \mod l_{i_1 n_{i_1}}$
holds due to Proposition~\ref{prop:ell-q-smooth}~(iv).
So, Proposition~\ref{prop:hor-P-roots-gamma} shows
that $u_\mu$ is a horizontal $P$-root at $(x^-,i_0,i_1)$
and, in addition yields~(ii).
Lemma~\ref{lem:ugamma} gives us
\begin{eqnarray*}
\bangle{u_\mu,v_{i_0 n_{i_0}}}
& = & 
l_{i_0 n_{i_0}}
\left(
(l_{i_0n_{i_0}} - 
\mu l_{i_1 n_{i_1}}) 
\sum_{i=0}^r m_{i n_i}
+ \frac{1}{l_{i_1 n_{i_1}}}
\right)
\\
& = & 
l_{i_0 n_{i_0}}
\left(
(l_{i_0n_{i_0}} - 
\mu l_{i_1 n_{i_1}}) 
\frac{-1}{l_{i_0 n_{i_0}}l_{i_1 n_{i_1}}}
+ \frac{1}{l_{i_1 n_{i_1}}}
\right)
\\
& = &
\mu.
\end{eqnarray*}
Concerning~(iii), there is only something
to show for $\varepsilon \ge 1$.
Then $l_{i_0n_{i_0}} \ge l_{i_1n_{i_1}}$ holds
and Proposition~\ref{prop:roots-i0i1-switched}
shows that $u$ is only horizontal $P$-root
at $(x^-,i_1,i_0)$.
Assertions~(i) and~(ii) just verified yield
$u = u(i_1,i_0,l_{i_1n_{i_1}})$.
Now the desired identity is directly
checked via Definition~\ref{def:ugamma}:
\begin{eqnarray*}
u_\mu - 
u_\alpha
& = &
u(i_0,i_1,l_{i_0 n_{i_0}}-\mu l_{i_1 n_{i_1}})
-
u(i_0,i_1,l_{i_0 n_{i_0}}-
\alpha l_{i_1 n_{i_1}})
\\
& = &  
(\alpha - \mu)l_{i_1 n_{i_1}}
e'_{r+1} 
-
(\alpha - \mu)l_{i_1 n_{i_1}}
\sum_{i\neq i_0, r+1}
m_{i n_i}(e'_i-e'_{i_0})
\\
& = & 
\left(\alpha - \mu\right)
\left(
l_{i_1 n_{i_1}}
e'_{r+1} 
+ 
l_{i_1 n_{i_1}}
\sum_{i \neq r+1} m_{i n_i}(e'_i - e'_{i_0})
\right)\\
& = & 
\left(\alpha - \mu\right) u.
\end{eqnarray*}
\end{proof}

\begin{definition}
Consider a rational projective $\KK^*$-surface 
$X = X(A,P)$ where~$P$ is irredundant.
\begin{enumerate}
\item
Assume that there is a parabolic fixed point curve
$D^+ \subseteq X$.
A \emph{vertical $P$-root} at $D^+$ 
is a vector $u \in \ZZ^{r} \times \ZZ_{< 0}$
such that 
$$
\bangle{u, v^+} \ = \ -1, 
\qquad
\bangle{u, v_{i1}} \ \ge \ 0,
\quad i = 0, \ldots, r.
$$ 
\item
Assume that there is a parabolic fixed point curve
$D^- \subseteq X$.
A \emph{vertical $P$-root} at $D^-$ 
is a vector $u \in \ZZ^{r} \times \ZZ_{> 0}$
such that 
$$
\bangle{u, v^-} \ = \ -1, 
\qquad
\bangle{u, v_{in_i}} \ \ge \ 0,
\quad i = 0, \ldots, r.
$$ 
\end{enumerate}
\end{definition}

\goodbreak

\begin{proposition}
\label{prop:vrootchar}
Consider a rational projective $\KK^*$-surface
$X(A,P)$, assume $P$ to be irredundant
and let $u \in \ZZ^{r+1}$.
\begin{enumerate}
\item
If there is a curve $D^+ \subseteq X$,
then the following statements are equivalent:
\begin{enumerate}
\item
$u$ is is a vertical $P$-root at $D^+$,
\item
$u_{r+1} = -1$ and $u_{i} \ge m_{i1}$ for all $i=0, \ldots, r$,
\item
$u_{r+1} = -1$ and $u_{i} \ge m_{ij}$ for all
$i=0, \ldots, r$, $j = 1,\ldots, n_i$.
\end{enumerate}
If one of the statements (a), (b) or~(c) holds,
then we have $(D^+)^2 \ge 0$.
\item
If there is a curve $D^- \subseteq X$,
then the following statements are equivalent:
\begin{enumerate}
\item
$u$ is is a vertical $P$-root at $D^-$,
\item
$u_{r+1} = 1$ and $u_{i} \le m_{in_i}$ for all $i=0, \ldots, r$,
\item
$u_{r+1} = 1$ and $u_{i} \le m_{ij}$ for all
$i=0, \ldots r$, $j = 1,\ldots, n_i$.
\end{enumerate}
If one of the statements (a), (b) or~(c) holds,
then we have $(D^-)^2 \ge 0$.
\end{enumerate}
In particular, $(u,k) \mapsto u$ defines a
one-to-one correspondence between the
vertical Demazure $P$-roots and the vertical
$P$-roots.
\end{proposition}

\begin{proof}
In each of the items, the equivalence of~(a) and~(b)
is clear by Remark~\ref{rem:u0} and the
equivalence of~(b) and~(c) holds due to slope orderedness.
The assertions on the self intersection numbers are clear
by the definition of vertical $P$-roots and
Remarks~\ref{rem:selfintKstar} and~\ref{rem:parselfint}.
\end{proof}

\begin{corollary}
\label{cor:conditions_vert_roots}
Let $X(A,P)$ be a $\KK^*$-surface,
assume $P$ to be irredundant,
let $u \in \ZZ^{r+1}$ and fix
$0 \le i_0 \le r$.
\begin{enumerate}
\item
Assume that there is a curve $D^+ \subseteq X$.
Then $u$ is a vertical $P$-root at $D^+$
if and only if
$$
\qquad
u_i \ \ge \ m_{i1}  \text{ for all } i \ne i_0,r+1,
\qquad
\sum_{i \ne i_0, r+1} u_i \ \le \ - m_{i_0 1}.
$$
\item
Assume that there is a curve $D^- \subseteq X$.
Then $u$ is a vertical $P$-root at $D^-$
if and only if
$$
\qquad
u_i \ \ge \ -m_{i n_i} \text{ for all } i \ne i_0, r+1,
\qquad
\sum_{i \ne i_0, r+1} u_i \ \le \  m_{i_0 n_{i_0}}.
$$
\end{enumerate}
\end{corollary}

\begin{proposition}
\label{prop:qsefp2novroots}
Let $X = X(A,P)$ be non-toric and~$P$
irredundant.
If there is a quasismooth simple elliptic
fixed point $x \in X$, then there are no
vertical $P$-roots.
\end{proposition}

\begin{proof}
We may assume $x = x^-$ having
leading indices $0,1$, exceptional
index $0$ and that $P$ is adapted
to the sink.
Suppose that $D^+ \subseteq X$
admits a vertical $P$-root
$u \in \ZZ^{r+1}$. Then
Proposition~\ref{prop:vrootchar} yields
$$
u_i \ge m_{i1} \text{ for } 1 \le i \le r,
\qquad 
-u_0 = u_1 + \ldots + u_r  \le - m_{01}.
$$
For $i = 2, \ldots, r$, we have $l_{in_i} = 1$
and thus $m_{in_i} = 0$, as $P$ is adapted to
the sink.
Irredundance of $P$ implies
$m_{i1} > 0$ and hence $u_i \ge 1$ for
$i = 2, \ldots, r$.
Using Proposition~\ref{prop:qsm-str-ell-char}~(ii)
for the inequality, we obtain
$$
m_{11} + (r-1) 
\ \le \ 
u_1 + \ldots + u_r 
\ \le \
- m_{01}
\ \le \
- m_{0 n_0}
\ \le \
m_{1 n_1} + \frac{1}{l_{1 n_1}}.
$$
We claim $r \le 1$.
For $l_{1n_1} \ge 2$ this
follows from $m_{11} \ge m_{1n_1}$.
If $l_{1n_1} = 1$, then $m_{1n_1} = 0$,
hence $m_{11} > 0$ by irredundance
and the claim follows.
Now, $r \le 1$ means that $X$ is toric,
which contradicts to our assumptions.
\end{proof}

\begin{proposition}
\label{prop:vroot+2novroot-}
Let $X = X(A,P)$ with $P$ irredundant
and assume that $X$ has 
fixed point curves~$D^+$ and~$D^-$.
If $D^+$ admits a vertical $P$-root,
then there is no vertical $P$-root at $D^-$.
\end{proposition}

\begin{proof}
We may assume that $P$ is adapted to the source.
Let $u^+ \in \ZZ^{r+1}$ be a vertical root at $D^+$.
Proposition~\ref{prop:vrootchar} yields
$$
u_i^+  \ge  m_{i1} > -1
\text{ for } 1 \le i \le r,
\qquad
u_0^+ = -u_1^+ - \ldots - u_r^+  \ge  m_{01}.
$$
We conclude $u_i^+ \ge 0 $ for $i = 1, \ldots, r$
and hence $m_{01} \le 0$.
Now suppose that there is a vertical $P$-root
$u^-$ at $D^-$.
Then
$$
u^-_i \ge -m_{i n_i} \ge - m_{i 1} \ge 0 \text{ for } 1 \le i \le r,
\qquad
0 \le -u^-_0 \le m_{0 n_0} \le m_{01} \le 0.
$$
Consequently $m_{01} =  m_{0 n_0} = 0$, which in turn
implies $n_0 = 1$ and $l_{0 1} = 1$.
This contradicts to the assumption that
$P$ is irredundant.
\end{proof}

%% file: gen-root-groups.tex
\section{Generating root groups}
\label{sec:generatingroots}

In this section, we provide suitable generators
for the unipotent part of the automorphism group
of a non-toric rational projective $\KK^*$-surface.

\begin{definition}
Consider $X = X(A,P)$. We denote by
$U(X) \subseteq \Aut(X)$ the subgroup
generated by all root groups of $X$.
\end{definition}

Note that $U(X) \subseteq \Aut(X)^0$ holds.
We have two cases. The first one is that~$U(X)$
is generated by the root groups
stemming from horizontal $P$-roots.
In this situation, we prove the following.

\begin{proposition}
\label{prop:gen-hor-P-roots}
Let $X = X(A,P)$ be non-toric with
horizontal $P$-roots.
Then there are a quasismooth
simple elliptic fixed point
$x \in X$ and $0 \le i_0,i_1 \le r$
such that~$U(X)$ is generated
by the root groups
arising from horizontal $P$-roots at
$(x,i_0,i_1)$ or $(x,i_1,i_0)$.
\end{proposition}

According to Proposition~\ref{prop:qsefp2novroots},
the remaining case is that $U(X)$
is generated by the root groups
given by the vertical roots.
Here we obtain the following.

\begin{proposition}
\label{prop:gen-ver-P-roots}
Let $X = X(A,P)$ be non-toric
with vertical $P$-roots at $D^+$
and let $0 \le i_0,i_1 \le r$.
Then $U(X)$ is generated by
the root groups arising
from vertical $P$-roots~$u$ at $D^+$
with $0 \le \bangle{u, v_{i1}} < l_{i1}$
for all $0 \le i \le r$ different from
$i_0,i_1$.
\end{proposition}

We begin with discussing the horizontal
case. 
First, we summarize the necessary
background.
By Proposition~\ref{prop:hor-P-roots-gamma},
all horizontal $P$-roots at~$x^-$
are of the form $u(i_0,i_1,\gamma)$.
According to Proposition~\ref{prop:P-root-to-P-dem},
each such $u(i_0,i_1,\gamma)$ defines
a Demazure $P$-root
in the sense of Definition~\ref{def:Pdemroot}:
$$
\tau(i_0,i_1,\gamma)
\ := \
(u,i_0,i_1,C^-),
\quad
u \ := \ u(i_0,i_1,\gamma),
\quad
C^- \ := \ (n_{0}, \ldots,n_{r}).
$$
Construction~\ref{constr:DEMLND} associates
with $\tau(i_0,i_1,\gamma)$
a locally nilpotent derivation on $R(A,P)$
which in turn gives rise to a root group
$$
\lambda_{\tau(i_0,i_1,\gamma)} \colon
\KK \ \to \ \Aut(X).
$$
Our statement involves the unique vectors
$\beta = \beta(A,i_0,i_1)$ in the row
space of the defining matrix $A$ having
$i_0$-th coordinate zero and $i_1$-th coordinate
one as introduced in Construction~\ref{constr:DEMLND}.
Moreover, the following will
be frequently used.

\begin{definition}
For the defining matrix $P$ of $X(A,P)$,
we denote by $I(P) \subseteq \{0, \ldots, r\}$
the set of all indices $i$ with $l_{in_i} = 1$.
\end{definition}

\begin{proposition}
\label{prop:hor-P-root-rel}
Let $X = X(A,P)$ be non-toric
with an elliptic fixed point $x^- \in X$.
Then we obtain the following relations
among the root subgroups associated
with horizontal $P$-roots at~$x^-$.
\begin{enumerate}
\item
Let $i_1,\iota_1 \in I(P)$ and
$0 \le i_0 \le r$.
If there are horizontal $P$-roots
$u(i_0,i_1,\gamma)$ and 
$u(i_0,\iota_1,\gamma) $, then, 
for every $s \in \KK$, we have
$$
\lambda_{\tau(i_0,i_1,\gamma)}(s)
\ = \ 
\lambda_{\tau(i_0,\iota_1,\gamma)}
(
\beta(A,i_0,\iota_1)_{i_1}^{-1} s
).
$$
\item
Let $i_0,\iota_0 \in I(P)$ and 
$0 \le i_1 \le r$.
If there are horizontal $P$-roots
$u(i_0,i_1,1)$, $u(\iota_0,i_1,1) $
and $u(i_1,\iota_0,\nu)$,
$\nu = 1, \ldots, l_{i_1 n_{i_1}}$,
then, for every $s \in \KK$, we have
$$
\qquad\qquad
\lambda_{\tau(i_0,i_1,1)}(s)
\ = \
\lambda_{\tau(\iota_0,i_1,1)}(s)
%\circ 
\prod_{\nu=1}^{l_{i_1 n_{i_1}}}
\lambda_{\tau(i_1,\iota_0,\nu)}
\left(
\beta(A,i_0,i_1)_{\iota_0}
{l_{i_1 n_{i_1}}\choose \nu}s^\nu
\right).
$$
\end{enumerate}
\end{proposition}

\begin{lemma}
\label{lem:betarel}
Consider the defining matrix $A$
of $X = X(A,P)$.
Then the vectors $\beta \in \KK^{r+1}$
introduced in Construction~\ref{constr:DEMLND}
satisfy
\begin{eqnarray*}
\beta(A,i_0,i_1)
& = &
\beta(A,i_0,i_1)_{i_1}^{-1}\beta(A,i_0,\iota_1),
\\
\beta(A,i_1,\iota_0)
& = &  
\beta(A,i_0,i_1)_{\iota_0}^{-1}
\left( \beta(A,i_0,i_1) - \beta(A,\iota_0,i_1) 
\right).
\end{eqnarray*}
\end{lemma}

\begin{proof}
The identities follow
from the fact
that $\beta(A,i_0,i_1)$ is the
unique vector in the row space of $A$ having
$i_0$-th coordinate zero and $i_1$-th
coordinate one.
\end{proof}

\begin{lemma}
\label{lem:ugammarel}
Consider the defining matrix $P$ of $X(A,P)$
and the linear forms $u(i_0,i_1,\gamma)$
from Definition~\ref{def:ugamma}.
Define
$$
u(i_0,i_1,\gamma)_{\nu,\iota}
\ := \
\nu u + e_{i_1}' - e_{\iota}'
\ \in \ZZ^{r+1},
\qquad
\nu = 1,\ldots, l_{i_1n_{i_1}},
\quad
\iota \ne i_0,i_1,
$$
as we did in Construction~\ref{constr:uvi}
in the case of Demazure $P$-roots.
Then, for any two indices
$i_1,\iota_1 \in I(P)$, we have 
$$
u(i_0,i_1,\gamma)_{1,\iota_1} 
\ = \
u(i_0,\iota_1,\gamma).
$$
Moreover, if $l^-m^- = -1$ and
there is an $i_1$ with $\iota \in I(P)$
for all $\iota \ne i_1$,
then, for any two $i_0,\iota_1 \in I(P)$,
we have
$$
u(i_0,i_1,1)_{\nu,\iota_1}
\ = \ 
u(i_1,\iota_1,\nu).
$$
\end{lemma}

\begin{proof}
For the first identity, observe that
we have $\nu = 1$.
Now, using the definition of $u(i_0,i_1,\gamma)$
and $l_{i_1 n_{i_1}} = l_{\iota_1 n_{\iota_1}} = 1$, we compute
\begin{eqnarray*}
u(i_0,i_1,\gamma)_{1,\iota_1}
& = &
\gamma e_{r+1}'
- (e_{i_1}' - e_{i_0}')
-\sum_{i \neq i_0} 
\gamma m_{i n_i} (e_i'-e_{i_0}') 
+ e_{i_1}' - e_{\iota_1}' 
\\
& = &
\gamma e_{r+1}'
- (e_{\iota_1}' - e_{i_0}')
-\sum_{i \neq i_0}
\gamma m_{i n_i} (e_i'-e_{i_0}')
\\
& = &
u(i_0,\iota_1,\gamma).
\end{eqnarray*}
We prove the second identity. Due to the assumptions,
we have $l_{i_1n_{i_1}}^{-1} = -m^-$.
Consequently, we obtain
\begin{eqnarray*}
u(i_0,i_1,1)_{\nu,\iota_1} 
& = & 
\nu
\left(
e_{r+1}'
+ m^-(e_{i_1}' - e_{i_0}')
-\sum_{i \neq i_0} m_{i n_i}
(e_i' -  e_{i_0}')
\right)
+
e_{i_1} - e_{\iota_1}   
\\
& = & 
\nu e_{r+1}' - (e_{\iota_1}' - e_{i_1}')
- \sum_{i \neq i_1} 
\nu m_{i n_i} (e_i' - e_{i_1}') \\
& = & 
u(i_1,\iota_1,\nu).
\end{eqnarray*}
\end{proof}

\begin{lemma}
\label{lem:hu-hzeta}
Consider the defining matrix $P$
of $X(A,P)$,
let $1 \le i_0,i_1,\iota_1 \le r$
with $i_1,\iota_1 \in I(P)$
and set $C^- := (n_0,\dots,n_r)$.
Then the monomials $h^u$ and $h^\zeta$ 
from Construction~\ref{constr:DEMLND}
satisfy
$$
\frac{h^{u(i_0,i_1,\gamma)}}
{h^{\zeta(i_0,i_1,C^-)}}
\ = \
\frac{h^{u(i_0,\iota_1,\gamma)}}
{h^{\zeta(i_0,\iota_1,C^-)}}.
$$
\end{lemma}

\begin{proof}
By Lemma~\ref{lem:ugammarel},
the linear form $u(i_0,i_1,\gamma)$
equals 
$u(i_0,\iota_1,\gamma)_{1,\iota_1}$.
From Construction~\ref{constr:uvi} we infer
how the latter evaluates and conclude
$$
\frac{h^{u(i_0,i_1,\gamma)}}
{h^{u(i_0,\iota_1,\gamma)}}
\ = \
\frac{T_{\iota_1}^{l_{\iota_1}}}
{T_{i_1}^{l_{i_1}}}
 \ = \ 
 \frac{\prod_{\iota \neq i_0,i_1,\iota_1}T_{\iota}^{l_{\iota}}
 \prod_{i\neq i_0} T_{i n_i}^{-1} 
 T_{\iota_1}^{l_{\iota_1}}}
{\prod_{\iota \neq i_0,i_1,\iota_1}T_{\iota}^{l_{\iota}}
 \prod_{i\neq i_0} T_{i n_i}^{-1} 
 T_{i_1}^{l_{i_1}}}
\ = \ 
\frac{h^{\zeta(i_0,i_1,C^-)}}
{h^{\zeta(i_0,\iota_1,C^-)}}.
$$
\end{proof}

\begin{proof}[Proof of Proposition~\ref{prop:hor-P-root-rel}]
We prove~(i).
It suffices to verify the corresponding relation
for the locally nilpotent derivations
associated with $\tau(i_0,i_1,\gamma)$
and $\tau(i_0,\iota_1,\gamma)$; see
Construction \ref{constr:DEMLND}.
Lemmas~\ref{lem:betarel} and~\ref{lem:hu-hzeta}
yield
$$
\beta(A,i_0,\iota_1)_{i_1}
\delta_{\tau(i_0,i_1,\gamma)}
\ = \
\delta_{\tau(i_0,\iota_1,\gamma)}.
$$
We turn to~(ii).
First consider the map $\varphi_{u(i_0,i_1,1)}(s)$
as given in Theorem~\ref{thm:restrrootauts}.
For the $\alpha(s,\nu,\iota)$ defined
there, we write
$$
\alpha_{i_0,i_1,\iota}
\ := \
\alpha_{i_0,i_1,\iota}(\nu,s)
\ := \
\beta(A,i_0,i_1)_\iota
{l_{i_1 n_{i_1}}\choose \nu}
s^\nu,
$$
which allows to specify in the
case of varying $i_0$ and $i_1$.
Now, using Lemma~\ref{lem:ugammarel}
and $\iota,i_1 \in I(P)$, we obtain%
\begin{eqnarray*}
\varphi_{u(i_0,i_1,1)}(s)
& = & 
\prod_{\iota \neq i_0,i_1}
\prod_{\nu=1}^{l_{i_1 n_{i_1}}}
\lambda_{u(i_0,i_1,1)_{\nu,\iota}}
(\alpha_{i_0,i_1,\iota}) 
\\
& = &  
\prod_{\nu=1}^{l_{i_1 n_{i_1}}}
\prod_{\iota \neq i_0,i_1}
\lambda_{u(i_1,\iota,\nu)}
(\alpha_{i_0,i_1,\iota}),
\\
\varphi_{u(\iota_0,i_1,1)}(s)^{-1}
& = & 
\prod_{\nu=1}^{l_{i_1 n_{i_1}}}
\prod_{\iota \neq \iota_0,i_1}
\lambda_{u(i_1,\iota,\nu)}
(-\alpha_{\iota_0,i_1,\iota}),
\end{eqnarray*}
where for the last equation,
we used
$\varphi_{u(1,\iota_0,i_1)}(s)^{-1}
= \varphi_{u(1,\iota_0,i_1)}(-s)$.
Next we observe 
$$
\alpha_{i_0,i_1,i_0} = 0,
\qquad
\alpha_{i_0,i_1,\iota} - 
\alpha_{\iota_0,i_1,\iota} 
\ = \ 
\beta(A,i_1,\iota_0)_{\iota}\beta(A,i_0,i_1)_{\iota_0}
{l_{i_1 n_{i_1}}\choose \nu}
s^\nu,
$$
where the second identity 
follows from Lemma~\ref{lem:betarel}.
With the aid of these considerations,
we compute%
\begin{eqnarray*}
\psi 
& := & %1
\varphi_{u(\iota_0,i_1,1)}(s)^{-1}\varphi_{u(i_0,i_1,1)}(s)
\\
& = & %2
\prod_{\nu=1}^{l_{i_1 n_{i_1}}}
\lambda_{u(i_1,\iota_0,\nu)}
(\alpha_{i_0,i_1,\iota_0})
\lambda_{u(i_1,i_0,\nu)}
(-\alpha_{\iota_0,i_1,i_0}) 
\\
& & 
\prod_{\iota \neq i_0,\iota_0,i_1}
\lambda_{u(i_1,\iota,\nu)}
(\alpha_{i_0,i_1,\iota}-\alpha_{\iota_0,i_1,\iota})
\\
& = & %3
\prod_{\nu=1}^{l_{i_1 n_{i_1}}}
\lambda_{u(i_1,\iota_0,\nu)}
(\alpha_{i_0,i_1,\iota_0})
\\ 
& & 
\prod_{\iota \neq \iota_0,i_1}
\lambda_{u(i_1,\iota,\nu)}
(\alpha_{i_0,i_1,\iota}
-\alpha_{\iota_0,i_1,\iota}) 
\\
& = & %4
\prod_{\nu=1}^{l_{i_1 n_{i_1}}}
\lambda_{u(i_1,\iota_0,\nu)}
\left(\beta(A,i_0,i_1)_{\iota_0}{l_{i_1 n_{i_1}}\choose \nu}s^\nu\right)
\\
& &  
\prod_{\iota \neq \iota_0,i_1}
\lambda_{u(i_1,\iota_0,\nu)_{1,\iota}}
\left(\beta(A,i_1,\iota_0)_{\iota}\beta(A,i_0,i_1)_{\iota_0}
{l_{i_1 n_{i_1}}\choose \nu}s^\nu\right)
\\
& = & %5
\prod_{\nu=1}^{l_{i_1 n_{i_1}}}
\lambda_{u(i_1,\iota_0,\nu)}
\left(\beta(A,i_0,i_1)_{\iota_0}{l_{i_1 n_{i_1}}\choose \nu}s^\nu\right)
\\
& &  
\prod_{\iota \neq \iota_0,i_1}
\lambda_{u(i_1,\iota_0,\nu)_{1,\iota}}
\left(
\alpha\left(
\beta(A,i_1,\iota_0)_{\iota}{l_{i_1 n_{i_1}}\choose \nu}s^\nu,1,\iota
\right)\right)
\\
& = & %6
\prod_{\nu=1}^{l_{i_1 n_{i_1}}}
\lambda_{\tau(i_1,\iota_0,\nu)}
\left(\beta(A,i_0,i_1)_{\iota_0} {l_{i_1 n_{i_1}}\choose \nu}s^\nu\right).
\end{eqnarray*}
Note that we used Theorem \ref{thm:restrrootauts}
for the last equality.
Thus, the right hand side of our equation is 
given as $\lambda_{\tau(\iota_0,i_1,1)}(s) \circ \psi$.
Using Theorem~\ref{thm:restrrootauts} we obtain
\begin{eqnarray*}
\lambda_{\tau(i_0,i_1,1)} (s)
& = &
\lambda_{u(\iota_0,i_1,1)}(s)
\varphi_{u(i_0,i_1,1)}(s)
\\
&  = & 
\lambda_{u(\iota_0,i_1,1)}(s)
\varphi_{u(\iota_0,i_1,1)}(s)
\psi
\\ 
& = &
\lambda_{\tau(\iota_0,i_1,1)}(s)
\psi.
\end{eqnarray*}
\end{proof}

\begin{definition}
Provided that the defining matrix~$P$ of $X = X(A,P)$
is adapted to the sink in the sense of
Definition~\ref{def:adpated}~(ii), we
say that~$P$ is \emph{normalized}
if $l_{0n_0} \ge \ldots \ge l_{rn_r}$
and for all $i < j$ with $l_{in_i} = l_{jn_j}$
and $n_i,n_j \ge 2$, we have
$$
m_{i n_i-1}-m_{i n_i}
\ \le \
m_{j n_j-1}-m_{j n_j}.
$$
\end{definition}

\begin{remark}
The above definition of normalized
coincides with the one given in
the introduction as 
Remark~\ref{rem:selfintKstar} ensures
$$
m_{i n_i-1}-m_{i n_i}
 \le 
m_{j n_j-1}-m_{j n_j}
\ \Leftrightarrow \
D_{in_i}^2 \le D_{jn_j}^2.
$$  
\end{remark}

\begin{lemma}
\label{lem: relations Delta(i0,i1)}
Let the defining matrix $P$ of
$X = X(A,P)$ be adapted to the sink
and normalized.
Consider the intervals $[\xi_i, \eta_k]$
and $\Delta(\iota,\kappa)$ from
Construction~\ref{constr:intervalls}.
\begin{enumerate}
\item
If $i, \iota \in I(P)$ satisfy $i \le \iota$,
then we have
$$
[\xi_i, \eta_k] \ \subseteq \ [\xi_\iota, \eta_k],
\qquad
\Delta(\iota,k) \ \subseteq \ \Delta(i,k).
$$
\item
For any two $k, \kappa \in I(P)$ and every $i = 0, \ldots, r$,
we have
$$
[\xi_i, \eta_k] \ = \ [\xi_i, \eta_\kappa],
\qquad
\Delta(i,k) \ = \ \Delta(i,\kappa).
$$
\item
Assume $l^-m^- = -1$.
Let $i$ with $\iota \in I(P)$
for all $\iota \ne i$
and $k \ge 2$.
Then 
$$
1 \in \Delta(k,i)
\quad \Rightarrow \quad 
\Delta(i,k) \cap \ZZ 
\ = \ [1,l_{i n_i}] \cap \ZZ.$$ 
\end{enumerate}
\end{lemma}

\begin{proof}
Irredundance of $P$ implies $n_i \ge 2$
for all $i \in I(P)$.
Consider $\iota,i,k \in I(P)$ with 
$i \le \iota$.
Then Construction~\ref{constr:intervalls}
and the fact that $P$ is normalized yield
$$
\xi_\iota
\ = \ 
\frac{1}
{m_{\iota n_\iota-1}-m_{\iota n_\iota}} 
\ \le \ 
\frac{1}{m_{i n_i-1} - m_{i n_\iota}} 
\ = \
\xi_i,
\qquad
\eta_{k}
\ = \
- \frac{1}{l_{k n_k} m^-}.
$$
This gives the first assertion.
The second one is obvious. 
For the third one, observe
$l_{1n_1} = \ldots = l_{rn_r} = 1$.
Thus $l^- = l_{0n_0}$ and 
$l_{0n_0}m^- = -1$.
We conclude
$$
\eta_\kappa
\ = \
- \frac{1}{l_{\iota n_\kappa}m^-}
\ = \
\frac{l_{0n_0}}{l_{\kappa n_\kappa}}
\ = \
\begin{cases}
l_{0n_0},  &  \kappa \ge 1,
\\
1,  &  \kappa = 0.
\end{cases}
$$
Now, let $1 \in \Delta(k,i)$.
Then, by the definition of
$\Delta(k,i)$, we have
$\xi_\iota \le 1$ for all
$\iota \ne k$.
Moreover, there is a
$\kappa \in \{0,1\} \cap I(P)$
with $\kappa \ne i$.
We claim
$$
\Delta(i,k) \cap \ZZ
\ = \
\Delta(i,\kappa) \cap \ZZ
\ = \
\bigcap_{\iota \ne i} [\xi_\iota,\eta_\kappa] \cap \ZZ
\ = \
[1, \eta_\kappa] \cap \ZZ
\ = \
[1, l_{i n_i}] \cap \ZZ.
$$
The first equality is due to~(ii).
The second one holds by
definition.
For the third one, use $k \ge 2$
to see $\xi_k \le \xi_\kappa \le 1$.
For the last equality, use
$l_{in_i}= 1$ for $\kappa = 0$
and 
$$
\kappa = 1, \, i = 0  \ \Rightarrow \ l_{in_i} = l_{0n_0},
\qquad\qquad
\kappa = 1, \, i \ge 2  \ \Rightarrow \ l_{in_i} = 1 = l_{0n_0}.
$$
\end{proof}

\begin{proof}[Proof of Proposition~\ref{prop:gen-hor-P-roots}]
Proposition~\ref{prop:hor-P-root-constraints}
tells us that there is a unique elliptic
fixed point $x \in X$ admitting horizontal $P$-roots.
We may assume that $x = x^-$ holds
and that $P$ is adapted to the sink
and normalized.
We claim that then $i_0 = 0$ and
$i_1 = 1$ are as wanted.
So, given any horizontal $P$-root 
$u(\iota_0,\iota_1,\gamma)$
at $(x,\iota_0,\iota_1)$,
the task is to show that the
associated root group maps
into the subgroup
generated by all the root subgroups
arising from horizontal $P$-roots
at $(x,0,1)$ and $(x,1,0)$.

\medskip
\noindent
\emph{Let $\iota_0, \iota_1 \ne 0$}.
Then we have $l_{0 n_0} = \ldots = l_{r n_r} = 1$.
Proposition~\ref{prop:roots-i0i1-switched}~(i)
says that $x^-$ is smooth and
$\bangle{u,v_{\iota_0 n_{\iota_0}}} = 0$ holds.
Thus, Proposition~\ref{prop:ell-q-smooth}
yields $m^- = -1$ and Lemma~\ref{lem:ugamma}
shows $\gamma = 1$.
By Lemma~\ref{lem: relations Delta(i0,i1)}~(ii)
we have $1 \in \Delta(\iota_0,\iota_1) = \Delta(\iota_0,0)$.
Hence, there is a horizontal $P$-root $u(\iota_0,0,1)$.
Proposition~\ref{prop:hor-P-root-rel}~(i)
implies 
$$
\lambda_{\tau(\iota_0,\iota_1,1)}(\KK) 
\ = \
\lambda_{\tau(\iota_0,0,1)}(\KK). 
$$

\medskip
\noindent
\emph{Let $\iota_0 = 0$ and $\iota_1 \ne 1$}.
Then we have $l_{1n_1} = l_{\iota_1 n_{\iota_1}} = 1$.
Moreover, using 
Lemma~\ref{lem: relations Delta(i0,i1)}~(i),
we see
$\gamma \in \Delta(0,\iota_1) = \Delta(0,1)$.
Thus, Proposition~\ref{prop:hor-P-root-rel}~(i)
applies and we obtain
$$
\lambda_{\tau(0,\iota_1,\gamma)}(\KK) 
\ = \
\lambda_{\tau(0,1,\gamma)}(\KK). 
$$

\medskip
\noindent
\emph{Let $\iota_1 = 0$ and $\iota_0 \ne 1$.}
Then we have $l_{1 n_1} = 1$.
Proposition~\ref{prop:roots-i0i1-switched}~(i)
says that $x^-$ is smooth and that
$\bangle{u,v_{\iota_0n_{\iota_0}}} = 0$ holds.
According to Lemma~\ref{lem:ugamma},
the latter means
$$
0
\ = \
\gamma m^- + \frac{1}{l_{\iota_1 n_{\iota_1}}}
\ = \
\gamma m^- + \frac{1}{l_{0n_0}}
\ = \
-\frac{\gamma}{l_{0n_0}} + \frac{1}{l_{0n_0}}.
$$
where the last equality is due to
Proposition~\ref{prop:ell-q-smooth},
showing $l_{0n_0}m^- = l^-m^- = -1$.
We conclude $\gamma = 1$.
Proposition~\ref{prop:hor-P-roots-gamma}
gives $1 \in \Delta(\iota_0,0)$. 
Now, Lemma~\ref{lem: relations Delta(i0,i1)}~(i)
shows that $1 \in \Delta(\iota_0,0)=\Delta(1,0)$,
hence there is a horizontal $P$-root $u(1,0,1)$.
Furthermore, Lemma~\ref{lem: relations Delta(i0,i1)}~(iii)
shows
$$
\Delta(0,1) \cap \ZZ 
\ = \
[1,l_{0 n_0}] \cap \ZZ.
$$
Consequently, there is a a horizontal $P$-root
$u(0,1,\nu)$ for every $1 \le \nu \le l_{0 n_0}$.
Now Proposition~\ref{prop:hor-P-root-rel}
tells us
$$
\lambda_{\tau(\iota_0,0,1)}(\KK)
\ \subseteq \
\lambda_{\tau(1,0,1)}(\KK)
\prod_{\nu = 1}^{l_{0 n_0}}\lambda_{\tau(0,1,\nu)}(\KK).
$$
\end{proof}

We enter the vertical case.
According to Proposition~\ref{prop:vrootchar},
every vertical $P$-root $u$ corresponds to 
a vertical Demazure $P$-root $\kappa = (u,i)$
and via the associated locally nilpotent derivation
of $R(A,P)$ we obtain the root group
$$
\lambda_\kappa = \lambda_{u}
\colon \KK \ \to \ \Aut(X).
$$

\begin{lemma}
\label{lem:gijk}
Let $A\in \mathrm{Mat}(2,r+1;\KK)$ and 
$g_{i_1,i_2,i_3}$ be as in 
Construction~\ref{constr:RAP} and 
$\beta(i_1,i_2,A)$, $\beta(i_2,i_1,A)$ 
as in Construction~\ref{constr:DEMLND}.
Then there is a $b_{i_1,i_2,i_3} \in \KK^*$ with 
$$
b_{i_1,i_2,i_3} g_{i_1,i_2,i_3} 
\ = \ 
T_{i_3}^{l_{i_3}} - \beta(i_1,i_2,A)_{i_3} T_{i_1}^{l_{i_1}} 
- \beta(i_2,i_1,A)_{i_3} T_{i_2}^{l_{i_2}}.
$$
\end{lemma}

\begin{proof}
Consider $A' = [a_{i_1},a_{i_2},a_{i_3}] \in \Mat(2,3,\KK)$.
As a direct computation shows, we have
$$
B \cdot A'
\ = \  
\left[
\begin{array}{rrr}
1 & 0 & \beta(A,i_2,i_1)_{i_3}
\\
0 & 1 & \beta(A,i_1,i_2)_{i_3}
\end{array}
\right]
$$
with a unique matrix $B \in \GL(2,\KK)$.
Setting $b_{i_1,i_2,i_3} := \mathrm{det}(B)^{-1}$,
we infer the assertion from
$$
g_{i_1,i_2,i_3}
\ = \
b_{i_1,i_2,i_3}
\det 
\left[
\begin{array}{ccc}
T_{i_1}^{l_{i_1}} & T_{i_2}^{l_{i_2}} & T_{i_2}^{l_{i_2}}
\\
1 & 0 & \beta(A,i_2,i_1)_{i_3}
\\
0 & 1 & \beta(A,i_1,i_2)_{i_3} 
\end{array}
\right].
$$
\end{proof}

\begin{lemma}
\label{lem:vert-root-rel}
Consider an $X = X(A,P)$ with a curve $D^+ \subseteq X$.
Fix $u \in \ZZ^{r+1}$ and $0 \le \iota \le r$.
For any $0 \le i_0 \le r$ with $i_0 \ne \iota$ set
$$
u_{\iota,i_0} \ := \ u  + e'_{i_0} - e'_\iota,
\qquad
e_0' \ := \ 0,
\qquad
e_i' \ := \ e_i, \quad i = 0, \ldots, r.
$$
Now assume that $u$ is a vertical $P$-root at $D^+$
with $\bangle{u,v_{\iota 1}} \ge l_{\iota 1}$
and let $0 \le i_0,i_1 \le r$ with $i_0 \ne i_1$.
Then $u_{\iota,i_0}$ and $u_{\iota,i_1}$ are vertical
$P$-roots at $D^+$ and we have
$$
\lambda_{u}(s)
\ = \ 
\lambda_{u_{\iota, i_0}}(\beta(A,i_1,i_0)_\iota s) 
\lambda_{u_{\iota, i_1}}(\beta(A,i_0,i_1)_\iota s).
$$
\end{lemma}

\begin{proof}
First let us see how $u_{\iota,i_0}$ evaluates
on the vectors $v_{ij}$, $v^+$ and $v^-$, if present.
We have  
$$
\bangle{u_{\iota,i_0},v_{ij}}
\ = \
\begin{cases}
\bangle{u,v_{ij}}, & i \ne \iota, i_0,
\\  
\bangle{u,v_{\iota j}} - l_{\iota j}, & i = \iota,
\\  
\bangle{u,v_{i_0 j}} + l_{i_0 j},  & i = i_0,
\end{cases}
\qquad
\bangle{u_{\iota,i_0},v^\pm} \ = \ \mp 1.
$$
In particular, we see that
$u_{\iota,i_0}$ is a vertical $P$-root at $D^+$:
using Remark~\ref{rem:u0} and slope-orderedness of $P$,
we infer $\bangle{u_{\iota,i_0},v_{\iota j}} \ge 0$
for any $j = 1, \ldots, n_\iota$ from 
$$
\bangle{u,v_{\iota 1}} \ge l_{\iota 1}
\ \Rightarrow \
u_\iota \ge 1 + m_{\iota 1}
\ \Rightarrow \
u_\iota \ge 1 + m_{\iota j}
\ \Rightarrow \
\bangle{u,v_{\iota j}} \ge l_{\iota j} .
$$
Moreover, we observe that the locally nilpotent
derivation $\delta_u$ provided by Construction~\ref{constr:DEMLND}
gives us a polynomial
$$
f
\ := \
\delta_{u}(S_1)T_\iota^{-l_\iota}
\ = \
S_1T_\iota^{-l_\iota} \prod_{i,j} T_{ij}^{\bangle{u,v_{ij}}}
\ \in \
\KK[T_{ij},S_k].
$$
Next we claim that the locally nilpotent derivation $\delta_{u}$
on $R(A,P)$ coincides with
$\beta(A,i_1,i_0)_\iota \delta_{u_{\iota,i_0}}
+
\beta(A,i_0,i_1)_\iota \delta_{u_{\iota,i_1}}$.
Indeed, we compute
\begin{eqnarray*}
\delta_{u}(S_1) - fb_{i_0,i_1,\iota}g_{i_0,i_1,\iota}
& = &
\delta_{u}(S_1) - f\left(T_\iota^{l_\iota} 
- \beta(A,i_1,i_0)_\iota T_{i_0}^{l_{i_0}} 
- \beta(A,i_0,i_1)_\iota T_{i_1}^{l_{i_1}}\right) 
\\
& = & 
\beta(A,i_1,i_0)_\iota f T_{i_0}^{l_{i_0}} 
+ \beta(A,{i_0},{i_1})_\iota f T_{i_1}^{l_{i_1}}
\\
& = & 
\beta(A,i_1,i_0)_\iota \delta_{u_{\iota,i_0}}(S_1) 
+\beta(A,i_0,i_1)_\iota \delta_{u_{\iota,i_1}}(S_1),
\end{eqnarray*}
using Lemma~\ref{lem:gijk}.
Computing the associated root 
groups according to~Proposition~\ref{prop:comm-der}~(i)
gives the assertion.
\end{proof}

\begin{proof}[Proof of Proposition~\ref{prop:gen-ver-P-roots}]
According to Propositions~\ref{prop:qsefp2novroots}
and~\ref{prop:vroot+2novroot-}, the
group $U(X)$ is generated by  the
rout groups arising
from vertical $P$-roots at $D^+$.
Given a vertical $P$-root $u$ 
with $\bangle{u,v_{\iota 1}} \ge l_{\iota 1}$,
take any two distinct $0 \le i_0,i_1 \le r$
differing from $\iota$.
Then Lemma~\ref{lem:vert-root-rel} tells us 
$$
\lambda_u(\KK) 
\ \subseteq \
\lambda_{u_{\iota,i_0}}(\KK) 
\lambda_{u_{\iota,i_1}}(\KK).
$$
Recall that the evaluations of
the linear forms 
$u_{\iota,i_0}$ and $u_{\iota,i_1}$
at the vectors $v_{ij}$ are given by
$$
\bangle{u_{\iota,i_0},v_{ij}}
 = 
\begin{cases}
\bangle{u,v_{ij}}, & i \ne \iota, i_0,
\\  
\bangle{u,v_{\iota j}} - l_{\iota j}, & i = \iota,
\\  
\bangle{u,v_{i_0 j}} + l_{i_0 j},  & i = i_0,
\end{cases}
\quad
\bangle{u_{\iota,i_1},v_{ij}}
 = 
\begin{cases}
\bangle{u,v_{ij}}, & i \ne \iota, i_1,
\\  
\bangle{u,v_{\iota j}} - l_{\iota j}, & i = \iota,
\\  
\bangle{u,v_{i_1 j}} + l_{i_0 j},  & i = i_1.
\end{cases}
$$
Thus, the automorphism $\lambda_{u}(s)$
can be expressed as a composition of automorphisms
stemming from vertical $P$-roots evaluating
strictly smaller at $v_{\iota 1}$ and equal to $u$
at all other $v_{i1}$ with $i \ne i_0,i_1$.
Suitably iterating this process, we arrive at
the assertion.
\end{proof}

\begin{definition}
Consider the defining matrix $P$ of $X(A,P)$
and let $0 \le i_0,i_1 \le r$.
Define an interval
$$
\Gamma(i_0,i_1)
\ := \
\Bigl[
m_{i_1 1},
\
-m_{i_0 1} 
- \sum_{i \ne i_0,i_1} 
\lceil m_{i1} \rceil
\Bigr]
\ \subseteq \
\QQ.
$$
Moreover, denote $e_0' := 0 \in \ZZ^{r+1}$
and $e_i' := e_i \in \ZZ^{r+1}$ for $i = 1, \ldots, r+1$.
Given $\alpha \in \QQ$, define
$$
u(i_0,i_1,\alpha)
\ := \
- e_{r+1}'
+ \alpha (e_{i_1}' - e_{i_0}')
+
\sum_{i \ne i_0,i_1,r+1} \lceil m_{i1} \rceil (e_i'  - e_{i_0}')
\ \in \
\QQ^{r+1}.
$$
\end{definition}

\begin{proposition}
\label{prop:ver-P-roots-gamma}
Assume that $X = X(A,P)$ has a
parabolic fixed point curve $D^+ \subseteq X$.
Then we have mutually inverse bijections
\begin{eqnarray*}
\left\{
\begin{array}{ll}  
\text{vertical $P$-roots $u$ at $D^+$ such that}
\\
\text{$0 \le \bangle{u,v_{i1}} < l_{i_1}$ for all $i \ne i_0,i_1$}
\end{array}  
\right\}
& \longleftrightarrow &
\Gamma(i_0,i_1) \cap \ZZ
\\
u & \mapsto & u_{i_1}
\\                
u(i_0,i_1,\alpha) & \mapsfrom & \alpha
\end{eqnarray*}
\end{proposition}

\begin{proof}
First we consider any vertical $P$-root $u$ at $D^+$.
Let $u_0 = -u_1 - \ldots - u_r$ as in Remark~\ref{rem:u0}
and set $\varepsilon_i := u_i-m_{i1}$.
Using Proposition~\ref{prop:vrootchar} we obtain
$$
m_{i1}
\ \le \
\bangle{u,v_{i1}}
\ = \
u_il_{i1} - d_{i1}
\ = \
\varepsilon_il_{i1} .
$$
Now let $u$ stem from the left hand side set above.
Then we must have $0 \le \varepsilon_i < 1$ and hence
$u_i = \lceil m_{i1} \rceil$ for all $i \ne i_0,i_1$.
Corollary~\ref{cor:conditions_vert_roots} yields
$$
m_{i_11}
\ \le \
u_{i_1}
\ \le \
-m_{i_01}
-
\sum_{i \ne i_0,i_1} u_i
\ = \
-m_{i_01}
-
\sum_{i \ne i_0,i_1}  \lceil m_{i1} \rceil.
$$
One directly checks that any $\alpha \in \Gamma(i_0,i_1) \cap \ZZ$
delivers an $u(i_0,i_1\alpha)$ in the left
hand side set and the assignments are inverse
to each other.
\end{proof}

%% file: equiv-resolution.tex
\section{Root groups and resolution of singularities}
\label{sec:equiv-resolution}

In this section, we show how to lift the 
root groups arising from the horizontal or 
vertical $P$-roots of $X = X(A,P)$ with respect 
to the minimal resolution of singularities 
$\pi \colon \tilde X \to X$.
The following theorem gathers the essential 
results; observe that items~(iii) and~(iv)
are as well direct consequences of 
the general existence of a functorial 
resolution in characteristic zero,
whereas~(i) and~(ii), used later,
are more specific.

\begin{theorem}
\label{thm:equiv-resolution}
Consider $X = X(A,P)$ and its minimal 
resolution $\pi \colon \tilde X \to X$,
where $\tilde X =  X(A,\tilde{P})$.
\begin{enumerate}
\item
There is a natural bijection 
$\lambda \mapsto \tilde \lambda$
between the root 
groups of $X$ and those of $\tilde X$,
made concrete in terms of defining data 
in Propositions~\ref{prop:hor-P-root-lift} 
and~\ref{prop:vert-P-root-lift}.
\item 
For every root group 
$\lambda \colon \KK \to \Aut(X)$ 
and every $s\in \KK$ we have a 
commutative diagram
$$ 
\xymatrix{
{\tilde X}
\ar[r]^{\tilde \lambda(s)}
\ar[d]_{\pi}
&
{\tilde X}
\ar[d]^{\pi}
\\
X
\ar[r]_{\lambda(s)}
&
X}
$$
\item 
The isomorphism $\pi \colon \pi^{-1}(X_{\reg}) \to X_\reg$ 
gives rise to a canonical isomorphism of groups 
$\Aut(X)^0 \cong \Aut(\tilde X)^0$.
\item
Every action $G \times X \to X$ of a connected algebraic 
group $G$ lifts to an action $G \times \tilde X \to \tilde X$. 
\end{enumerate}  
\end{theorem}

The proof of Theorem~\ref{thm:equiv-resolution}
essentially relies on the preceding results
showing that we either have only horizontal roots
at a common simple quasismooth elliptic fixed
point or there are only vertical
roots at a common parabolic fixed point curve.
This allows us to relate the resolution of
singularities closely to resolving toric surface
singularities.
We begin the preparing discussion with a brief
reminder on Hilbert bases of two-dimensional
cones and then enter the horizontal case.

\begin{remark}
\label{rem:hilbertbasis}
Consider two primitive vectors $v_0$ 
and $v_1$ in $\ZZ^2$.
Assume $\det(v_0,v_1)$ to be positive.
Set $v_0' := v_0$ and let
$v_1' \in \ZZ^2$ be the unique vector
with 
$$
v_1' \in \cone(v_0', v_1),
\qquad
\det(v_0',v_1') = 1,
\qquad
0 \le \det(v_1',v_1) < \det(v_0',v_1).  
$$
Iterating gives us a finite sequence  
$v_0 = v_0', v_1', \ldots, v_q', v_{q+1}' = v_2$,
the \emph{Hilbert basis} $\mathcal{H}(\sigma)$ 
of $\sigma = \cone(v_0,v_1)$ in $\ZZ^2$.
We have 
$$
\det(v_i',v_{i+1}') \ = \ 1,
\qquad\qquad
v_{i-1}' + v_{i+1}' \ = \ c_i v_i'
$$
with unique integers $c_1, \ldots, c_{q} \in \ZZ_{\ge 2}$.
Subdividing $\sigma$ along the Hilbert basis
gives us the fan of the minimal resolution
of the affine toric surface $Z_\sigma$.
\end{remark}

\begin{construction}
\label{constr:qsmellresolve}
Assume that $X = X(A,P)$ has a quasismooth elliptic fixed
point $x^- \in X$ with leading indices $i_0, i_1$.
Consider
$$
v_0 \ := \ (-l_{i_0n_0}, d_{i_0n_0}),
\qquad
v_1 \ := \ (l_{i_1n_{i_1}}, d_{i_1n_{i_1}}),
\qquad
\sigma \ := \ \cone(v_0,v_1).
$$
As earlier, let $e_0 = -e_1 - \ldots - e_r$, where $e_i \in \ZZ^{r+1}$
are the canonical basis vectors.
Then, with every $v' = (l', d') \in \mathcal{H}(\sigma)$,
we associate $\tilde v \in \ZZ^{r+1}$ by
$$
\tilde v
\ := \ 
\begin{cases}
- l' e_{i_0 n_{i_0}} + d' e_{r+1}, & l' < 0,
\\
l' e_{i_1 n_{i_1}} + d' e_{r+1}, & l' > 0,
\\
-e_{r+1}, &  l' = 0.
\end{cases}
$$
Inserting the columns $\tilde v$, 
where $v_0,v_1 \ne v' \in \mathcal{H}(\sigma)$,
at suitable places of $P$ produces a slope 
ordered defining matrix $P'$.
\end{construction}

\begin{proposition}
\label{prop:resolve}
Let $A,P$ and $P'$ be as in~\ref{constr:qsmellresolve}.
Consider $X' := X(A,P')$ and the natural morphism 
$\pi' \colon X' \to X$.
Then $\pi'$ is an isomorphism over $X \setminus \{x^-\}$,
each $x' \in X'$ over $x^- \in X$ is smooth and 
the minimal resolution $\tilde X \to X$
factors as
$$
\xymatrix{
{\tilde X}
\ar[rr]
\ar[dr]_{\pi}
&
&
X'
\ar[dl]^{\pi'}
\\
&
X
&
}
$$
\end{proposition}

\begin{proof}
We may assume that $P'$ is adapted to the sink.
Remark~\ref{rem:m+m-} ensures $\det(v_0,v_1) > 0$. 
Let $v_0', \ldots, v_{q+1}'$ be the
members of $\mathcal{H}(\sigma)$, 
constucted as 
in Remark~\ref{rem:hilbertbasis}.
Write $v_i' = (l_i', d_i')$.
There are unique integers
$1 \le k^- < k^+ \le q$ with 
$$
l_i' < 0 \text{ for } i = 0, \ldots, k^-,
\qquad\qquad
l_i' > 0 \text{ for } i = k^+ , \ldots, q+1,
$$
where $k^+ = k^-+1$ if all $l_i'$ differ
from zero and otherwise we have $k^+ = k^-+2$
and $l_{k^0}' = 0$ for $k^0 = k^- + 1$.
The curve of $X'$ corresponding to a column
$\tilde v_i$ of $P'$ lies in $\mathcal{A}_{i_0}$ 
if $l_i' < 0$, in $\mathcal{A}_{i_1}$ if
$l_i' > 0$ and equals $D^-$ if $i = k^0$.

We verify smoothness of the points $x' \in X'$ 
lying over $x^- \in X$.
First consider the case that $x' \in X'$
is a parabolic or a hyperbolic fixed point. 
According to Remark~\ref{rem:hilbertbasis} 
we have 
$$ 
\det(v_i',v_{i+1}') \ = \ 1.
$$
This gives us precisely the smoothness conditions
from Propositions~\ref{prop:smooth-par-li1}~(ii)
and~\ref{prop:hyp-q-smooth}.
If  $x' \in X'$ is an elliptic fixed point, 
then we can apply
Proposition~\ref{prop:ell-q-smooth} to
obtain smoothness of $x'$ in terms of $P'$:
$$
\tilde l^- \tilde m^-
=
\tilde l_{i_0 \tilde n_{i_0}} \tilde d_{i_1 \tilde n_{i1}}
+
\tilde l_{i_1 \tilde n_{i_1}} \tilde d_{i_0 \tilde n_{i0}}
=
-l_{k^-}'d_{k^ +}' + l_{k^+}'d_{k^-}'
=
-\det(v_{k^-},v_{k^+})
= 
-1.
$$

We show minimality of $X' \to X$.
The conditions 
$v_{i-1}' + v_{i+1}' = c_i v_i'$ 
from Remark~\ref{rem:hilbertbasis}
translate to the equations  
Lemma~\ref{lem:intnosDij} 
for the exceptional curves $E_i$ of $X' \to X$
corresponding to $v_i'$.
We conclude $E_i^2 \le -2$.
\end{proof}

\begin{proposition}
\label{prop:hor-P-root-lift}
Consider the minimal resolution
$\pi \colon \tilde X \to X$
of $\KK^*$-surfaces defined by
$(A,P)$ and $(A,\tilde P)$.
Assume that there is a quasismooth
simple elliptic fixed point $x^- \in X$
and let $\tilde x^- \in \pi^{-1}(x^-)$ be
the corresponding elliptic fixed point.
Given $0 \le i_0, i_1 \le r$ and $u \in \ZZ^{r+1}$,
the following statements are equivalent.
\begin{enumerate}
\item
The linear form $u \in \ZZ^{r+1}$ is a horizontal
$P$-root at $(x^-,i_0,i_1)$.  
\item
The linear form $u \in \ZZ^{r+1}$ is a horizontal
$\tilde P$-root at $(\tilde x^-,i_0,i_1)$.  
\end{enumerate}
\end{proposition}

\begin{proof}
As $x^- \in X$ is a simple elliptic fixed point
with exceptional index $i_0$,
we have $\tilde v_{i \tilde n_i} = v_{in_i}$
for all $i \ne i_0$.
The fiber of $\pi \colon \tilde X \to X$
over $x^- \in X$ is of the form
$$
\pi^{-1}(x^-)
\ = \
\tilde D_{i_0 \tilde n_{i_0}-q-1}
\cup \ldots \cup
\tilde D_{i_0 \tilde n_{i_0}}
$$
By Proposition~\ref{prop:resolve},
the corresponding columns
$\tilde v_{i_0 \tilde n_{i_0}-q-1}, \ldots, \tilde v_{i_0 \tilde n_{i_0}}$
of $\tilde P$ are obtained by running
Construction~\ref{constr:qsmellresolve}
with the initial data
$$
v_0
\ := \
(-l_{i_0n_{i_0}},d_{i_0n_{i_0}}),
\qquad\qquad
v_1
\ := \
(l_{i_1n_{i_1}},d_{i_1n_{i_1}}).
$$
In the notation of Remark~\ref{rem:hilbertbasis},
the vector $\tilde v_{i_0 \tilde n_{i_0}}$
stems from the penultimate Hilbert basis member
$v_q' \in \mathcal{H}(\sigma)$
of $\sigma = \cone(v_0,v_1)$ which is determined
by the conditions
$$
\det(v_q',v_1) \ = \ 1,
\qquad\qquad
0 \ \le \ \det(v_0,v_q') \ <  \ \det(v_0,v_1).
$$
In terms of $\tilde P$, we have
$v_q' = (- \tilde l_{i_0 \tilde n_{i_0}},\tilde d_{i_0 \tilde n_{i_0}})$.
Together with the definitions of $v_0$ and $v_1$
this gives us 
$$
\tilde l_{i_0 \tilde n_{i_0}} d_{i_1n_{i_1}} \ \equiv \ -1  \, \mathrm{mod} \, l_{i_1n_{i_1}},
\qquad
- \frac{1}{l_{i_0n_{i_0}} m^-} - l_{i_1n_{i_1}}
\ < \ 
\tilde l_{i_0 \tilde n_{i_0}}
\ \le \
- \frac{1}{l_{i_0n_{i_0}} m^-},
$$
where the estimate is obtained by resolving the first
characterizing condition of $v_q'$
for $\tilde d_{i_0 \tilde n_{i_0}}$
and plugging the result into the second one.
Next look at
$$
v_0
\ := \
(l_{in_i},d_{in_i}),
\qquad\qquad
v_1
\ := \
(l_{in_i-1},d_{in_i-1})
$$
in case $n_i \ge 1$.
Then, according to Remark~\ref{rem:hilbertbasis},
the Hilbert basis member
$v_1' \in \mathcal{H}(\sigma)$
of $\sigma = \cone(v_0,v_1)$
is characterized by the conditions
$$
\det(v_0,v_1') \ = \ 1,
\qquad\qquad
0 \ \le \ \det(v_1',v_1) \ <  \ \det(v_0,v_1).
$$
Similarly as before, we have
$v_1' = (\tilde l_{\tilde n_i-1},\tilde d_{\tilde n_i-1})$
and making the above conditions explicit,
we arrive at
$$
\tilde l_{\tilde n_i-1} d_{in_i} \equiv  -1 \, \mathrm{mod} \, l_{in_i},
\quad
\frac{1}{l_{in_i} (m_{n_i-1}-m_{n_i})} 
 \le
\tilde l_{\tilde n_i-1}
 < 
\frac{1}{l_{in_i} (m_{n_i-1}-m_{n_i})} + l_{in_i}.
$$
Now consider a horizontal $P$-root $u \in \ZZ^{r+1}$
at $(i_0,i_1,x^-)$. 
Proposition~\ref{prop:hor-P-roots-gamma} yields
$u = u(i_0,i_1,\gamma)$ with a non-negative
integer $\gamma$ satisfying
$\gamma d_{i_1n_{i_1}} \equiv -1 \, \mathrm{mod} \, l_{i_1n_{i_1}}$
and
$$
\frac{1}{l_{in_i} (m_{n_i-1}-m_{n_i})}
\ = \ 
\xi_i
\ \le \ 
\gamma
\ \le \
\eta_{i_0n_{i_0}}
\ = \ 
- \frac{1}{l_{i_0n_{i_0}} m^-},
$$
for all $i = 0, \ldots, r$ with $i \ne i_0$
and $n_i > 1$.
We compare $\gamma$ with $\tilde l_{i_0 \tilde n_{i_0}}$
and $\tilde l_{\tilde n_i-1}$.
First, using the modular identities, we observe
$$
(\tilde l_{i_0 \tilde n_{i_0}} -\gamma)  d_{i_1n_{i_1}}
\ \in \ l_{i_1n_{i_1}} \ZZ,
\qquad
(\gamma - \tilde l_{i_1 \tilde n_{i_1}-1})  d_{i_1n_{i_1}}
\ \in \ l_{i_1n_{i_1}} \ZZ.
$$
As $l_{i_1n_{i_1}}$ and $d_{i_1n_{i_1}}$ are coprime,
$\tilde l_{i_0 \tilde n_{i_0}} -\gamma$
as well as $\gamma - \tilde l_{i_1 \tilde n_{i_1}-1}$
are multiples of~$l_{i_1n_{i_1}}$.
Thus, the previous estimates and $l_{in_i} = 1$ for
$i \ne i_0,i_1$ give us
$$
\tilde l_{\tilde n_i-1}
\ \le \
\gamma
\ \le \ 
\tilde l_{i_0 \tilde n_{i_0}},
\qquad 
i = 0, \ldots, r,
\ i \ne i_0,
\ n_i > 1.
$$
Now we can directly check the defining
conditions of a horizontal
$\tilde P$-root at $(\tilde x, i_0,i_1)$
for $u = u(i_0,i_1,\gamma)$:
Lemma~\ref{lem:ugamma} together with
Propositions~\ref{prop:hyp-q-smooth}
and~\ref{prop:ell-q-smooth} yields
$$
\begin{array}{ccccc}
\bangle{u,\tilde v_{i_0 \tilde n_{i_0}}}
& =
& \frac{\tilde l_{i_0 \tilde n_{i_0}} - \gamma}{\tilde l_{i_1 \tilde n_{i_1}}}
& \ge
& 0,
\\[1em]
\bangle{u,\tilde v_{i_1 \tilde n_{i_1}-1}}
& =
& \frac{\gamma - \tilde l_{i_1 \tilde n_{i_1}-1}}{ \tilde l_{i_1 \tilde n_{i_1}-1}}
& \ge
& 0,
\\[1em]
\bangle{u,v_{i \tilde n_i-1}} 
& =
& \gamma
& \ge 
& l_{i \tilde n_i-1},
\end{array}  
$$
where $i \ne i_0,i_1$ with $n_i > 1$ in the last case.
This, verifies ``(i)$\Rightarrow$(ii)''.
The reverse implication is a direct consequence of 
Proposition~\ref{prop:P-root-to-P-dem}.
\end{proof}

\begin{remark}
\label{rem:gamma-max-tilde-l}
From the proof of
Proposition~\ref{prop:hor-P-root-lift}
we infer that
$\gamma = \tilde{l}_{i_0 \tilde{n}_{i_0}}$
is the maximal integer such that
$u(i_0,i_1,\gamma)$ is a horizontal
$P$-root $x^-$.
\end{remark}

\begin{proposition}
\label{prop:vert-P-root-lift}
Consider the minimal resolution
$\pi \colon \tilde X \to X$
of $\KK^*$-surfaces defined by
$(A,P)$ and $(A,\tilde P)$.
Assume that there is a parabolic 
fixed point curve $D^+ \subseteq X$
and let $\tilde D^+ \subseteq \tilde X$ 
be the proper transform.
Given $u \in \ZZ^{r+1}$,
the following statements are equivalent.
\begin{enumerate}
\item
The linear form $u \in \ZZ^{r+1}$ is a
vertical $P$-root at $D^+$.  
\item
The linear form $u \in \ZZ^{r+1}$ is a
vertical $\tilde P$-root at $\tilde D^+$.  
\end{enumerate}
\end{proposition}

\begin{proof}
The implication ``(ii)$\Rightarrow$(i)'' is
clear due to Proposition~\ref{prop:vrootchar}.
We care about ``(i)$\Rightarrow$(ii)''.
By Remark~\ref{rem:surfsingres}, the columns
$\tilde v_{i1}, \ldots, \tilde v_{q_i} = v_{i1}$
of $\tilde P$, where $i = 0, \ldots, r$,
arise from subdividing
$\cone(v_{i1},v^+)$ along the Hilbert basis.
Consider
$$
v_0 \ := \ (0,-1),
\qquad
v_1 \ := \  (l_1,-d_{i1}),
$$
where $i = 0, \ldots, r$.
In the setting of Remark~\ref{rem:hilbertbasis},
we have $v_1' = (\tilde l_{i1}, - \tilde d_{i1})$.
Moreover, the conditions on the determinants
lead to
$$
1 \ = \ \det(v_0,v_1') \ = \ \tilde l_{i1},
\qquad
l_{i1}\tilde d_{i1} - d_{i1}
\ = \ 
\det(v_1',v_1)
\ < \
\det(v_0,v_1)
\ = \
l_{i1}.
$$
Using also slope orderedness of $\tilde P$, we
see $m_{i1} \le \tilde d_{i1} < m_{i1} +1$.
Now, let $u \in \ZZ^{r+1}$ be a vertical $P$-root
at $D^+$.
Proposition~\ref{prop:vrootchar}
ensures $u_i \ge m_{i1}$.
This implies $u_i \ge \tilde d_{i1}$.
Using Proposition~\ref{prop:vrootchar}
again, we obtain that $u$ is a vertical
$\tilde P$-root at $\tilde D^+$.
\end{proof}

\begin{remark}
\label{rem:Pada2tPada}
From the proof of
Proposition~\ref{prop:vert-P-root-lift}
we infer that $m_{i1} \le \tilde d_{i1} < m_{i1} +1$
holds for $i = 0, \ldots, r$.
In particular, if $P$ is adpated to the source,
then also $\tilde P$ is adapted to the source.
\end{remark}

\begin{proposition}
\label{prop:contract2aut}
Consider $\tilde X = X(A, \tilde P)$ and $X = X(A,P)$,
where each column of $P$ also occurs as a column of~$\tilde P$.
\begin{enumerate}
\item
There is a proper birational morphism
$\pi \colon \tilde X \to X$ contracting precisely
the curves $\tilde D_{ij}$ and $\tilde D^\pm$, where
$\tilde v_{ij}$ and $\tilde v^\pm$ is not a column of $P$.
\item
If there is an elliptic fixed point $\tilde x^- \in \tilde X$
and $u$ is a horizontal $P$-root at $(\tilde x^-,i_0,i_1)$.
Then $x^- = \pi(\tilde x^-) \in \tilde X$ is an elliptic
fixed point forming the sink and $u$ is a horizontal
$P$-root at $(x^-,i_0,i_1)$.
\item
If we have a parabolic source $\tilde D^+ \subseteq \tilde X$
and $u$ is a vertical $\tilde P$-root at $\tilde D^+$, then
$D^+ = \pi(\tilde D^+) \subseteq X$ is a curve forming the source 
and $u$ is a vertical $P$-root at~$D^+$.
\end{enumerate}
Moreover, the root groups
$\tilde \lambda \colon \KK \to \Aut(\tilde X)$
and
$\lambda \colon \KK \to \Aut(X)$
arising from a common root $u$
fit for every $s\in \KK$ into the
commutative diagram
$$ 
\xymatrix{
{\tilde X}
\ar[r]^{\tilde \lambda(s)}
\ar[d]_{\pi}
&
{\tilde X}
\ar[d]^{\pi}
\\
X
\ar[r]_{\lambda(s)}
&
X}
$$

\end{proposition}

\begin{proof}
For~(i) observe that each cone of the fan $\tilde \Sigma$
of the ambient toric variety $\tilde Z$ of $\tilde X$ 
is contained in a cone of the fan $\Sigma$
of the ambient toric variety $Z$ of $X$.
The corresponding toric morphism $\tilde Z \to Z$
restricts to the morphism $\pi \colon \tilde X \to X$.
Assertion~(ii) is clear by Proposition~\ref{prop:P-root-to-P-dem}
and~(iii) follows from Proposition~\ref{prop:vrootchar}.

We prove the supplement. 
Consider the Cox rings $R(A,P)$ of~$X$ and
$R(A,\tilde P)$ of~$\tilde X$.
Recall that these are given as factor
algebras
$$
R(A,P)
\ = \
\KK[T_{ij}, S^\pm]
/
\bangle{g_I; \ I \in \mathfrak{I}},
\qquad
R(A,\tilde P)
\ = \
\KK[\tilde T_{ij}, \tilde S^\pm]
/
\bangle{\tilde g_I; \ I \in \tilde {\mathfrak{I}} },
$$
where $R(A,P)$ is graded by $K = \Cl(X)$
and $R(A,\tilde P)$  by $\tilde K = \Cl(\tilde X)$
see Construction~\ref{constr:RAP} for the details.
Define a homomorphism of the graded polynomial
algebras 
$$
\Psi \colon 
\KK[\tilde T_{ij}, \tilde S^\pm]
\ \to \
\KK[T_{ij}, S^\pm]
$$
by sending $\tilde T_{ij}$ and $\tilde S^\pm$ to the
variables of $\KK[T_{ij}, S^\pm]$ corresponding
to $\pi(\tilde D_{ij})$ and $\pi(\tilde D^\pm)$
in case these divisors are not exceptional
and to~$1  \in \KK[T_{ij}, S^\pm]$ otherwise.
Then $\Psi$ descends to a homomorphism
$$
\psi \colon  R(A, \tilde P) \ \to \ R(A,P). 
$$
Now note that any Demazure root $\delta_u$ on the
fan $\tilde \Sigma$ having $\tilde P$ as generator
matrix is as well a Demazure root on the fan
$\Sigma$ having $P$ as generator matrix.
Moreover, we have a commutative diagram
$$ 
\xymatrix@R=1.5em{
{\KK[\tilde T_{ij}, \tilde S^\pm]}
\ar[r]^{\delta_u}
\ar[d]_{\Psi}
&
{\KK[\tilde T_{ij}, \tilde S^\pm]}
\ar[d]^{\Psi}
\\
\KK[T_{ij}, S^\pm]
\ar[r]_{\delta_u}
&
\KK[T_{ij}, S^\pm]}
$$
Next, given a horizontal or vertical $P$-root
and its corresponding $\tilde P$-root,
we look at the associated Demazure
$P$-root~$\kappa$ and Demazure
$\tilde P$-root~$\tilde \kappa$.
Presenting 
$\bar \lambda_\kappa(s)^*$
and~$\bar \lambda_{\tilde \kappa}(s)^*$
as in Theorem~\ref{thm:restrrootauts}
and using commutativity of the previous
diagram we see that the following
diagram commutes as well:
$$ 
\xymatrix@R=1.5em{
{R(A,\tilde{P})}
\ar[r]^{\bar \lambda_{\tilde \kappa}(s)}
\ar[d]_{\psi}
&
{R(A,\tilde{P})}
\ar[d]^{\psi}
\\
R(A,P)
\ar[r]_{\bar \lambda_\kappa(s)}
&
R(A,P)}
$$
Cover $\pi^{-1}(X_\reg)$ by affine open
subsets of the form
$\tilde X_{[\tilde D],\tilde f}$,
where $\tilde D$ is a Weil divisor
of $\tilde X$ having
$\tilde f \in R(A,\tilde P)$
as a section and
$\tilde X_{[\tilde D],\tilde f}$
is obtained by removing the support
of $\tilde D + \div(\tilde f)$ from
$\tilde X$.
Set $D = \pi_* \tilde D$ and
$f = \psi(\tilde f)$.
Then we have commutative diagrams
$$ 
\xymatrix@R=1.5em{
&
R(A,\tilde{P})_{\tilde f}
\ar[r]^{\psi}
\ar@{<-}[d]
&
R(A,P)_f
\ar@{<-}[d]
&
\\
\Gamma(\tilde X_{[\tilde D],\tilde f},\mathcal{O}_{\tilde X})
\ar@{=}[r]
&
R(A, \tilde P)_{(\tilde f)}
\ar@{<->}[r]^{\psi_0}_{\pi^*}
&
R(A,P)_{(f)}
\ar@{=}[r]
&
\Gamma(X_{[D], f},\mathcal{O}_X)
}
$$
where the lower row represents the degree
zero part of the upper one.
The homomorphisms $\psi_0$ and $\pi^*$ in
the lower row are directly seen to be
inverse to each other;
see~\cite[Prop.~1.5.2.4]{ArDeHaLa}.
Passing to the spectra and gluing gives
us a commutative diagram
$$ 
\xymatrix@R=1.5em{
{p^{-1}(X_\reg)}
\ar[r]^{\varphi \quad}
\ar[d]_{p}
&
{\tilde p^{-1}(\pi^{-1}(X_\reg))}
\ar[d]^{\tilde p}
\\
X_\reg
\ar@{<->}[r]_{\pi \quad}^{\varphi_0 \quad}
&
{\pi^{-1}(X_\reg)}
}
$$
where $p$ and $\tilde p$ denote the quotients
of characteristic spaces of $X$ and $\tilde X$ 
by the respective characteristic quasitori;
use again~\cite[Prop.~1.5.2.4]{ArDeHaLa}. 
By construction, the morphisms
$\varphi$ arising from $\psi$ and
$\varphi_0$ arising from $\psi_0$ satisfy
$$
\varphi \circ \bar \lambda_{\kappa}(s)
\ = \
\bar \lambda_{\tilde \kappa}(s) \circ \varphi,
\qquad \qquad
\varphi_0 \circ \lambda_{\kappa}(s)
\ = \
\lambda_{\tilde \kappa}(s) \circ \varphi_0.
$$
\end{proof}

\begin{proof}[Proof of Theorem~\ref{thm:equiv-resolution}]
Propositions~\ref{prop:hor-P-root-lift} 
and~\ref{prop:vert-P-root-lift} provide
us with a bijection between the $P$-roots
and the $\tilde P$-roots.
Applyying Proposition~\ref{prop:contract2aut},
we obtain proves the second assertion of the Theorem.
Assertions~(iii) and~(iv) are then direct consequences.
\end{proof}

%% file: group-structure.tex
\section{Structure of the automorphism group}
\label{sec:group-structure}

Here we prove Theorem~\ref{thm:main}.
In a first step, we express the number of
necessary $P$-roots to generate the
unipotent part of the automorphism group 
of a $\KK^*$-surface $X = X(A,P)$ in terms
of intersection numbers of invariant curves
of $X$.
We make use of the numbers
defined in the introduction:
$$
c_i(D^+)
\ = \
\mathrm{CF}_{q_i} (-E_{i1}^2,\ldots,-E_{i q_i}^2)^{-1},
\qquad
c(x^-)
\ = \
\mathrm{CF}_{q}(-E_{q}^2,\ldots,-E_{1}^{-1})^{-1},
$$
where the $E_{ij} \subseteq \tilde X$ are the
exceptional curves lying over $D^+ \subseteq X$
in and the $E_i \subseteq \tilde X$ over $x^- \in X$
with respect to the minimal resolution
of singularities $\tilde X \to X$.

\begin{definition}
Let $X = X(A,P)$ be non-toric
with a fixed point curve
$D^+ \subseteq X$.
Given $0 \le i_0,i_1 \le r$,
we call a vertical $P$-root~$u$
at $D^+$ \emph{essential}
with respect to $i_0,i_1$,
if $0 \le \bangle{u, v_{i1}} < l_{i1}$
for all $i \ne i_0,i_1$.
\end{definition}

\begin{proposition}
\label{prop:number-of-roots}
Consider a non-toric $\KK^*$-surface $X = X(A,P)$.
\begin{enumerate}
\item
Assume that there is a curve
$D^+ \subseteq X$ and let $0 \le i_0,i_1 \le r$.
Then the number of vertical $P$-roots at $D^+$
essential to $i_0,i_1$ is given by 
$$
\max
\bigl(
0, \
(D^+)^2 +1 - \sum_{i=0}^r c_i(D^+)
\bigr).
$$
\item
Assume that there is a quasismooth simple $x^- \in X$
and that $P$ is normalized.
Then the number of horizontal $P$-roots 
at $(x^-,0,1)$ is given by 
$$
\qquad\qquad
\mathrm{max}
\biggl(
0, \
\left\lfloor
l_{1n_1}^{-1}
\min_{i \ne 0}
(
l_{in_i}D_{i n_i}^2 + (l_{in_i}-l_{1 n_1}) D_{i n_i} D_{1 n_1}
) 
- c(x^-)
\right\rfloor+1
\biggr).
$$
\item
Assume that there is a quasismooth simple $x^- \in X$
and that $P$ is normalized.
Then there is a horizontal $P$-root at $(x^-,1,0)$
if and only if
$$
\qquad 
l_{i n_i}D_{i n_i}^2
\ \ge \ 
(l_{0 n_0} - l_{in_i})D_{in_i}D_{0 n_0},
\quad
\text{ for all } i \ne 1.
$$
Moreover, if these conditions hold, then
$x^- \in X$ is smooth and there exists
precisely one horizontal $P$-root at $(x^-,1,0)$.
\end{enumerate}
\end{proposition}

\begin{proof}
Let $\tilde X = X(A, \tilde P)$ be the
minimal resolution.
We verify~(i).
By Proposition~\ref{prop:vert-P-root-lift},
the numbers $\rho$ of vertical $P$-roots
and $\tilde \rho$ of vertical $\tilde P$-roots
essential to $i_0,i_1$ coincide.
By Proposition~\ref{prop:ver-P-roots-gamma},
the number $\tilde \rho$ equals
the number of integers in the interval
$$
\tilde{\Gamma}(i_0,i_1)
\ = \
\Bigl[
\tilde{m}_{i_1 1},
\
-\tilde{m}_{i_0 1} 
- \sum_{i \ne i_0,i_1} 
\lceil \tilde{m}_{i1} \rceil
\Bigr].
$$
Since $\tilde X$ is smooth, the
slopes $\tilde{m}_{i 0}, \ldots, \tilde{m}_{i r}$
are all integral numbers; see
Proposition~\ref{prop:smooth-par-li1}.
Thus, we see that $\rho = \tilde \rho$
equals the maximum of zero and the number
$$
-\tilde{m}_{i_0 1}
- \tilde{m}_{i_1 1}
- \sum_{i\neq i_0,i_1} \tilde{m}_{i 1}  
+ 1
\ = \
-\tilde{m}^+ + 1
\ = \
(\tilde{D}^+)^2 + 1,
$$
where the last equality holds by
Remark~\ref{rem:selfintKstar}.
We have $\tilde D_{ij} = E_i$
for $j = 1, \ldots, q_i$.
Moreover, $\tilde D_{iq_i+1} = D_{i1}$
and $\tilde m_{iq_i+1} = m_{i1}$.
Thus, applying Corollary~\ref{cor:intno-cfrac}~(i)
with $j_i = q_i+1$,
we see that $\rho = \tilde \rho$
equals the maximum of zero and
$$
(\tilde{D}^+)^2 + 1
\ = \
(D^+)^2 - \sum_{i = 0}^r c_i(D^+) + 1.
$$
We verify~(ii).
According to
Proposition~\ref{prop:hor-P-roots-gamma},
the number~$\rho$ of horizontal $P$-roots 
at $(x^-,0,1)$ equals the number of
integers $\gamma$ satisfying
$$
\gamma
\ \in \
\Delta(0,1) 
\ = \
\bigcap_{i\neq 0}[\xi_i,\eta_1],
\qquad
\gamma d_{1 n_1} \equiv -1  \mod l_{1 n_1},
$$
see Construction~\ref{constr:intervalls}
for the notation.
By Remark~\ref{rem:gamma-max-tilde-l},
the maximal integer~$\gamma$ satisfying
these conditions is $\tilde{l}_{0 \tilde{n}_0}$.
Thus, we can replace $[\xi_i,\eta_1]$
with $[\xi_i,\tilde{l}_{0 \tilde{n}_0}]$.
So, the number of integers $\gamma$ in 
$[\xi_i,\tilde{l}_{0 \tilde{n}_0}]$
with $\gamma d_{1 n_1} \equiv -1  \mod l_{1 n_1}$
is the maximum of zero and the round down
$\vartheta(i) \in \ZZ$ of
\begin{eqnarray*}
\frac{\tilde{l}_{0 \tilde{n}_0} - \xi_i}{l_{1 n_1}} + 1 
& = &
\frac{(\eta_1 - \xi_i) - (\eta_1  - \tilde{l}_{0 \tilde{n}_0})}{l_{1 n_1}}
+ 1
\\
& = &
l_{1 n_1}^{-1}
\left(
l_{in_i}D_{i n_i}^2 + (l_{in_i}-l_{1 n_1}) D_{i n_i} D_{1 n_1} - c(x^-)   
\right)
+ 1.
\end{eqnarray*}
Here, the second equality needs explanation.
First, we express $\eta_1 - \xi_i$ in terms
of intersections numbers according to
Remark~\ref{rem: intervalls self-intersection-no}.
Moreover, the definition of $\eta_1$,
Remark~\ref{rem:m+m-}, quasismoothness of $x^-$
and Proposition~\ref{prop:ell-q-smooth}
yield
$$
\eta_1 = \frac{1}{l_{1 n_1}m^-},
\qquad
l^-m^- = \det(\sigma^-),
\qquad
l^- = l_{0 n_0} l_{1 n_1}.
$$
Proposition~\ref{prop:roots-i0i1-switched}~(ii)
says that $0$ is the exceptional index
of $x_i \in X$ and thus $E_j = \tilde D_{0n_0-q+j}$
holds for $j = 1,\ldots, q$.
Using Corollary~\ref{cor:intno-cfrac}~(ii),
we obtain
$$
\eta_1  - \tilde{l}_{0 \tilde{n}_0}
\ = \
\frac{l_{0 n_0}}{\det(\sigma^-)} -\tilde{l}_{0 \tilde{n}_0}
\ = \ 
c(x^-).
$$
Since $\Delta(0,1)$ is the intersection
over the intervals $[\xi,\eta_1]$, where $i \ne i_0$,
we see that the number of all the wanted $\gamma$
we have to take the minimum of the above
round downs $\vartheta(i)$ as an upper bound.

We care about~(iii).
By Proposition~\ref{prop:hor-P-roots-gamma},
there exists a horizontal $P$-root at
$(x^-,1,0)$ if and only if $\Delta(1,0)$
is non-empty.
The latter precisely means 
$\eta_0 - \xi_i \ge 0$ for all $i \ne 1$.
This in turn is equivalent to 
$$
l_{i n_i}D_{i n_i}^2
\ \ge \ 
(l_{0 n_0} - l_{i n_i})D_{in_i}D_{0 n_0},
\quad
\text{ for all } i \ne 1,
$$
see Remark~\ref{rem: intervalls self-intersection-no}.
Now, if there is a horizontal $P$-root at $(x^-,1,0)$,
then Proposition~\ref{prop:roots-i0i1-switched}~(i)
yields that there is no further one and that
$x^- \in X$ is smooth.
\end{proof}

The final puzzle piece for the proof of
Theorem~\ref{thm:main} is the following
controled contraction $\tilde X \to X'$
of the minimal resolution $\tilde X$ of
$X$ onto a toric surface that allows to
keep track of the relevant roots.

\goodbreak

\begin{proposition}
\label{prop:contract2toric}
Consider a minimal resolution $\tilde X \to X$,
where $X = X(A,P)$ is non-toric and
$\tilde X = X(A,\tilde P)$.
% Assume that $P$ is adapted to the sink
% and that we have a fixed point curve
% $\tilde D^+ \subseteq \tilde X$.
\begin{enumerate}
\item
Let $x^- \in X$ be a quasismooth simple elliptic
fixed point and let $P$ be adapted to the sink.
Then there is a $\KK^*$-equivariant
morphism $\pi \colon \tilde X \to X'$
onto the toric smooth projective
$\KK^*$-surface~$X'$ defined 
the matrix
$$
\qquad
P'
\ = \ 
\left[
\begin{array}{rrrrrrr}
-\tilde l_{0 1} & \ldots & -\tilde l_{0 \tilde n_0} & \tilde l_{1 1} & \ldots & \tilde l_{1 \tilde n_1} & 0
\\
\tilde d_{0 1} & \ldots & \tilde d_{0 \tilde n_0}  & \tilde d_{1 1} & \ldots &\tilde  d_{1 \tilde n_1} & 1
\end{array}   
\right].
$$
The horizontal $\tilde P$-roots at $(\tilde x^-,0,1)$
map injectively to the horizontal $P'$-roots
at $(\pi(\tilde x^-),0,1)$ via
$$
\qquad
\ZZ^{r+1} \ni u(0,1,\gamma) \ \mapsto \ u(0,1,\gamma) \in \ZZ^2.
$$
Similarly, the horizontal $\tilde P$-roots at $(\tilde x^-,1,0)$
map injectively to the horizontal $P'$-roots
at $(\pi(\tilde x^-),1,0)$ via
$$
\qquad
\ZZ^{r+1} \ni u(1,0,\gamma) \ \mapsto \ u(1,0,\gamma) \in \ZZ^2.
$$
\item
Assume that we have a curve $D^+ \subseteq X$
admitting vertical $P$-roots and let~$P$
be adapted to the source.
Then we obtain a $\KK^*$-equivariant morphism
$\pi \colon \tilde X \to X'$ onto the smooth
toric $\KK^*$-surface $X'$ defined by the matrix
$$
\qquad
P'
\ = \ 
\left[
\begin{array}{rrrrrrrr}
- \tilde l_{0 1} & \ldots & - \tilde l_{0 \tilde n_0} & \tilde l_{1 1} & \ldots & \tilde l_{1 \tilde n_1} & 0 & 0
\\
\tilde d_{0 1} & \ldots & \tilde d_{0n_0}  & \tilde d_{1 1} & \ldots & \tilde d_{1\tilde n_1} & 1 & -1
\end{array}   
\right].
$$
The image $\pi(\tilde D^+) \subseteq X'$
is a curve forming the source
and the vertical $\tilde P$-roots 
at $\tilde D^+$ essential with respect to $0,1$
map injectively to the vertical $P'$-roots at
$\pi(\tilde D^+)$ essential with respect to
$0,1$ given by
$$
\ZZ^{r+1} \ni u(0,1,\alpha) \ \mapsto \ u(0,1,\alpha) \in \ZZ^2.
$$
\end{enumerate}
Moreover, the root groups
$\tilde \lambda \colon \KK \to \Aut(\tilde X)$
and
$\lambda' \colon \KK \to \Aut(X')$
arising from a common root $u$
fit for every $s \in \KK$ into the
commutative diagram
$$ 
\xymatrix{
{\tilde X}
\ar[r]^{\tilde \lambda(s)}
\ar[d]_{\pi}
&
{\tilde X}
\ar[d]^{\pi}
\\
X'
\ar[r]_{\lambda'(s)}
&
X'
}
$$
Finally, each Demazure $P'$-root is also a
Demazure root on the complete fan $\Sigma'$ with
generator matrix $P'$ and the respective root groups
in the sense of Constructions~\ref{constr:DEMLND}
and~\ref{constr:toriclnd}
coincide.
\end{proposition}

\begin{proof} 
First we convince ourselves
that in setting~(i)
we have
$\tilde l_{i \tilde n_i} = 1$
as well as 
$\tilde d_{i \tilde n_i} = 0$
for all $i \ge 2$
and, moreover, that there is a curve
$\tilde D^+ \subseteq \tilde X$.
Note that
$x^- \in X$ has exceptional index
$0$ or $1$ by
Propositions~\ref{prop:qsm-str-ell-char}
and~\ref{prop:contract2aut}.
Thus, for any $i \ge 2$,
Proposition~\ref{prop:ell-q-smooth}
yields $\tilde l_{i \tilde n_i} = l_{in_i} = 1$
and the fact that $P$ is adapted to
the sink
ensures $\tilde d_{i \tilde n_i} = d_{in_i} = 0$.
The existence of $\tilde D^+ \subseteq \tilde X$
is guaranteed by Corollary~\ref{cor:qsm2}.
In the situation of~(ii),
the existence of $\tilde D^+ \subseteq \tilde X$
is clear.
We ensure $\tilde l_{i 1} = 1$
and $\tilde d_{i 1} = 0$ for all
$i \ge 1$.
Moreover, by Proposition~\ref{prop:qsefp2novroots},
there must be a
curve $\tilde D^- \subseteq \tilde X$.
Now, as $P$ is adapted to the source,
Remark~\ref{rem:Pada2tPada}
yields that also $\tilde P$ is adapted to the
source.
Thus, we have $\tilde l_{i1} = 1$ and $\tilde d_{i1} = 0$
for~$i = 1, \ldots, r$.
From now, we treat Settings~(i) and~(ii)
together. Consider the data
$$
n''_0 = n_0, \ v_{0j}'' =  v_{0j},
\quad 
n''_1 = n_1, \ v_{1j}'' =  v_{1j},
\quad
n''_i = 1,  \ v_{i1}'' =  v_{i n_i}, \ i \ge 2.
$$
These, together with $v^+$ in the setting of~(i)
and $v^+,v^-$ in the setting of~(ii) are the
columns of a matrix $P''$.
It defines a $\KK^*$-surface $X'' = X(A,P'')$
which is smooth due to
Propositions~\ref{prop:smooth-par-li1}
and~\ref{prop:ell-q-smooth}.
Proposition~\ref{prop:contract2aut}
gives us a morphism $\tilde X \to X''$
having the desired properties concerning
the roots and the associated root groups.

Now, the matrix $P''$ is highly redundant.
Removing all these redundancies,
that means erasing the column $v_{in_i}$ and
the $i$-th rows for $i = 2, \ldots, r$
turns $P''$ into $P'$.
The $\KK^*$-surface $X''$ is isomorphic
to the toric $\KK^*$-surface $X' = X(A',P')$,
where $A'$ is the $2 \times 2$ unit matrix.
Using 
$\tilde l_{i \tilde n_i} = 1$
as well as 
$\tilde d_{i \tilde n_i} = 0$
in~(i)
as well as 
$\tilde l_{i1} = 1$ and $\tilde d_{i1} = 0$
in~(ii) as seen before,
one checks that
the $P''$-roots turn into $P'$-roots
as claimed.
The supplement is directly verified.
\end{proof}

\begin{proof}
[Proof of Theorem~\ref{thm:main}]
Let $X:=X(A,P)$ be a non-toric $\KK^*$-surface.
Theorem~\ref{thm:cpl1aut} on the automorphism
group of a rational projective variety with
torus action of complexity one says
that $\Aut(X)^0$ is generated by
the acting torus and the additive one-parameter
groups associated with the Demazure $P$-roots.
In the surface case, the latter ones
are given by horizontal and vertical
$P$-roots; see 
Propositions~\ref{prop:hor-P-roots-gamma} 
and~\ref{prop:vrootchar}.
Thus, $\Aut(X)^0 = \KK^*$ if
and only if there are neither horizontal
nor vertical $P$-roots.
Horizontal $P$-roots only exist
if $X$ admits a quasismooth simple
elliptic fixed point, and in this case
there is no other such fixed point;
see Proposition~\ref{prop:hor-P-root-constraints}
and Theorem~\ref{thm:qs-str-ell-fp}.
This setting is Case~(ii) of Theorem~\ref{thm:main}.
Moreover, existence of vertical $P$-roots
requires a non-negative parabolic fixed
point curve and excludes quasismooth simple
elliptic fixed points; see
Propositions~\ref{prop:vrootchar}
and~\ref{prop:qsefp2novroots}.
This setting restitutes Case~(i) of
Theorem~\ref{thm:main}.
Recall that $U(X) \subseteq \Aut(X)^0$
denotes the subgroup generated by all root
subgroups.

We determine $\Aut(X)^0$ in Case~(i)
of Theorem~\ref{thm:main}.
Thus, we have to deal with a
non-negative parabolic
fixed point curve hosting vertical
$P$-roots, if present, and which we
may assume to be $D^+ \subseteq X$.
Moreover, we may assume that the defining
matrix $\tilde P$ of the minimal resolution
$\tilde X$ of $X$ is adapted to the source.
Proposition~\ref{prop:vroot+2novroot-}
yields that $D^+$ is the only fixed
point curve admitting vertical roots.
Fix any two distinct $0 \le i_0,i_1 \le r$.
Then Proposition~\ref{prop:gen-ver-P-roots}
says that $U(X)$ is 
generated by all root groups
arising from vertical $P$-roots
being essential at $i_0,i_1$.
Proposition~\ref{prop:number-of-roots}~(i)
shows that~$\rho$ from
Theorem~\ref{thm:main}~(i) equals the number 
the number of vertical $P$-roots essential
to $i_0,i_1$.
Theorem~\ref{thm:equiv-resolution}
and Proposition~\ref{prop:contract2toric}
realize $U(X)$ as the subgroup generated by
Demazure roots at a common primitive
ray generator of the automorphism group
of a suitable toric surface $X'$.
Moreover the original $\KK^*$-action of $X$
is given on $X'$ by the one parameter
group $\KK^* \to \TT^2$ sending $t$
to $(1,t)$.
Applying Proposition~\ref{prop:comm-der}
yields the desired isomorphism
$\KK^\rho \rtimes_\psi \KK^* \cong \Aut(X)^0$.%

We enter Case~(ii) of Theorem~\ref{thm:main}.
The pattern of arguments is similar to that
of the preceding case.
Now we have a unique quasismooth simple
elliptic fixed point which we can
assume to be $x^- \in X$.
Moreover, we can assume $P$ to be normalized.
By Proposition~\ref{prop:gen-hor-P-roots},
the group $U(X)$ is generated
by the root groups stemming from the 
horizontal $P$-roots at 
$(i_0,i_1,x^-)$ and $(i_1,i_0,x^-)$.
Due to Proposition~\ref{prop:roots-i0i1-switched},
we may assume $i_0=0$ and $i_1 = 1$,
and, moreover, that $i_0=0$ is the exceptional
index of $x^- \in X$.
Proposition~\ref{prop:number-of-roots}~(ii)
shows that~$\rho$ and $\zeta$ from
Theorem~\ref{thm:main}~(ii) equal the
numbers of horizontal $P$-roots
at $(0,1,x^-)$ and $(1,0,x^-)$,
respectively.
Theorem~\ref{thm:equiv-resolution}
and Proposition~\ref{prop:contract2toric}
realize $U(X)$ as the subgroup generated by
Demazure roots at two common primitive
ray generators of the automorphism group
of a suitable toric surface $X'$.
Here, using
Proposition~\ref{prop:roots-i0i1-switched}~(i)
and Corollary~\ref{cor:all-hor-roots},
we see that the Demazure roots of $X'$
corresponding to horizontal $P$-roots
at $(0,1,x^-)$ and $(1,0,x^-)$
are as in the setting
of Proposition~\ref{prop:rel-tor-roots}.
Moreover, $\KK^*$ acts $X'$
via the one parameter group
$\KK^* \to \TT^2$ sending $t$ to $(1,t)$.
Thus, Proposition~\ref{prop:rel-tor-roots},
yields the desired isomorphism
$(\KK^\rho \rtimes_\varphi \KK) \rtimes_\psi \KK^* \cong \Aut(X)^0$.
\end{proof}

%% file: almhom.tex
\section{Almost homogeneous $\KK^*$-surfaces}
\label{sec:almhom}

Here, we investigate almost transitive
algebraic group actions on rational projective
$\KK^*$-surfaces $X = X(A,P)$.
Our considerations fit together to the
proof of Theorem~\ref{thm:maincor},
given at the end of the section.
Moreover, in Propositions~\ref{prop:all-k-k*-groups}
and~\ref{prop:alladdact} we specify the
two-dimensional subgroups of $\Aut(X)^0$
that act almost transitively on $X$.
The first observation of the section
says in particular that almost
transitive actions can only exist in the
presence of horizontal $P$-roots.

\begin{proposition}
\label{prop:vert2notalmhom}
Consider a non-toric $\KK^*$-surface $X = X(A,P)$.
\begin{enumerate}
\item
If $\lambda \colon \KK \to \Aut(X)$
is a root group defined by a vertical $P$-root,
then each orbit of $\lambda(\KK)$
is contained in the closure of a $\KK^*$-orbit.
\item
If $X$ admits vertical $P$-roots,
then $\Aut(X)$ acts with orbits
of dimension at most one.
\end{enumerate}
\end{proposition}

\begin{proof}
The first statement is a consequence
of~\cite[Cor.~5.11~(ii)]{ArHaHeLi}.
Alternatively, it directly follows from
Remark~\ref{rem:DemPandDem},
Construction~\ref{constr:toriclnd} and
the definition of the $\KK^*$-action on
$X = X(A,P)$.
For the second statement, recall from
Proposition~\ref{prop:qsefp2novroots}
that the presence of a vertical $P$-roots
excludes quasismooth simple elliptic
fixed points and hence also excludes
horizontal $P$-roots. 
\end{proof}

From now on, we assume presence of horizontal
$P$-roots and we work in the setting of
Theorem~\ref{thm:main}~(ii).

\begin{construction}
\label{constr:almhom1}
Let $X = X(A,P)$ be non-toric with $x^- \in X$ and
$P$ normalized. Consider the unit component of its 
automorphism group
$$
\Aut(X)^0
\ = \
\left(
\KK^\rho \rtimes_\varphi \KK^\zeta
\right)
\rtimes_\psi \KK^* ,
$$
where the numbers $\rho$ and $\zeta$ as well as
the  twisting homomorphisms $\varphi$ and $\psi$ 
are specified in Theorem~\ref{thm:main}.
Moreover, define lines in $\KK^{\rho + \zeta}$
by 
$$
U_k
\ := \
\KK e_k,
\qquad
k = 1, \ldots, \rho + \zeta.
$$
Then each of the following semidirect
products $G_k$ is a two-dimensional subgroup
of $\Aut(X)^0$ containing $\KK^*$:
$$
G_k \ = \ U_k \rtimes_{\psi_k} \KK^*,
\qquad 
\psi_k(t)(s)
\ = \
\begin{cases}
t^{\tilde l_{0 \tilde n_0} -(k-1)l_{1n_1}}s,
&  
k = 1, \ldots, \rho,
\\
t^{l_{1 n_1}}s,
& \zeta = 1, \ k = \rho +1.
\end{cases}
$$
\end{construction}

\begin{remark}
\label{rem:lines-2-hor-roots}
Let $X = X(A,P)$ be as in Construction~\ref{constr:almhom1}.
Then each of the lines $U_1, \ldots, U_{\rho + \zeta}$ is 
a root group. More precisely, the following holds.
\begin{enumerate}
\item
For $k = 1, \ldots, \rho$, the lines $U_k \subseteq \Aut(X)^0$ 
are precisely the root groups defined by the horizontal 
$P$-roots $u(0,1,\gamma_k)$ with 
$\gamma_k = \tilde l_{0 \tilde n_0} -(k-1)l_{1n_1}$.
\item
For $\zeta = 1$ and $k = \rho + \zeta$, 
the line $U_k \subseteq \Aut(X)^0$ is precisely the root 
group defined by the horizontal $P$-root $u(1,0,\gamma_k)$ 
with $\gamma_k = l_{1n_1}$.
\end{enumerate}
\end{remark}

By a \emph{$\KK^*$-general point} of a rational projective
$\KK^*$-surface $X = X(A,P)$, we mean an $x \in X$
which is not a fixed point and not contained in any
arm of $X$. Observe that a point $x \in X$ is $\KK^*$-general
if and only if each of its Cox coordinates is non-zero.

\begin{lemma}
\label{lem:generalKorbit}
Let $X = X(A,P)$ admit a horizontal $P$-root $u$ at $(x^-,i_0,i_1)$.
Then, for every $\KK^*$-general $x \in X$,
the root group $\lambda \colon \KK \to \Aut(X)$ given by $u$ 
satisfies
$$
D_{in_i}
\cap 
\lambda(\KK) \cdot x
\ \ne \
\emptyset,
\quad
i \ne i_0,
\qquad
D_{ij}
\cap
\lambda(\KK) \cdot x
\ = \
\emptyset,
\quad
i = i_0 \text { or } j \ne n_i.
$$
Moreover, if there is a fixed point curve
$D^+ \subseteq X$, then we have
$\lambda(\KK) \cdot x \cap D^+ = \emptyset$.
Finally, we have
$\lambda(\KK) \cdot D_{i_0n_{i_0}} = D_{i_0n_{i_0}}$.
\end{lemma}

\begin{proof}
Theorem~\ref{thm:restrrootauts} yields
$\bar \lambda(s)^*(S^+) = S^+$
in the case that there is a fixed point
curve $D^+ \subseteq X$.
Moreover it shows
\smallskip
$$
\begin{array}{rcll}
i = i_0:
&  
\bar \lambda(s)^*(T_{i_0 n_{i_0}})
& = &
T_{i_0 n_{i_0}},
\\[1em]
i = i_1:
&  
\bar \lambda(s)^*(T_{i_1 n_{i_1}})
& = &
T_{i_1 n_{i_1}} 
+ 
s \delta_u(T_{i_1 n_{i_1}}),
\\[1em]
i  \ne i_0,i_1:
&  
\bar \lambda (s)^* (T_{i n_i}) 
& = &
T_{i n_i}
+
\sum_{\nu = 1}^{l_{i_1 n_{i_1}}}
\alpha(s,\nu,i) 
\delta_{u_{\nu,i}}(T_{i n_i}),
\\[1em]
j \ne n_i:
&
\bar \lambda (s)^* (T_{ij})
& = &
T_{ij}.
\end{array}          
$$

\smallskip\noindent
Thus,
$\lambda(\KK) \cdot x \cap D^+ = \emptyset$
and
$D_{ij} \cap \lambda(\KK) \cdot x  = \emptyset$
provided $i = i_0$ or $j \ne n_i$.
For $i \ne i_0$, a suitable
choice of $s$ yields $\bar \lambda (s)^* (T_{in_i}) = 0$
and hence $\lambda(s) \cdot x \in D_{in_i}$.
\end{proof}

\begin{proposition}
\label{prop:two-dim-subgroups-1}
Let $X = X(A,P)$ be as in Construction~\ref{constr:almhom1}.
Then we have the following statements on the actions 
of the subgroups $G_1, \ldots, G_{\rho + \zeta} \subseteq \Aut(X)^0$.
\begin{enumerate}
\item
For each $k = 1, \ldots, \rho$, the group $G_k$ acts almost 
transitively on $X$ and at any point of its open orbit, $G_k$ has 
cyclic isotropy group of order $l_{1n_1}$. 
\item
For $k = 1, \ldots, \rho$, the $G_k$-action turns $X$ 
into an equivariant $G_k$-compactification if and only if 
$l_{1n_1} = 1$ holds.
\item
For $\zeta = 1$ and $k = \rho + \zeta$, 
the group $G_k$ acts almost transitively on $X$ and 
at any point of its open orbit, $G_k$ has 
cyclic isotropy group of order $l_{0n_0}$.
\item
For $\zeta = 1$ and $k = \rho + \zeta$, the $G_k$-action 
turns $X$ into an equivariant $G_k$-compactification if 
and only if $l_{0n_0} = 1$ holds.
\end{enumerate}
\end{proposition}

\begin{proof}
We prove~(i).
Lemma~\ref{lem:generalKorbit} shows that
each of the groups $G_k= U_k \rtimes_{\psi_k} \KK^*$
acts almost transitively;
see also~\cite[Cor.~5.11]{ArHaHeLi}.
Now, fix $k$ and set $G := G_k$.
For the $\KK^*$-general point $x \in X$,
the isotropy group $G_x$ projects via
$G \to \KK^*$ isomorphically onto a finite 
cyclic group.
Take a generator $g \in G_x$.
As~$g$ is semisimple, we have $s g s^{-1} \in \KK^*$ 
for suitable $s \in U_k$.
Thus $s G_x s^{-1} = G_{s \cdot x} = \KK^*_{s \cdot x}$.
Remark~\ref{rem:lines-2-hor-roots}
and Lemma~\ref{lem:generalKorbit} yield
the assertion.
Assertion~(i) is a direct consequence of~(ii).
Assertions~(iii) and~(iv) are proven in the 
same way. 
\end{proof}

Together with Remark~\ref{rem:lines-2-hor-roots},
the above Proposition gives us in particular
the following.

\begin{corollary}
Let $X = X(A,P)$ be non-toric.
Then $\Aut(X)$ acts almost transitively on $X$
if and only if $X$ admits horizontal $P$-roots.  
\end{corollary}

Besides the obvious subgroups $G_k \subseteq \Aut(X)^0$
acting almost transitively on $X$,
we sometimes also encounter the following more hidden
family of two-dimensional subgroups of $\Aut(X)^0$.

\begin{construction}
\label{constr:timos-group}
Let $X = X(A,P)$ be non-toric with $x^- \in X$
and $P$ normalized. 
Assume $\rho \ge 1$, $\zeta = 1$ and consider 
$$
\Aut(X)^0
\ = \
\left(
\KK^\rho \rtimes_\varphi \KK^\zeta
\right)
\rtimes_\psi \KK^* .
$$
Then every choice of a non-zero element 
$w_\rho \in \KK$
gives rise to a one-dimensional subgroup of
$\KK^\rho \rtimes_\varphi \KK^\zeta$ via
\begin{eqnarray*}
U(w_\rho)
& := &
\{(s^{\rho} w_1, s^{\rho-1} w_2, \ldots, sw_\rho,s); \ s \in \KK\},
\\[.5em]
w_k
& := &
\frac{1}{(\rho - k + 1)} {\rho -1 \choose k-1} w_\rho.
\end{eqnarray*}
If $l_{0n_0} = \rho$ and $l_{1n_1} = 1$ hold,
then $U(w_\rho) \subseteq \Aut(X)^0$ is normalized 
by $\KK^*$ and thus gives rise to a 
two-dimensional subgroup
$$
G(w_\rho)
\ := \
U(w_\rho) \rtimes_\psi \KK^*
\ \subseteq \
\Aut(X)^0 .
$$
\end{construction}

\begin{remark}
In the setting of Construction~\ref{constr:timos-group},
assume $\rho=l_{0 n_0}=l_{1 n_1}=1$.
Then the unit component of the automorphism 
group of $X$ is given by
$$
\Aut(X)^0 \ = \ \KK^2 \rtimes_\psi \KK^*,
\qquad 
\psi(t) \ = \ \diag(t^{-1},t^{-1}).
$$
Moreover, the subgroups $U_1$, $U_2$ and $U(w_1)$,
where $w_1 \in \KK^*$ are precisely the lines through
the origin of $\KK^2$.
\end{remark}

The fact that $U(w_\rho) \subseteq \KK^\rho \rtimes_\varphi \KK^\zeta$
is indeed a subgroup as well as the condition for
being normalized by $\KK^*$ are direct consequences
of the subsequent two more general observations.

\begin{proposition}
\label{prop:timos-group}
Let $X = X(A,P)$ be non-toric with $x^- \in X$
and $P$ normalized. 
Assume $\rho \ge 1$ and $\zeta = 1$. 
Consider a subset 
of $\KK^\rho \rtimes_\varphi \KK^\zeta$ 
of the form
$$ 
U(w) 
\ := \
\{(w(s),s); \ s \in \KK\},
\qquad
w(s) 
\ := \
(s^{\gamma_1}w_1,\ldots, s^{\gamma_\rho}w_\rho),
$$
with a given
non-zero $w = (w_1, \ldots, w_\rho) \in \KK^\rho$
and given integers $\gamma_1 > \ldots > \gamma_\rho > 0$.
Then $U(w) \subseteq \KK^\rho \rtimes_\varphi \KK^\zeta$ 
is a subgroup if and only if $U(w) = U(w_\rho)$ holds.
\end{proposition}

\begin{proof}
First, recall from Theorem~\ref{thm:main}
the matrix $A(s)$
defining the twisting homomorphism
$\varphi \colon \KK^\zeta \to \Aut(\KK^\rho)$.
The subset $U(w)$ of 
$\KK^\rho \rtimes_\varphi \KK^\zeta$ 
is a subgroup if and only if 
$$
(w(r),r) \circ (w(s),s)
\ = \
(w(r) + A(r) \cdot w(s), r + s)
\ = \
(w(r+s), r+s)
$$
holds for any two $r,s\in \KK$.
Now, for $k = 1, \ldots, \rho$,
the $k$-th coordinate of the
second of the above equalities
gives us the following identities
of polynomials $p_k$ and $q_k$
in~$r,s$, altogether
characterizing the
subgroup property of $U(w)$:
\begin{eqnarray*}
p_k(r,s)
& := &
w_kr^{\gamma_k}
+
\sum_{i = k}^{\rho} {i-1 \choose {k-1}} w_i r^{i-k} s^{\gamma_i} 
\\
& = &
w_k r^{\gamma_k}
+
\sum_{i = 0}^{\rho-k} {k+i-1 \choose {k-1}}
w_{k+i}r^{i}s^{\gamma_{k+i}}
\\
& = &
 w_k (r+s)^{\gamma_k}
\\
& = &
\sum_{i=0}^{\gamma_k}
{\gamma_k \choose i} w_k  r^is^{\gamma_k-i}
\\
& =: &
q_k(r,s).
\end{eqnarray*}

We claim that if $U(w)$ is a subgroup
of $\KK^\rho \rtimes_\varphi \KK^\zeta$,
then $w_1, \ldots, w_\rho \in \KK^*$ holds.
Otherwise, let~$k$ be minimal
with $w_k = 0$.
Then $p_k = q_k$ implies
$w_i = 0$ for $i = k, \ldots, \varrho$.
Due to $w \ne 0$, we have $k > 1$.
The equation $p_{k-1} = q_{k-1}$ 
yields
$$
r^{\gamma_{k-1}} + s^{\gamma_{k-1}}
\ = \
\sum_{i=0}^{\gamma_{k-1}}
{\gamma_{k-1} \choose i}  r^is^{\gamma_{k-1}-i} .
$$
This is only possible for $\gamma_{k-1} = 1$.
As we have $\gamma_1 > \ldots > \gamma_\rho > 0$,
we conclude $k-1 = \rho$.
A contradiction to the choice of $k$.
Thus, if $U(w)$ is a subgroup
of $\KK^\rho \rtimes_\varphi \KK^\zeta$,
then we must have $w_1, \ldots, w_\rho \in \KK^*$.

This reduces our task to showing
that a given $U(w)$ with
$w_1, \ldots, w_\rho \in  \KK^*$
is a subgroup of 
$\KK^\rho \rtimes_\varphi \KK^\zeta$ ,
if and only if $U(w)$ equals $U(w_\rho)$ from
Construction~\ref{constr:timos-group}.
Comparing the number of terms of $p_k$
and $q_k$, we obtain $\gamma_k = \rho-k+1$.
Now, comparing the coefficients of
$p_k$ and $q_k$ leads to the following 
identities, characterizing
the subgroup property of $U(w)$
by
$$
(\star) \qquad
{k+i-1 \choose k-1}w_{k+i}
\ = \ 
{\gamma_k \choose i} w_k
\ = \
{\rho - (k-1) \choose i} w_k,
$$
where $k = 1, \ldots, \rho$
and $i = 0, \ldots, \gamma_k$.
In particular, taking $k$ and $i = \rho - k$,
the identities $(\star)$ bring us to the 
conditions 
$$
w_{\rho}
\ = \
w_{k+(\rho-k)}
\ = \
{\rho - k +1 \choose \rho -k}
{\rho -1 \choose k-1}^{-1}w_k,
$$
which in turn are equivalent to the defining conditions
of $U(w_\rho)$ from Construction~\ref{constr:timos-group}.
Conversely, we retrieve the characterizing identities $(\star)$
from the above conditions by an explicit computation:
\begin{eqnarray*}
{k+i-1 \choose k-1} w_{k+i} 
& = &
{k+i-1 \choose k-1}
{\rho -1 \choose k+i-1}
{\rho - (k+i) +1 \choose \rho -(k+i)}^{-1}
w_\rho
\\[.5em]
& = & 
{\rho - (k-1) \choose i} w_k.
\end{eqnarray*}
\end{proof}

\begin{proposition}
\label{prop:timos-group-norm}
Let $X = X(A,P)$ be non-toric with $x^- \in X$
and $P$ normalized.
Assume $\rho \ge 1$, $\zeta = 1$.
For any one-dimensional closed subgroup
$U \subseteq \KK^\rho \rtimes_\varphi \KK^\zeta$
neither contained in $\KK^\rho$ nor in $\KK^\zeta$
the following statements are equivalent.
\begin{enumerate}
\item
The group $U \subseteq \Aut(X)^0$ is normalized by
$\KK^*$.
\item
We have $l_{0n_0} = \rho$, $l_{1n_1} = 1$ and
$U = U(w_\rho)$.
\end{enumerate}
\end{proposition}

\begin{proof}
First assume that $U$ is normalized by $\KK^*$.
By assumption we have $\zeta = 1$ and
the projection
$U \to U_{\rho+\zeta} \cong \KK$
is surjective.
As $U$ is unipotent, $U \to U_{\rho+\zeta}$
is an isomorphism.
In particular, there is a unique element
$w = (w_1,\ldots,w_\rho,1) \in U$
with $w_k \in \KK$ and $w_k \ne 0$ at
least once.
As $U$ is normalized by $\KK^*$,
we obtain 
$$
(0,t) \circ (w,1) \circ (0,t^{-1})
\ = \
(\psi(t)(w),1)
\ = \
(t^{\gamma_1}w_1, \ldots, t^{\gamma_\rho}w_\rho, t^{l_1},1)
$$
for all $t \in \KK^*$,
where we set $\gamma_k := l_{0 n_0} -(k-1)l_{1n_1}$
and $l_1 := l_{1n_1}$ for the moment.
The right hand side gives us a parametric representation
of the variety $U \subseteq \KK^\rho \times \KK^\zeta$.
Thus, setting $c_k := w_k^{l_1}$, we obtain
defining equations for 
$U \subseteq \KK^\rho \times \KK^\zeta$ by
$$
T_k^{l_1} - c_k T_{\rho + \zeta}^{\gamma_k} \ = \ 0,
\qquad
k = 1, \ldots, \rho.
$$
Observe that $\gamma_k$ and $l_1$ are coprime
due to smoothness of $x^- \in X$; see
Proposition~\ref{prop:ell-q-smooth}.
Since at least one of the $w_k$ is non-zero
and $0 \in \KK^\rho \times \KK^\zeta$
is a smooth point of $U$, we conclude $l_1 = 1$.
Thus, setting $l_0 = l_{0n_0}$, we have 
$$
U
=
\{(w(s),s); \ s \in \KK\}
\subseteq 
\KK^\rho \rtimes_\varphi \KK^\zeta,
\quad
w(s)
 := 
(s^{l_0}w_1, s^{l_0-1}w_2, \ldots, s^{l_0-\rho+1} w_\rho).
$$
Since $U$ is a subgroup of $\KK^\rho \times \KK^\zeta$,
Proposition~\ref{prop:timos-group} yields
that we have $l_0 = \rho$ and the~$w_k$ arise 
from a $w_\rho \in \KK^*$
as in Construction~\ref{constr:timos-group}.
Conversely, if~(ii) holds, then one directly
checks that $U = U(w_\rho)$ is normalized
by $\KK^*$.
\end{proof}

\begin{proposition}
\label{prop:timos-group-acts}
Let $X = X(A,P)$ be non-toric with $x^- \in X$
and $P$ normalized. 
Assume $\rho \ge 1$, $\zeta = 1$ and 
$l_{0n_0} = \rho$, $l_{1n_1} = 1$. 
Then $G(w_\rho) \subseteq \Aut(X)^0$
acts almost transitively and the 
isotropy group of a general $x \in X$  
is cyclic of order $l_{0n_0} = \rho$. 
\end{proposition}

\begin{proof}
Set for short $G = G(w_\rho)$.
It suffices to show that 
for the general point $x \in X$,
the isotropy group $G_x$ is cyclic
of order $l_{0n_0}$.
For this note first that
any element $\vartheta \in U(w_\rho)$
decomposes as
$$
\vartheta
\ = \
\vartheta_\rho \circ \vartheta_\zeta,
\qquad
\vartheta_\rho \in \KK^\rho,
\qquad
\vartheta_\zeta \in \KK^\zeta.
$$
Using Lemma~\ref{lem:generalKorbit}, we see that there are
an $\vartheta \in U(w_\rho)$ and a $\KK^*$-general
$x \in X$ such that $\vartheta_\zeta(x) \in D_{0n_0}$.
Applying Lemma~\ref{lem:generalKorbit} again shows
$\vartheta_\rho(\vartheta_\zeta(x)) \in D_{0n_0}$.
As in the proof of Proposition~\ref{prop:two-dim-subgroups-1},
we conclude that $G_x$ is cyclic of order $l_{0n_0}$.
\end{proof}

\begin{proposition}
\label{prop:all-k-k*-groups}
Let $X = X(A,P)$ be non-toric with $x^- \in X$ and
$P$ normalized.
Let $G \subseteq \Aut(X)^0$ be a two-dimensional 
subgroup containing~$\KK^*$ and
acting almost transitively on $X$.
\begin{enumerate}
\item
If $G \subseteq \KK^\rho \rtimes_\psi \KK^*$
or $G \subseteq \KK^\zeta \rtimes_\psi \KK^*$
holds, then $G$ is equal to one of the
subgroups $G_1, \ldots, G_{\rho + \zeta}$.
\item
If neither $G \subseteq \KK^\rho \rtimes_\psi \KK^*$
nor $G \subseteq \KK^\zeta \rtimes_\psi \KK^*$,
then $l_{0n_0} = \rho$, $l_{1n_1} = 1$ 
and $G = G(w_\rho)$ with $w_\rho \in \KK^*$.
\end{enumerate}
Up to conjugation, items~(i) and~(ii)
list all closed two-dimensional
subgroups of $\Aut(X)^0$ that act almost
transitively on $X$ and have a maximal
torus of dimension one.
\end{proposition}

\begin{proof}
We show~(i).
For $G \subseteq \KK^\zeta \rtimes_\varphi \KK^*$,
the group $G$ equals $G_{\rho + \zeta}$
by dimension reasons.
Assume $G \subseteq \KK^\rho \rtimes_\varphi \KK^*$.
By assumption, $p \colon G \to \KK^*$
is surjective.
Thus, $U := \ker(p)$ is a one-dimensional
subgroup of~$\KK^\rho$ and hence a line.
Moreover, we find an element
$(w,1) \in G$, where $w \in U$.
Then, for all $t \in \KK^*$, we have
$$
(0,t) \circ (w,1) \circ (0,t^{-1})
\ = \
(\psi(t)(w),1)
\ = \
(t^{\gamma_1}w_1, \ldots, t^{\gamma_\rho}w_\rho,1),
$$
where we set 
$\gamma_k := \tilde l_{0 \tilde n_0} -(k-1)l_{1n_1}$
as before and evaluate 
the twisting homomorphism
$\psi \colon \KK^* \to \Aut(\KK^\rho)$
according to its definition.
In particular,
$$
(t^{\gamma_1}w_1, \ldots, t^{\gamma_{\rho}}w_{\rho})
\ \in \
U 
\ \subseteq \ 
\KK^\rho
$$
holds for all $t \in \KK^*$.
As the $\gamma_k$ are pairwise distinct, we 
see that $U$ is one of the $U_i$.
Now, using again the definition of the twisting
homomorphism
$\psi \colon \KK^* \to \Aut(\KK^\rho)$, 
we arrive at the assertion.

We turn to~(ii).
Surjectivity of
$p \colon G \to \KK^*$ implies
that $U := \ker(p)$ is a
one-dimensional subgroup of
$\KK^\rho \rtimes_\varphi \KK^\zeta$.
Because of $\KK^* \subseteq G$,
the subgroup
$U \subseteq \Aut(X)^0$
is normalized by $\KK^*$. 
Proposition~\ref{prop:timos-group-norm}
shows  $l_{0n_0} = \rho$ and $l_{1n_1} = 1$
as well as $U = U(w_\rho)$. 
In particular, we arrive at $G = G(w_\rho)$.
\end{proof}

\begin{proposition}
\label{prop:aut-orb-dim}
Consider $X = X(A,P)$ with $x^- \in X$
and $P$ normalized.
Then each of the subgroups
$\KK^\rho, \KK^\zeta \subseteq \Aut(X)^0$
acts with orbits of dimension at most one.
Moreover, any subgroup
$G \subseteq \KK^\rho \rtimes_\psi \KK^\zeta$
containing $U_1$ and $U_{\rho + \zeta}$
acts almost transitively.
\end{proposition}

\begin{proof}
Consider $X' \leftarrow \tilde X \to X$ as provided by
Theorem~\ref{thm:equiv-resolution} and
Proposition~\ref{prop:contract2toric}.
Then $\KK^\rho$ is generated by the root groups
stemming from horizontal $P$-roots at $(x^-,0,1)$,
see Propositions~\ref{prop:gen-hor-P-roots}
and~\ref{prop:gen-ver-P-roots}.
Thus, the corresponding subgroup of $\Aut(X')$
is generated by the root groups coming from Demazure
roots at a common ray.
Looking at the resulting root groups of $X'$ in
Cox coordinates, we directly see that~$\KK^\rho$
acts with at most one-dimensional orbits.
For $\KK^\zeta$ and $G$, we succeed by the
same idea.
\end{proof}

\begin{proposition}
\label{prop:two-dim-unipot}
Let $X = X(A,P)$ be non-toric with $x^- \in X$ and
$P$ normalized.
Assume $\rho \ge 1$ and $\zeta = 1$ and
consider
$G := U_1 \rtimes_\varphi U_{\rho+\zeta} \subseteq \Aut(X)^0$.
\begin{enumerate}
\item
The group $G$ is isomorphic to the vector group $\KK^2$.
\item
The $G$-action turns $X$ into an equivariant $G$-compactification.
\item
The subgroup $G \subseteq \Aut(X)^0$ is normalized
by $\KK^* \subseteq \Aut(X)^0$.
\end{enumerate}
\end{proposition}

\begin{proof}
Assertion~(i) can be directly 
deduced from the structure of
$\KK^\rho \rtimes_\varphi \KK^\zeta$.
We show~(ii).
According to Proposition~\ref{prop:aut-orb-dim},
the group $G$ acts with an open orbit.
In particular, its general isotropy group
is finite, hence trivial due to~(i).
Assertion~(iii) is again directly checked.
\end{proof}

\begin{proof}[Proof of Theorem~\ref{thm:maincor}]
Clearly, (i) implies (ii).
Assume that~(ii) holds.
Then we infer 
from Proposition~\ref{prop:vert2notalmhom}
that there must be horizontal $P$-roots.
Thus, we may assume that there is a
fixed point $x^- \in X$ and that
$P$ is normalized.
Then we have $\rho + \zeta > 0$
and this translates to~(iii),
see Proposition~\ref{prop:number-of-roots}.
Now, if~(iii) holds, then there
is a horizontal $P$-root and an associated
root group as in Remark~\ref{rem:lines-2-hor-roots}.
By Proposition~\ref{prop:two-dim-subgroups-1},
this yields an almost transitive action of
$G = \KK \rtimes \KK^* \subseteq \Aut(X)^0$
on $X$. So, we made our way back to~(i).

We turn to the supplement concerning
the case that a two-dimensional
subgroup $G \subseteq \Aut(X)$ acts
almost transitively on $X$.
First note that $G$ is either
solvable with one-dimensional
maximal torus or $G$ is unipotent.
Thus, we either can assume by
suitably conjugating that 
$\KK^* \subseteq G$ holds
or we must have
$G \cong \KK \rtimes \KK$
and hence $G \cong \KK^2$.
Then~(iv) is covered by
Propositions~\ref{prop:two-dim-subgroups-1},
\ref{prop:timos-group-acts}
and~\ref{prop:all-k-k*-groups}.
For~(v) observe first that both
series of inequalities being valid
means $\rho \ge 1$ and $\zeta = 1$
due to Proposition~\ref{prop:number-of-roots}.
Then the assertion is a direct consequence
of~(iv) and Proposition~\ref{prop:two-dim-unipot}.
\end{proof}

Finally, we descibe all the subgroups
$G \subseteq \Aut(X)$ that are isomorphic
to the vector group $\KK^2$ and act almost
transitively on $X$.

\begin{proposition}
\label{constr:timos-group-2}
Let $X = X(A,P)$ be non-toric with $x^- \in X$
and $P$ normalized. 
Assume $\rho > 1$ and $\zeta = 1$. 
For $0 \ne w_\rho \in \KK$, set
$$ 
V(w_\rho) 
\ := \ 
U_1 U(w_\rho) 
\ \subseteq \ 
\KK^\rho \rtimes_\varphi \KK^\zeta,
$$
where $U_1$ and $U(w_\rho)$ 
are the subgroups of $\KK^\rho \rtimes_\varphi \KK^\zeta$ 
from Constructions~\ref{constr:almhom1}
and~\ref{constr:timos-group-2}.
Then the following holds.
\begin{enumerate}
\item 
$V(w_\rho)$ is a subgroup of
$\KK^\rho \rtimes_\varphi \KK^\zeta$,
isomorphic to $\KK^2$.
\item 
$V(w_\rho)$ is normalized by $\KK^*$ 
in $\Aut(X)^0$ if and only if 
$l_{0n_0} = \rho$ and $l_{1n_1} = 1$.
\item 
$V(w_\rho)$ acts almost transitively on $X$.
\end{enumerate}
\end{proposition}

\begin{proof}
We show~(i). 
Clearly, $U_1 \cap U(w_\rho)$ contains only the 
zero element.
Moreover, we directly see that the each element 
of $U_1$ commutes with each element of $U(w_\rho)$.
Thus, $V(w_\rho)$ is a
subgroup of $\KK^\rho \rtimes_\varphi \KK^\zeta$
isomorphic to the vector group~$\KK^2$.
Assertion~(ii) holds because $U_1$ is normalized
by $\KK^*$ and $U(w_\rho)$ is normalized by $\KK^*$;
see Proposition~\ref{prop:timos-group-norm}
for the latter.
For~(iii), we use Theorem~\ref{thm:restrrootauts}
and Lemma~\ref{lem:generalKorbit}
to show that 
$D_{0n_0}$ and $D_{1n_1}$ lie in the orbit of 
$V(w_\rho)$ through $x^- \in X$.
Thus the orbit of $V_\rho$ through $x^- \in X$
is open in $X$.
\end{proof}

\begin{remark}
\label{rem:families}
Let $X = X(A,P)$ be non-toric with $x^- \in X$
and $P$ normalized. 
Assume $\rho \ge 1$ and $\zeta = 1$.
\begin{enumerate}
\item
If $l_{0n_0} = \rho$ and $l_{1n_1} = 1$ hold, then
we have a one-parameter family of one-dimensional
unipotent subgroups
$$
\{U(w_\rho); \ w_r \in \KK^*\}
\ = \
\{t U(1)t^{-1}; \ t \in \KK^*\}
\ \subseteq \ 
\KK^\rho \rtimes_\varphi \KK^\zeta,
$$
and we have a one-parameter family of subgroups
isomorphic to the vector group~$\KK^2$:
$$
\{t V(1)t^{-1}; \ t \in \KK^*\}
\ \subseteq \ 
\KK^\rho \rtimes_\varphi \KK^\zeta.
$$
\item
If $l_{0n_0} \ne \rho$ or $l_{1n_1} \ne 1$ holds, then we have
a two-parameter family of one-dimensional unipotent subgroups
$$
\{t U(w_\rho)t^{-1}; \ w_\rho \in \KK^*, \, t \in \KK^*\}
\ \subseteq \ 
\KK^\rho \rtimes_\varphi \KK^\zeta,
$$
and we have a two-parameter family of 
subgroups isomorphic to the vector group~$\KK^2$:
$$
\{t V(w_\rho)t^{-1}; \  w_\rho \in \KK^*, \, t \in \KK^*\}
\ \subseteq \ 
\KK^\rho \rtimes_\varphi \KK^\zeta.
$$
\end{enumerate}  
\end{remark}

\begin{proposition}
\label{prop:alladdact}
Let $X = X(A,P)$ be non-toric with $x^- \in X$
and $P$ normalized.
Then $X$ admits additive actions
if and only if
$\rho \ge 1$ and $\zeta = 1$.
In this case, the additive
actions on $X$ are given
by the groups
$G = U_1 \rtimes_\varphi U_{\rho + \zeta}$
and, up to conjugation by elements from
$\KK^*$,
the groups $G = V(w_\varrho)$, where $w_\rho \in \KK^*$.
\end{proposition}

\begin{proof}
The first statement is clear by
Propositions~\ref{prop:two-dim-unipot}
and~\ref{prop:aut-orb-dim}.
For the second one, we consider once more
$X' \leftarrow \tilde X \to X$
as provided by Theorem~\ref{thm:equiv-resolution}
and Proposition~\ref{prop:contract2toric}.
This realizes $\Aut(X)^0$ as a subgroup of the
automorphism group of the smooth toric
surface $X'$ in such a way that $\KK^*$
becomes a subtorus of the (two-dimensional)
acting torus $\TT'$ of~$X'$.
In particular,
$U_1 \rtimes_\varphi U_{\rho + \zeta}$
and the family~$\mathbb{V}$ defined
by $V(1)$ in the sense of
Remark~\ref{rem:families} show up
in of $\Aut(X')$, where~$\mathbb{V}$
is a locally closed subvariety isomorphic
to a torus of dimension one or two.
According to~\cite[Thm.~3]{Dz}, there
are only two additive actions on $X'$
up to conjugation by~$\TT'$.
One them is normalized by $\TT'$.
Proposition~\ref{prop:rel-tor-roots}~(iii)
tells us that this is 
$U_1 \rtimes_\varphi U_{\rho + \zeta}$.
The other additive action
$G'$ has a non-trivial orbit $\TT' * G'$
under the $\TT'$-action on $\Aut(X')$
via conjugation.
Thus, $\TT' * G'$ is isomorphic
as a variety either to a
one-dimensional torus or
to a two-dimensional
torus.
This reflects exactly the cases~(i) and~(ii)
of Remark~\ref{rem:families}.
Thus, injectivity of
the morphism 
$\Aut(X)^0 \to \Aut(X')^0$ 
gives the assertion.
\end{proof}